\documentclass[reqno]{amsart}
\usepackage{setspace} 
\usepackage{geometry} 
\usepackage[dvipsnames]{xcolor}
\usepackage{amsfonts}
\usepackage{amsmath}
\usepackage{graphicx}
\usepackage{mathrsfs}
\usepackage{amssymb}
\usepackage{amsthm}
\usepackage{tikz}
\usepackage{tikz-cd}
\usepackage{listings}
\usepackage{color}
\usepackage{enumitem}
\usepackage{imakeidx}
\usepackage[colorinlistoftodos]{todonotes}
\usepackage{upgreek}
\usepackage{hyperref} 
\hypersetup{colorlinks=true, linkcolor=., citecolor=ForestGreen, urlcolor=blue}
\usepackage[capitalise, nameinlink]{cleveref}
\crefformat{equation}{(#2#1#3)} 
\usepackage{mdframed}
\usepackage{pgfplots}
\usepackage{manfnt}
\usepackage[new]{old-arrows}
\usepackage{changepage}
\usepackage{stmaryrd}
\usepackage{appendix}
\usepackage[normalem]{ulem}
\usepackage{ytableau}
\usepackage{blkarray}
\usepackage{mathtools}
\usepackage{verbatim}
\usepackage{wasysym}

\newcommand\linkcolor{BrickRed}
\let\crefinner=\cref
\newcommand\mycref[1]{{\color{\linkcolor}\crefinner{#1}}}
\def\cref{\mycref}
\let\refinner=\ref
\newcommand\myref[1]{{\color{\linkcolor}\refinner{#1}}}
\def\ref{\myref}

\DeclareMathAlphabet{\matheuler}{U}{eus}{m}{n}

\setlist[enumerate]{align=right, label=\llap{\bfseries{(\arabic*)}}}

\newenvironment{alphabetize}{\begin{enumerate}[label=\llap{\bfseries{(\alph*)}}]}{\end{enumerate}}

\pgfplotsset{width=7cm,compat=1.9}

\makeindex[intoc]

\definecolor{codegreen}{rgb}{0,0.6,0}
\definecolor{codegray}{rgb}{0.5,0.5,0.5}
\definecolor{codepurple}{rgb}{0.58,0,0.82}
\definecolor{backcolour}{rgb}{0.95,0.95,0.92}

\lstdefinestyle{mystyle}{
	backgroundcolor=\color{backcolour},   
	commentstyle=\color{codegreen},
	keywordstyle=\color{magenta},
	numberstyle=\tiny\color{codegray},
	stringstyle=\color{codepurple},
	basicstyle=\footnotesize,
	breakatwhitespace=false,         
	breaklines=true,                 
	captionpos=b,                    
	keepspaces=true,                 
	numbers=left,                    
	numbersep=5pt,                  
	showspaces=false,                
	showstringspaces=false,
	showtabs=false,                  
	tabsize=2
}

\allowdisplaybreaks

\tikzcdset{
  cells={font=\everymath\expandafter{\the\everymath\displaystyle}},
}

\newcommand{\MAT}[7]{\begin{pmatrix}{#1}&{#2}&{#3}&{#4}\\{#5}&{#6}&{#7}&\MATHelper}
\newcommand{\MATHelper}[9]{{#1}\\{#2}&{#3}&{#4}&{#5}\\{#6}&{#7}&{#8}&{#9}\end{pmatrix}}
\newcommand{\Matrix}[7]{\begin{pmatrix}{#1}&{#2}&{#3}&{#4}&{#5}\\{#6}&{#7}&\MatrixHelper}
\newcommand{\MatrixHelper}[9]{{#1}&{#2}&{#3}\\{#4}&{#5}&{#6}&{#7}&{#8}\\{#9}&\MatrixHelperii}
\newcommand{\MatrixHelperii}[9]{{#1}&{#2}&{#3}&{#4}\\{#5}&{#6}&{#7}&{#8}&{#9}\end{pmatrix}}

\newcommand{\dual}[1]{#1^\vee}
\newcommand{\pdual}[1]{\p{#1}^\vee}

\newcommand{\op}[0]{^\text{op}}

\newcommand{\units}[1]{#1^{\times}}

\newcommand{\inv}[2][1]{#2^{-#1}}

\newcommand{\ab}[1]{#1^{\mathrm{ab}}}

\newcommand{\ctsHom}{\Hom_{\mrm{cts}}}

\DeclareMathOperator{\Hom}{Hom}

\DeclareMathOperator{\Isom}{Isom}

\DeclareMathOperator{\sHom}{\mathscr{H\mkern-7mu}\textit{om}}
\DeclareMathOperator{\sExt}{\mathscr{E\mkern-4mu}\textit{xt}}
\DeclareMathOperator{\sEnd}{\mathscr{E\mkern-3mu}\textit{nd}}
\DeclareMathOperator{\End}{End}

\DeclareMathOperator{\qz}{\mathbb Q/\mathbb Z}

\DeclareMathOperator{\coker}{coker}
\DeclareMathOperator{\Ext}{Ext}
\renewcommand{\hom}{\operatorname H}
\newcommand{\homR}{\mathrm R}

\DeclareMathOperator{\Stab}{Stab}
\DeclareMathOperator{\im}{im}

\newcommand{\zmod}[1]{\mathbb Z/#1\mathbb Z}

\newcommand{\zloc}[1]{\Z_{(#1)}}

\newcommand{\Zmod}[1]{\frac{\Z}{#1\Z}}

\DeclareMathOperator{\Aut}{Aut}

\DeclareMathOperator{\Gal}{Gal}

\DeclareMathOperator{\Br}{Br}
\DeclareMathOperator{\Char}{char}

\DeclareMathOperator{\Frac}{Frac}

\DeclareMathOperator{\Pic}{Pic}

\DeclareMathOperator{\lcm}{lcm}

\DeclareFontFamily{U}{wncy}{}
\DeclareFontShape{U}{wncy}{m}{n}{<->wncyr10}{}
\DeclareSymbolFont{mcy}{U}{wncy}{m}{n}
\DeclareMathSymbol{\Sha}{\mathord}{mcy}{"58} 

\makeatletter
\newcommand{\bigperp}{%
  \mathop{\mathpalette\bigp@rp\relax}%
  \displaylimits
}

\newcommand{\bigp@rp}[2]{%
  \vcenter{
    \m@th\hbox{\scalebox{\ifx#1\displaystyle2.1\else1.5\fi}{$#1\perp$}}
  }%
}
\makeatother

\newcommand{\openset}{\overset{\text{open}}\subset}

\DeclareMathOperator{\Cl}{Cl}

\renewcommand{\d}{\mathrm d} 

\newcommand{\Set}{\mrm{Set}}

\newcommand{\QCoh}{\mrm{QCoh}}

\newcommand{\der}[1]{#1^{\mrm{der}}}

\newcommand{\red}[1]{#1_{\mrm{red}}}

\newcommand{\tet}{\text{\'et}} 
\newcommand{\et}[1]{#1_{\text{\'et}}}
\newcommand{\liset}[1]{#1_{\opname{lis-\acute et}}} 

\newcommand{\lehomR}{\homR_{\opname{lis-\acute et}}}
\newcommand{\lisethomR}{\lehomR}
\newcommand{\fppf}[1]{#1_{\mathrm{fppf}}} 

\newcommand{\fhom}{\operatorname H_{\mathrm{fppf}}}
\newcommand{\fhomR}{\homR_{\mathrm{fppf}}}

\DeclareMathOperator{\Spec}{\mbf{Spec}}
\DeclareMathOperator{\spec}{Spec}

\DeclareMathOperator{\res}{res}

\DeclareMathOperator{\codim}{codim}

\DeclareMathOperator{\Sch}{Sch}

\newcommand{\abs}[1]{\left|#1\right|}

\let\templim\lim
\renewcommand{\lim}{\templim\limits}

\newcommand{\commsquare}[8]{
	\begin{tikzcd}[ampersand replacement=\&]
	{\displaystyle #1}\ar[r, "{#2}"]\ar[d, "{#4}" left]\&{\displaystyle #3}\ar[d, "{#5}" right]\\
	{\displaystyle #6}\ar[r, "{#7}" above]\&{\displaystyle #8}
	\end{tikzcd}
}

\newcommand{\mapover}[6]{
    \begin{tikzcd}[ampersand replacement=\&]
        {\displaystyle #1}\ar[rr, "{#2}" above]\ar[dr, "{#4}"']\&\&{#3}\ar[dl, "{#5}"]\\
        \&{\displaystyle #6}
    \end{tikzcd}
}
\newcommand{\compdiag}[6]{ 
    \begin{tikzcd}[ampersand replacement=\&]
        {\displaystyle #1}\ar[r, "{#2}" below]\ar[rr, bend left, "{#6}" above]\&{\displaystyle #3}\ar[r, "{#4}" below]\&{\displaystyle #5}
    \end{tikzcd}
}

\newcommand\homses[4]{
    \begin{tikzcd}[ampersand replacement=\&]
        0\ar[r]\&{\displaystyle #1}\ar[r, "{#2}"]\ar[r]\&{\displaystyle #3}\ar[r, "{#4}"]\&
        \homsesHELPER
}
\newcommand\homsesHELPER[9]{
        {\displaystyle #1}\ar[r]\& 0\\
        0\ar[r]\&{\displaystyle #5}\ar[r, "{#6}"]\&{\displaystyle #7}\ar[r, "{#8}"]\&{\displaystyle #9}\ar[r]\& 0
        \ar[from=lllu, to=lll, "{#2}"]\ar[from=llu, to=ll, "{#3}"]\ar[from=lu, to=l, "{#4}"]
    \end{tikzcd}
}

\newcommand\frontthreesimplex[1]{
    \begin{tikzcd}[every arrow/.append style={dash}, ampersand replacement=\&, row sep = large]
        \&{\displaystyle#1}
        \frontthreesimplexHELPER
}
\newcommand\frontthreesimplexHELPER[9]{
        \ar[ddl, "{#1}"']\ar[d, "{#2}" description]\ar[ddr, "{#3}"]\\
        \&{\displaystyle#4}\ar[dl, "{#5}"]\ar[dr, "{#6}"']\\
        {\displaystyle#7}\ar[rr, "{#8}"']\&\&{\displaystyle#9}
    \end{tikzcd}
}

\renewcommand{\phi}{\varphi}

\newcommand{\mbf}{\mathbf}

\newcommand{\msO}{\mathscr O}

\newcommand{\msI}{\mathscr I}

\newcommand{\msA}{\mathscr A}
\newcommand{\msB}{\mathscr B}
\newcommand{\msC}{\mathscr C}
\newcommand{\msD}{\mathscr D}

\newcommand{\msF}{\ms F}

\newcommand{\msL}{\ms L}
\newcommand{\msM}{\ms M}

\newcommand{\ms}{\mathscr}

\newcommand{\mc}{\mathcal} 
\newcommand{\mcA}{\mc A}

\newcommand{\mrm}{\mathrm}
\newcommand{\opname}{\operatorname}

\newcommand{\F}{\mathbb F}
\newcommand{\Q}{\mathbb Q}
\newcommand{\Z}{\mathbb Z}
\newcommand{\R}{\mathbb R}

\newcommand{\N}{\mathbb N}

\newcommand{\A}{\mathbb A}
\newcommand{\G}{\mathbb G}

\newcommand{\eps}{\varepsilon}

\renewcommand{\tau}{\uptau}
\renewcommand{\P}{\mathbb P}
\newcommand{\msE}{\ms E}

\newcommand{\me}{\matheuler}

\newcommand{\meD}{\me D}
\newcommand{\meE}{\me E}

\newcommand{\meG}{\me G}
\newcommand{\meH}{\me H}

\newcommand{\meM}{\me M}

\newcommand{\meT}{\me T}
\newcommand{\meU}{\me U}

\newcommand{\meX}{\me X}
\newcommand{\meY}{\me Y}
\newcommand{\meZ}{\me Z}

\newcommand{\p}[1]{\!\left(#1\right)} 

\newcommand{\sqbracks}[1]{\!\left[#1\right]} 
\newcommand{\sq}{\sqbracks}

\newcommand{\dparens}[1]{\!\left(\!\left(#1\right)\!\right)}
\newcommand\dps\dparens 

\newcommand{\from}{\leftarrow}

\newcommand{\xto}{\xrightarrow}
\newcommand{\xfrom}{\xleftarrow}
\newcommand{\too}{\longrightarrow}
\newcommand{\xtoo}{\xlongrightarrow}
\newcommand{\iso}{\xto\sim}
\newcommand{\isoo}{\xtoo\sim}
\newcommand{\xiso}[1]{\xto[#1]\sim}

\newcommand{\fiso}{\xfrom\sim}

\newcommand{\into}{\hookrightarrow}
\newcommand{\xinto}[1]{\overset{#1}\into}

\newcommand{\onto}{\twoheadrightarrow}

\makeatletter
\newcommand*{\da@rightarrow}{\mathchar"0\hexnumber@\symAMSa 4B }
\newcommand*{\da@leftarrow}{\mathchar"0\hexnumber@\symAMSa 4C }
\newcommand*{\xdashrightarrow}[2][]{%
  \mathrel{%
    \mathpalette{\da@xarrow{#1}{#2}{}\da@rightarrow{\,}{}}{}%
  }%
}
\newcommand{\xdashleftarrow}[2][]{%
  \mathrel{%
    \mathpalette{\da@xarrow{#1}{#2}\da@leftarrow{}{}{\,}}{}%
  }%
}
\newcommand*{\da@xarrow}[7]{%
  \sbox0{$\ifx#7\scriptstyle\scriptscriptstyle\else\scriptstyle\fi#5#1#6\m@th$}%
  \sbox2{$\ifx#7\scriptstyle\scriptscriptstyle\else\scriptstyle\fi#5#2#6\m@th$}%
  \sbox4{$#7\dabar@\m@th$}%
  \dimen@=\wd0 %
  \ifdim\wd2 >\dimen@
    \dimen@=\wd2 %
  \fi
  \count@=2 %
  \def\da@bars{\dabar@\dabar@}%
  \@whiledim\count@\wd4<\dimen@\do{%
    \advance\count@\@ne
    \expandafter\def\expandafter\da@bars\expandafter{%
      \da@bars
      \dabar@ 
    }%
  }%
  \mathrel{#3}%
  \mathrel{%
    \mathop{\da@bars}\limits
    \ifx\\#1\\%
    \else
      _{\copy0}%
    \fi
    \ifx\\#2\\%
    \else
      ^{\copy2}%
    \fi
  }%
  \mathrel{#4}%
}
\makeatother
\newcommand\xdashto\xdashrightarrow
\newcommand\xdashfrom\xdashleftarrow
\newcommand\dashtoo{\xdashto{\phantom\too}} 

\makeatletter
\providecommand{\leftsquigarrow}{%
  \mathrel{\mathpalette\reflect@squig\relax}%
}
\newcommand{\reflect@squig}[2]{%
  \reflectbox{$\m@th#1\rightsquigarrow$}%
}
\makeatother

\newcommand{\actson}{\curvearrowright}

\newcommand*{\rom}[1]{\textup{\uppercase\expandafter{\romannumeral#1}}}
\newcommand{\tbf}{\textbf}
\newcommand{\bp}[1]{\tbf{(#1)}} 

\newcommand{\sm}{\setminus}

\renewcommand{\l}{\ell}

\newcommand{\wh}{\widehat}
\newcommand{\ul}{\underline}

\renewcommand{\ast}[1]{#1^*}
\newcommand{\twocases}[4][.]{
	\begin{cases}
		\hfill\displaystyle #2\hfill&\text{if }\displaystyle #3\\
		\hfill\displaystyle #4\hfill&\text{otherwise#1}
	\end{cases}
}

\newcommand{\Twocases}[4]{
	\begin{cases}
		\hfill\displaystyle #1\hfill&\text{if }\displaystyle #2\\
		\hfill\displaystyle #3\hfill&\text{if }\displaystyle #4
	\end{cases}
}

\newcommand{\push}[1]{#1_*}
\newcommand{\ppush}[1]{\p{#1}_*}
\newcommand{\pull}[1]{#1^*}
\newcommand{\ppull}[1]{\p{#1}^*}
\newcommand{\by}{\times}
\newcommand{\mapdesc}[5]{
	\begin{matrix}
		\ifblank{#1}{}{\displaystyle#1:}&\displaystyle#2&\longrightarrow&\displaystyle#3\\
		&\displaystyle#4&\longmapsto&\displaystyle #5
	\end{matrix}
}

\renewcommand{\bar}{\overline}

\DeclareMathOperator{\id}{id}

\DeclareMathOperator{\pr}{pr}

\newcommand\colonequals\coloneqq
\newcommand\equalscolon\eqqcolon

\newcommand{\tand}{\,\text{ and }\,} 

\newcommand{\twhere}{\,\text{ where }\,}
\newcommand{\twith}{\,\text{ with }\,}
\newcommand{\twhile}{\,\text{ while }\,}

\newcommand{\tfor}{\,\text{ for }\,}

\newcommand{\tcomma}{\text{, }\,\,}

\newcommand{\tforall}{\,\text{ for all }\,}

\newcommand{\tforany}{\,\text{ for any }\,}
\newcommand{\tforsome}{\,\text{ for some }\,}

\renewcommand\t\text 
\newcommand\ttt\texttt

\newcommand\emphasize[1]{{\color{violet}{\texttt{#1}}}} 
\newcommand{\important}[1]{\emph{#1}}
\newcommand\noteworthy\important

\newcommand{\define}[1]{\emphasize{#1}\index{#1}}

\pgfmathdeclarerandomlist{thuses}{{Thus}{Therefore}{As such}{As we wished}{By this great fortune}{Alors}{Donc}{Following this thread}{Consequently}{Thusforth}{Rejoice}{This can only mean one thing...}{Ergo}{Voila}}

\pgfmathdeclarerandomlist{fibers}{{fiber}{fibre}}

\pgfmathdeclarerandomlist{wewins}{{we win}{it's over}{all is good}{victory has been achieved}{we've crossed the finish line}{we're happy campers}{we've done it}{that's a wrap}{we have the high ground}}

\tikzset{%
	symbol/.style={%
		,draw=none
		,every to/.append style={%
			edge node={node [sloped, allow upside down, auto=false]{$#1$}}}
	}
}

\makeatletter
\def\renewtheorem#1{%
	\expandafter\let\csname#1\endcsname\relax
	\expandafter\let\csname c@#1\endcsname\relax
	\gdef\renewtheorem@envname{#1}
	\renewtheorem@secpar
}
\def\renewtheorem@secpar{\@ifnextchar[{\renewtheorem@numberedlike}{\renewtheorem@nonumberedlike}}
\def\renewtheorem@numberedlike[#1]#2{\newtheorem{\renewtheorem@envname}[#1]{#2}}
\def\renewtheorem@nonumberedlike#1{  
	\def\renewtheorem@caption{#1}
	\edef\renewtheorem@nowithin{\noexpand\newtheorem{\renewtheorem@envname}{\renewtheorem@caption}}
	\renewtheorem@thirdpar
}
\def\renewtheorem@thirdpar{\@ifnextchar[{\renewtheorem@within}{\renewtheorem@nowithin}}
\def\renewtheorem@within[#1]{\renewtheorem@nowithin[#1]}
\makeatother

\DeclareDocumentEnvironment{MyFrame}{O{1cm}O{0.4pt}O{0.8cm}O{black}O{3}O{2ex}}
{\par\hfill\rlap{%
		\bgroup\color{#4}%
		\hskip-\dimexpr#1-#3\relax\rule{#1}{#2}%
		\hskip-\dimexpr#1/#5\relax\rule[-\dimexpr#1-\dimexpr#1/#5\relax]{#2}{#1}%
		\egroup
	}%
	\vskip-\dimexpr#1/#5+\dimexpr#1/#5-#6\relax%
}
{\par\nobreak\offinterlineskip\vskip-\dimexpr#1/#5+\dimexpr#1/#5-#6\relax\noindent%
	\hskip-#3\bgroup\color{#4}%
	\rule{#1}{#2}\hskip-\dimexpr#1-\dimexpr#1/#5-#2\relax%
	\rule[-\dimexpr#1/#5-#2\relax]{#2}{#1}\egroup\par
}

\newtheorem{innercustomgeneric}{\customgenericname}
\providecommand{\customgenericname}{}
\newcommand{\newcustomtheorem}[2]{%
	\newenvironment{#1}[1]
	{%
		\renewcommand\customgenericname{#2}%
		\renewcommand\theinnercustomgeneric{##1}%
		\innercustomgeneric
	}
	{\endinnercustomgeneric}
}

\makeatletter
\@ifclassloaded{amsart}{
}{%
    \@ifundefined{subjclass}{
    \newcommand{\subjclass}[2][1991]{%
      \let\@oldtitle\@title%
      \gdef\@title{\@oldtitle\footnotetext{#1 \emph{Mathematics subject classification.} #2}}%
    }
    }{}
    \@ifundefined{keywords}{
    \newcommand{\keywords}[1]{%
      \let\@@oldtitle\@title%
      \gdef\@title{\@@oldtitle\footnotetext{\emph{Key words and phrases.} #1.}}%
    }
    }{}
}
\makeatother

\theoremstyle{plain}
\newtheorem{thm}{Theorem}
\newtheorem{thma}{Theorem}
\newtheorem{lemma}[thm]{Lemma}
\newtheorem{cor}[thm]{Corollary}

\renewtheorem{conj}[thm]{Conjecture}

\newtheorem{prop}[thm]{Proposition}

\newtheorem{qn}[thm]{Question}

\newcustomtheorem{Prob}{Problem}
\newcustomtheorem{Exc}{Theoretical Exercise}

\theoremstyle{definition}
\newtheorem{defninner}[thm]{Definition}

\newtheorem{notn}[thm]{Notation}

\newtheorem{recinner}[thm]{Recall}
\newtheorem{setinner}[thm]{Setup}
\newtheorem{exinner}[thm]{Example}
\newtheorem{reminner}[thm]{Remark}

\newtheorem*{exampleinner}{Example}
\newtheorem*{assump}{Assumption}

\newtheorem*{nonexinner}{Non-example}
\newtheorem{warninner}[thm]{Warning}

\newtheorem*{ansinner}{Answer}

\newtheorem*{histinner}{History}

\theoremstyle{remark}

\newtheorem{constructinner}[thm]{Construction}

\crefname{thm}{Theorem}{Theorems}
\crefname{prop}{Proposition}{Propositions}
\crefname{lemma}{Lemma}{Lemmas}
\crefname{cor}{Corollary}{Corollaries}
\crefname{prob}{Problem}{Problems}
\crefname{exinner}{Example}{Examples}
\crefname{reminner}{Remark}{Remarks}
\crefname{ansinner}{Answer}{Answers}
\crefname{warninner}{Warning}{Warnings}
\crefname{nonexinner}{Non-example}{Non-examples}
\crefname{recinner}{Recall}{Recall}
\crefname{defninner}{Definition}{Definitions}
\crefname{histinner}{History}{History}
\crefname{constructinner}{Construction}{Constructions}

\newcommand\exsymbol{$\triangle$}
\newcommand\remsymbol{$\circ$}
\newcommand\anssymbol{$\star$}
\newcommand\warnsymbol{$\bullet$}
\newcommand\proofsymbol{$\blacksquare$}
\newcommand\nonexsymbol{$\triangledown$}
\newcommand\recsymbol{$\odot$}
\newcommand\defnsymbol{$\diamond$}
\newcommand\histsymbol{$\ominus$}
\newcommand\constructsymbol{$\octagon$}
\newcommand\setupsymbol{$\circledcirc$}

\newenvironment{ex}[1][]{\begin{exinner}[#1]\pushQED{\qed}\renewcommand\qedsymbol\exsymbol}{\popQED\end{exinner}}

\newenvironment{rem}[1][]{\begin{reminner}[#1]\pushQED{\qed}\renewcommand\qedsymbol\remsymbol}{\popQED\end{reminner}}

\newenvironment{warn}[1][]{\begin{warninner}[#1]\pushQED{\qed}\renewcommand\qedsymbol\warnsymbol}{\popQED\end{warninner}}

\newenvironment{rec}[1][]{\begin{recinner}[#1]\pushQED{\qed}\renewcommand\qedsymbol\recsymbol}{\popQED\end{recinner}}
\newenvironment{defn}[1][]{\begin{defninner}[#1]\pushQED{\qed}\renewcommand\qedsymbol\defnsymbol}{\popQED\end{defninner}}

\newenvironment{set}[1][]{\begin{setinner}[#1]\pushQED{\qed}\renewcommand\qedsymbol\setupsymbol}{\popQED\end{setinner}}

\let\proofinner=\proof
\newcommand\myproof[1][Proof]{\proofinner[#1]\renewcommand\qedsymbol\proofsymbol}
\def\proof{\myproof}

\usepackage[T1]{fontenc} 

\usepackage[backend=bibtex,style=alphabetic,url=true,backref=true,doi=false,isbn=false]{biblatex}
\usepackage{orcidlink}

\numberwithin{thm}{section}
\numberwithin{equation}{section}

\theoremstyle{definition}
\newtheorem{defnpropinner}[thm]{Definition/Proposition}
\newenvironment{defnprop}[1][]{\begin{defnpropinner}[{#1}]\pushQED{\qed}\renewcommand\qedsymbol\defnsymbol}{\popQED\end{defnpropinner}}

\crefname{thma}{Theorem}{Theorems}

\newcommand{\niven}[1]{{\color{magenta} \textsf{$\spadesuit\spadesuit\spadesuit$ Niven: [#1]}}}

\newcommand*\centermathcell[1]{\omit\hfil$\displaystyle#1$\hfil\ignorespaces}

\newcommand{\thickslash}{\mathbin{\!\!\pmb{\fatslash}}}
\DeclareMathOperator\Cor{Cor}

\author{Niven Achenjang \orcidlink{0000-0001-9551-5821}}
\address{Niven T. Achenjang \\
	Department of Mathematics \\
	Harvard University \\
	Cambridge, MA 02138 }
\email{achenjang@math.harvard.edu}

\title{On Brauer Groups of Tame Stacks}

\newcommand\ghom{\hom_{0}}
\DeclareMathOperator\Nat{Nat}
\DeclareMathOperator\GrpSch{GrpSch}
\DeclareMathOperator\Ideg{Ideg}

\date{\today}

\begin{document}

\begin{abstract}
    We develop some general tools for computing the Brauer group of a tame algebraic stack $\meX$ by studying the difference between it and the Brauer group of the coarse space $X$ of $\meX$. It is our hope that these tools will be used to simplify future computations of Brauer groups of stacks. Informally, we show that $\Br\meX$ is often built from $\Br X$ and ``information about the Picard groups of the fibers of $\meX\to X$''. Along these lines, we compute, for example, the Brauer group of the moduli stack $\meY(1)_S$ of elliptic curves, over any regular noetherian $\Z[1/2]$-scheme $S$ as well as the Brauer groups of (many) stacky curves (allowing generic stabilizers) over algebraically closed fields.
\end{abstract}
\maketitle
\vspace{-3em}

\noindent\hrulefill
\tableofcontents
\vspace{-3em}
\noindent\hrulefill
\hypersetup{linkcolor=\linkcolor}

\section{\bf Introduction}

Brauer groups of fields were classically studied objects whose definition was generalized to rings in work of Azumaya, Auslander, and Goldman \cite{azumaya,auslander-goldman}, and then later to schemes in work of Grothendieck \cite{Gr1,Gr2,Gr3}. They have been cemented as important cohomological invariants for their applications to class field theory, to understanding $\l$-adic cohomology (especially of curves), and to obstructions to rational points on varieties. In recent times, there has been growing interest in extending our understanding of Brauer groups from schemes to stacks. There is, for example, the breakthrough work of Antieau and Meier \cite{AM} on the Brauer group of the moduli stack $\meY(1)$ of elliptic curves (which inspired papers such as \cite{shin,meier-cs,achenjang2024brauer}), work of Lieblich \cite{lieblich2011period} who used the Brauer group of $B\mu_n$ to study the period-index problem, work of Shin computing the Brauer groups of a variety of stacks \cite{shin-Gm,shin-wps,Shin-Br=Br'}, work of Di Lorenzo and Pirisi computing Brauer groups via their theory of cohomological invariants \cite{dilorenzo2022cohomological,DiLorenzo-Pirisi-hyp}, and work of Santens \cite{santens2023brauermanin} who studied Brauer-Manin obstructions on (generically schemey) stacky curves. There is also recent work of Kresch--Tschinkel \cite{kresch2024unramifiedbrauergroupquotient} and Pirutka--Zhang \cite{pirutka2024computingequivariantbrauergroup} on computing Brauer groups of quotients of nice varieties by generically free finite group actions.

That said, especially in the context of stacky curves, many of these papers are largely devoted to computing the Brauer groups of very specific stacks (often even just a single stack viewed over various bases). For example, between the papers \cite{AM,shin,meier-cs,dilorenzo2022cohomological,achenjang2024brauer} which compute Brauer groups of modular curves, the only two such curves considered have been $\meY(1)$ and the moduli stack $\meY_0(2)$ of elliptic curves equipped with an \'etale subgroup of order $2$. It is the goal of the present paper to develop general enough tools for computing Brauer groups of (tame) stacks (especially of stacky curves; see \cref{sect:conventions} for a precise definition) so that, informally speaking, such computations can be carried out in only a handful of pages instead of a whole paper. We work out some examples in \cref{sect:examples} to help the reader evaluate how well we achieve this goal.

Let $\meX$ be a tame algebraic stack (in the sense of \cite{AOV}; see \cref{defn:tame}), with coarse space $c\colon\meX\to X$. We aim to study $\Br'(\meX)\coloneqq\hom^2(\meX,\G_m)_{\t{tors}}$ (or, rather, the cokernel of $\pull c\colon\Br' X\to\Br'\meX$) via the Leray spectral sequence. Our primary tool for making this study feasible is the notion of $\meX$ being `locally Brauerless', introduced in \cref{sect:vanishing-result}. This is a condition on the geometric automorphism groups of $\meX$ (i.e. those automorphism groups associated to points defined over separably closed fields) which, roughly speaking, ensures that, Brauer classes on $\meX$ vanish \'etale-locally over $X$. This is made precise in our main technical result.
\begin{thma}\label{thma:main}
    If $\meX$ is `locally Brauerless' in the sense of \cref{defn:locally-Brauerless}, then $\homR^2\push c\G_m=0$. Consequently, there is an exact sequence
    \begin{equation}\label{ses:main-theorem}
        \dots\too\hom^2(X,\G_m)\xtoo{\pull c}\hom^2(\meX,\G_m)\too\hom^1(X,\homR^1\push c\G_m)\too\dots
    \end{equation}
\end{thma}
\begin{proof}
    This is \cref{thm:R2Gm-vanishes}, and a longer version of \cref{ses:main-theorem} appears in \cref{cor:R2Gm-vanishes-sequence}. 
\end{proof}
\begin{rem}
    \cref{thma:main} is a generalization of \cite[Theorem 6]{meier-cs}, which essentially proves the DM case.\footnote{The word `tame' does not appear in \cite{meier-cs}, but realizing that this notion was lurking in the background of his work was one piece of motivation for starting this project.}
\end{rem}
\begin{rem}\label{rem:main-theorem-explanation}
    In words, \cref{thma:main} says that understanding the Brauer group $\hom^2(\meX,\G_m)$ of $\meX$ can be reduced to understanding the Brauer group of $X$ along with understanding the Picard groups of the fibers of $c$ (this is what $\homR^1\push c\G_m$ keeps track of); we include many examples showing these can often be reasonably well understood.
\end{rem}
The strategy alluded to in \cref{rem:main-theorem-explanation} is used in \cref{sect:examples} to carry out the following example computations.
\begin{thma}\label{thma:mod-curve}
    Let $\meY(1)$ (resp. $\meX(1)$, resp. $\meY_0(2)$) denote the moduli stack of elliptic curves (resp. generalized elliptic curves, resp. elliptic curves equipped with a subgroup of order $2$), and let $S$ be a noetherian $\Z[1/2]$-scheme. Then,
    \begin{enumerate}
        \item $\Br'\meX(1)_S\simeq\Br'\P^1_S\simeq\Br' S$.
        \item If $S$ is regular, there is an explicit isomorphism
        \[\Br'\A^1_S\oplus\hom^1(S,\zmod{12})\isoo\Br'\meY(1)_S.\]
        \item If $S$ is regular, there is an explicit isomorphism
        \[\Br'(\A^1_S\sm\{0\})\oplus\hom^1(S,\zmod4)\oplus\hom^0(S,\zmod2)\isoo\Br'\meY_0(2)_S.\]
    \end{enumerate}
\end{thma}
\begin{proof}
    Part \bp1 is proved in \cref{cor:BrX(1)-general} (note that $\Br'\P^1_S\simeq\Br'S$ by \cite[Chapter II, Theorem 2]{gabber:thesis}), part \bp2 is proved in \cref{thm:BrY(1),prop:Y(1)-splitting}, and part \bp3 is proved in \cref{thm:BrY0(2)-gen-1,prop:Y0(2)-splitting}.
\end{proof}
\begin{rem}
    Part \bp2 of \cref{thma:mod-curve} extends results of \cite{AM,meier-cs,dilorenzo2022cohomological}, at least away from characteristic $2$, and part \bp3 extends \cite[Theorem 1.1]{achenjang2024brauer}.
\end{rem}
\begin{rem}
    One potential future application of results like \cref{thma:mod-curve} may be to aid in computing integral Brauer--Manin or \'etale--Brauer obstructions to points on modular curves. In particular, the author suspects that the integral \'etale--Brauer obstruction set $\meY(1)(\R\by\wh\Z)^{\t{\'et-Br}}$ should be empty and would be excited to see such a computation carried out; this would give another proof of Tate's theorem (appearing in \cite{Ogg-EC-2power}) that there is no elliptic curve $E/\Q$ with everywhere good reduction.
\end{rem}
\begin{rem}
    This paper focuses mainly on stacks which are tame (or at least generically tame, e.g. $\meY(1)$ in characteristic $3$). However, it would be interesting to also understand how one can effectively compute Brauer groups of certain classes of wild stacks. For example, Shin \cite{shin} and Di Lorenzo--Pirisi \cite{dilorenzo2022cohomological} were able to compute $\Br'\meY(1)_k$ for characteristic two fields $k$; they show it is always an extension $0\to\Br'\A^1_k\oplus\hom^1(k,\zmod{12})\to\Br'\meY_0(2)_k\to\zmod2\to0$. In light of \cref{thma:mod-curve}\bp2, we view this $\zmod2$ as a `wild piece of the Brauer group' and think it would be interesting to know how to predict/compute such `wild pieces' more generally.
\end{rem}
The generality of \cref{thma:main} allows it to readily apply not just to specific examples like those considered above, but also to more general classes of stacks. To state the next result, we need the notion of a `multiply rooted stack'. By this, we simply mean the result of rooting a scheme along multiple (disjoint) divisors, each to potentially different degrees.
\begin{thma}\label{thma:stacky-tsen}
    Let $k$ be an algebraically closed field, let $\meY$ be a multiply rooted stack (i.e. of the form \cref{eqn:iterated-root-stack}) over a smooth $k$-curve $X$, let $G/k$ be a finite commutative linearly reductive group, and let $\meX\to\meY$ be a $G$-gerbe. If $\meX$ is locally Brauerless (in the sense of \cref{defn:locally-Brauerless}), then $\hom^2(\meX,\G_m)\simeq\hom^1(X,\dual G)$.
\end{thma}
\begin{proof}
    This is a special case of \cref{prop:stacky-tsen}.
\end{proof}
\begin{rem}
    \hfill\begin{itemize}
        \item In words, \cref{thma:stacky-tsen} says that if $\meX$ is a suitable stacky curve over $k=\bar k$ with coarse space $X$ and `generic stabilizer' $G/k$, then $\hom^2(\meX,\G_m)\simeq\hom^1(X,\dual G)$.
        \item Furthermore, it is well known that any tame DM stacky curve over a field is a gerbe over a multiply rooted stack (see \cite[Propositions 1.5 and 2.19]{lopez2023picard}).
        \qedhere
    \end{itemize}
\end{rem}
\begin{rem}
    In \cite[Remark 1.3]{achenjang2024brauer}, my co-authors and I remarked that Tsen's theorem (that Brauer groups of curves over algebraically closed fields vanish) continues to hold for some tame stacky curves, but fails for others (e.g. $\Br\meY(1)_{\bar k}=0$ but $\Br\meY_0(2)_{\bar k}=\zmod2$ by \cref{thma:mod-curve}\bp{2,3}). However, at the time, we were unable to determine what governed this. \cref{thma:stacky-tsen} helps clarify the situation.
\end{rem}
\begin{rem}\label{rem:prop-stacky-tsen}
    \cref{thma:stacky-tsen} can be strengthened if one assumes that $\meX$ is proper. We informally state two such extensions here and point the reader towards \cref{sect:stacky-curve-prop} for precise details.
    \begin{enumerate}
        \item If $\meX$ is a suitable proper stacky curve over a \important{separably closed} field $k$ with coarse space $X$ and `generic stabilizer' $G/k$, then $\hom^2(\meX,\G_m)\simeq\hom^1(X,\dual G)$; see \cref{cor:Br-prop-reg-local}.
        \item More generally, if $\meX$ is a suitable proper stacky curve over a regular, noetherian base scheme $S$ then $\hom^2(\meX,\G_m)$ admits a filtration whose graded pieces are
        \[\hom^2(X,\G_m)\tcomma\coker\p{\hom^1(S,\Pic_{X/S})\to\hom^1(S,\Pic_{\meX/S})},\,\,\t{and a subgroup of }\hom^0(S,\homR^1\push g\dual G),\]
        where $g\colon X\to S$ is the structure map of the coarse space.
        This is proved in \cref{thm:Br-prop-best statement I could muster}. Furthermore, while most of the examples in \cref{thma:mod-curve} are not proper, one can still see the above three pieces showing up in their Brauer groups.
        \qedhere
    \end{enumerate}
\end{rem}

The results quoted thus far ultimately come about from computing various spectral sequences. However, when computing Brauer groups of schemes, one often times makes use of residue maps in addition to (or in lieu of) spectral sequences. Along these lines, separate from \cref{thma:main}, we study Brauer residue maps on stacks in \cref{sect:res-exact}. Of note, this study allows us to compute \important{Picard groups} of stacky curves using their coarse space and generic residual gerbe, and it allows us to obtain a stacky version of the usual Faddeev exact sequence used for computing Brauer groups of (not necessarily proper) rational curves.
\begin{thma}[= \cref{cor:Pic-Cl}]\label{thma:Pic-Cl}
    Let $\meX$ be a regular, integral, noetherian stacky curve with coarse moduli space $c\colon\meX\to X$. Let $j\colon\meG\into\meX$ denote the residual gerbe of the generic point of $\meX$. Then, there is a short exact sequence
    \[0\too\Cl\meX\too\Pic\meX\xtoo{\pull j}\Pic\meG\too0,\]
    where $\Cl\meX$ is the (Weil) divisor class group of $\meX$ (see \cref{cor:Pic-Cl} for a precise definition).
\end{thma}
\begin{thma}[Stacky Faddeev]\label{thma:stacky-faddeev}
    Let $k$ be a perfect field, let $\P^1=\P^1_k$, and let $x_1,\dots,x_r\in\P^1(k)$ be distinct points. Choose pairwise coprime integers $e_1,\dots,e_r>1$, and let
    \[\meX\coloneqq\sqrt[e_1]{x_1/\P^1}\by_{\P^1}\dots\by_{\P^1}\sqrt[e_r]{x_r/\P^1},\]
    be $\P^1$ rooted at the $x_1,\dots,x_r$ by degrees $e_1,\dots,e_r$. Then,
    \begin{itemize}
        \item $\Br\meX\simeq\Br\P^1\simeq\Br k$.
        \item For any closed point $x\in\P^1$, set $e_x=1$ if $x\not\in\{x_1,\dots,x_r\}$ and set $e_x=e_i$ if $x=x_i$. Let $N\coloneqq\prod_{i=1}^re_i=\prod_xe_x$. There is an exact sequence
        \begin{equation}\label{es:stacky-faddeev}
            0\to\Br k\too\Br k(\P^1)\xto{\bigoplus_xe_x\cdot\res_x}\bigoplus_x\hom^1(\kappa(x),\Q/\Z)\xto{\Cor_\meX}\hom^1(k,\Q/\Z)\too0,
        \end{equation}
        where $\res_x\colon\Br k(\P^1)\to\hom^1(\kappa(x),\Q/\Z)$ is the usual Brauer residue map and $\Cor_\meX=\sum_x\frac N{e_x}\Cor_{\kappa(x)/k}$.
    \end{itemize}
\end{thma}
\begin{proof}
    This is a special case of \cref{prop:stacky-Faddeev}.
\end{proof}
There are more general results on residues maps of stacky curves (and their relation to residue maps on the coarse space) in \cref{sect:res-exact}. For example, when $\meX$ is a (generically schemey) stacky $\P^1$, \cref{sect:faddeev} contains a description of $\Br\meX$ even if the orders of its stacky points are \important{not} pairwise coprime.

\subsubsection*{\bf Computations beyond the tame, locally Brauerless setting}
As evidenced by our main \cref{thma:main}, this paper mainly focuses on stacks which are tame and locally Brauerless (see \cref{defn:locally-Brauerless}). However, there are many interesting stacks which fail to have one or both of these properties. For such stacks, the techniques and ideas of this paper can, nevertheless, still be used to significantly aid in the computations of their Brauer groups. While we stay in the tame, locally Brauerless setting in the body of this paper -- to keep its scope contained -- we collect here a few ideas which can be used on stacks $\meX$ which are not tame or not locally Brauerless. Write $c\colon\meX\to X$ for $X$'s coarse space map, assuming one exists.

\begin{enumerate}
    \item $\meX$ may be generically tame and locally Brauerless.

    Suppose there exists a dense open $U\subset X$ above which $\meX$ \important{is} tame and locally Brauerless, i.e. so that $\meU\coloneqq\inv c(U)$ is tame and locally Brauerless. Then, \cref{thma:main} still applies to show that $\homR^2\push c\G_m\vert_U=0$ and so a computation of $\meX$'s Brauer group can begin by checking, by hand, whether the stalks of $\homR^2\push c\G_m$ also vanish over $X-U$. If so, then the philosophy of \cref{rem:main-theorem-explanation} still applies and one can carry out the rest of their computation in the style of this paper.
    \begin{ex}[$\meX(1)_{\F_3}$]
        This idea is used in \cref{sect:BrX(1)/Z[1/2]} to compute the Brauer group of $\meX\coloneqq\meX(1)_{\F_3}$, for example. Note that $\meX$ is not tame above points with $j$-invariant $0$. However, the open subscheme $\meU\subset\meX$ of curves with nonzero (possibly infinite) $j$-invariant is tame and locally Brauerless, so \cref{thma:main} shows that $\homR^2\push c\G_m$ is supported on the point $j=0$. In \cref{prop:char-3-j0-stalk}, we show that vanishes above this point as well, so $\homR^2\push c\G_m=0$ despite \cref{thma:main} not directly applying to all of $\meX$. From here, we are able to finish the computation, ultimately proving \cref{thma:mod-curve}\bp1. This same by-hand computation (i.e. \cref{prop:char-3-j0-stalk}) feeds into the proof of \cref{thma:mod-curve}\bp2 as well.
    \end{ex}

    \item One can work with ``locally Brauerless morphisms''

    For conceptual simplicity, in this paper, we defined `locally Brauerless' to be an absolute property of a $\meX$. However, just like the notion of `tame algebraic stack' can be relativized to a property of morphisms of stacks (see \cite[Definition 3.3]{AOV-maps}), one can also similarly define what it means for a morphism $f\colon\meX\to\meY$ of algebraic stacks to be `locally Brauerless'.\footnote{For example, one can say $f$ is \important{locally Brauerless} if, for every every faithfully flat map $U\to\meY$, from an algebraic space $U$, the stack $\meX\by_\meY U$ is locally Brauerless in the sense of \cref{defn:locally-Brauerless}.} \cref{thma:main} readily generalizes to `tame, locally Brauerless morphisms' and so, if $\meX$ is not locally Brauerless, one may still hope to compute its Brauer group by factoring $c\colon\meX\to X$ as a sequence of locally Brauerless morphisms.
    \begin{ex}[$B\mu_n\by\mu_n$]
        As an example of this idea, here we sketch a computation of the Brauer group of $B(\mu_n\by\mu_n)$, the classifying stack for $(\mu_n\by\mu_n)$-torsors, once one has defined a relative notion of `locally Brauerless'. Fix any $n\ge1$ and let $F$ be a field; set $\meX\coloneqq B(\mu_n\by\mu_n)_F$. It follows from \cref{rem:comm-Brauerless=cyclic} that $\meX$ is \important{not} locally Brauerless, though it is tame, so \cref{thma:main} (and, more directly, \cref{prop:BG-Gm-coh}) does \important{not} applies to it. However, the coarse space map $c\colon\meX\to\spec F$ factors as
        \[\meX=B(\mu_n\by\mu_n)_F=B\mu_{n,F}\by_FB\mu_{n,F}\xtoo fB\mu_{n,F}\xtoo g\spec F,\]
        with $f,g$ both tame and locally Brauerless. Thus, two applications of \cref{prop:BG-Gm-coh} shows that
        \begin{alignat*}{5}
            \hom^2(B(\mu_n\by\mu_n)_F,\G_m) &\simeq \hom^2(B\mu_{n,F},\G_m) &&\oplus\,\,&&\hom^1(B\mu_{n,F}, \zmod n) \quad &&\tand\\
            \hom^2(B\mu_{n,F},\G_m) &\simeq \hom^2(F,\G_m) &&\oplus &&\hom^1(F,\zmod n)
            .
        \end{alignat*}
        It is not hard to show that $\hom^1(B\mu_{n,F},\zmod n)\simeq\hom^1(F,\zmod n)\oplus\Hom_F(\mu_n,\zmod n)$ (morphisms as $F$-group schemes); in brief, a $\zmod n$-torsor over $B\mu_{n,F}$ is a $\zmod n$-torsor $T$ over $\spec F$ equipped with an action of $\mu_n$ by torsor maps (i.e. a morphism $\mu_n\to\ul\Aut(T)=\zmod n$). Thus, we easily compute that
        \begin{align*}
            \hom^2(B(\mu_n\by\mu_n)_F,\G_m)
            &\simeq\hom^2(F,\G_m)\oplus\hom^1(F,\zmod n)\oplus\hom^1(F,\zmod n)\oplus\Hom_F(\mu_n,\zmod n)\\
            &\simeq\hom^2(F,\G_m)\oplus\hom^1(F,\zmod n\by\zmod n)\oplus\Hom_F(\mu_n,\zmod n)
            ,
        \end{align*}
        despite $B(\mu_n\by\mu_n)_F$ not being locally Brauerless. Note that this example generalizes the computation in \cite[Proposition 4.3.2]{lieblich2011period}; in this citation, the author assumes $\mu_{n,F}\simeq\ul{\zmod n}_F$ and writes down an expression which is more canonically identified with $\hom^2(B(\zmod n\by\zmod n)_F,\G_m)$ (after killing the constant classes).
    \end{ex}
\end{enumerate}

\subsubsection*{\bf Related papers appearing after this one} Since the posting of the first version of this paper, a couple other works studying Brauer groups of stacks which independently prove a few of the results appearing in this one have appeared. Since these may also be of interest to the reader, we mention them here.

In \cite{loughran2025mallesconjecturebrauergroups}, the authors study Brauer groups of stacks as they relate to Malle's conjecture on counting number fields. This, in particular, leads them to studying Brauer groups of classifying stacks $BG$, with the connection to Malle coming from the fact that, roughly speaking, a $G$-Galois extension of $\Q$ corresponds to a $\Q$-point on $BG$. In \cite[Section 5]{loughran2025mallesconjecturebrauergroups}, they prove versions of some of our results on classifying stacks, root stacks, and residue maps (see \cite[Remark 5.31]{loughran2025mallesconjecturebrauergroups}), along with new results on unramified Brauer groups of stacks (a concept not touched upon here).

In \cite{bishop2025brauergroupstamestacky}, the author studies Brauer groups of $\mu_r$-gerbes over tame (generically schemey, DM) stacky curves. As mentioned in the introduction to \cite{bishop2025brauergroupstamestacky}, neither it nor this paper is strictly more general than the other. In \cite{bishop2025brauergroupstamestacky}, the author works with ($\mu_r$-gerbes over generically schemey) tame DM stacky curves without further restriction on the stabilizer groups appearing; in particular, there is no required analogue of the `locally Brauerless' condition which is central to this paper. This allows the author to prove new results in cases not handled by this paper. Because their paper was written independently of this one, they end up reproving some of our results along the way; for example, some of the results of \cref{sect:gerbes} appear in \cite[Section 4]{bishop2025brauergroupstamestacky}.

\subsubsection*{\bf Paper organization} We begin in \cref{sect:background} by briefly introducing the background material used throughout the paper. In particular, it is here that we include the definitions of tameness and of Brauer groups. Following this, we begin our march towards \cref{thma:main} in earnest in \cref{sect:Blocally-diagonalizable}. As we will later mention, the automorphism groups of tame stacks are (locally) extensions of tame \'etale group schemes by diagonalizable group schemes. Related to this, before tackling \cref{thma:main} directly, we find it helpful to first compute the Brauer groups of classifying stacks of certain locally diagonalizable groups schemes in \cref{sect:Blocally-diagonalizable}. In \cref{sect:vanishing-result}, we prove \cref{thma:main}. Recall that the goal of this theorem is to show vanishing of some $\homR^2\push c\G_m$. Because such vanishing can be checked on stalks, after appealing to the local structure of tame stacks (\cref{cor:tame-stack-sh}), this amounts to computing cohomology groups of the form $\hom^2([X/G],\G_m)$ with $G$ an extension of a tame constant group by a diagonalizable group. We ultimately compute such groups by combining earlier work of Meier on quotients by tame constant groups \cite{meier-cs} with (a slight generalization of) the work of \cref{sect:Blocally-diagonalizable} on quotients by diagonalizable groups. Once \cref{thma:main} is proven, we work out its consequences for gerbes and root stacks in \cref{sect:gerbes,sect:root-stacks}. Recalling that a stacky curve can often be factored as a gerbe over a root a stack, in \cref{sect:stacky-curve}, we use the results of the previous two sections in order to study Brauer groups of stacky curves. In particular, by leveraging Tsen's theorem that $\Br(X)=0$ if $X$ is a curve over an algebraically closed field, we prove \cref{thma:stacky-tsen} and state some consequences for Brauer groups of stacky curves over non-perfect fields (see e.g. \cref{cor:curve-num-field}). In \cref{sect:stacky-curve-prop}, we study Brauer groups of \important{proper} stacky curves. Because Grothendieck showed that $\Br(X)=0$ if $X$ is a proper curve over a \important{separably closed} field, we can obtain results for bases more general than non-perfect fields in this case. We ultimately prove nontrivial statements about the structure of Brauer groups of proper stacky curves over regular, noetherian base schemes (see \cref{rem:prop-stacky-tsen} and \cref{thm:Br-prop-best statement I could muster}). With this completed, we then turn to the study of residue maps in \cref{sect:res-exact}; it is here that we prove \cref{thma:Pic-Cl,thma:stacky-faddeev}. Finally, in \cref{sect:examples}, we aim to show the utility of our work by carrying out a few example computations; namely, we prove \cref{thma:mod-curve}.

\subsubsection*{\bf Acknowledgements} First and foremost, I thank Deewang Bhamidipati, Aashraya Jha, Caleb Ji, and Rose Lopez, my coauthors on the previous paper \cite{achenjang2024brauer}. It was working on that project that got me interested in Brauer groups of stacks, so this paper would not have been possible without them. I would also like to thank my advisor Bjorn Poonen for many helpful conversations. I thank Lennart Meier for answering some questions I had about his notes \cite{meier-cs}, and I thank Bianca Viray for suggesting to me the possibility of obtaining a `stacky Faddeev' sequence as in \cref{thma:stacky-faddeev}. I have also benefited from helpful conversations with David Zureick-Brown who, among other things, helped me work through some confusions regarding Picard groups of stacky curves. I would also like to ask Tim Santens for answering a question asked in an earlier version of this paper, about residue maps on $\meY(1)$; see \cref{rem:Faddeev-Y(1)}. Finally, during this project, I have been supported by National Science Foundation grant DGE-2141064.

\subsection{Conventions}\label{sect:conventions} We end the introduction by very quickly establishing some of the conventions, many of which are standard, that we use throughout this paper.
\begin{itemize}
    \item By default, given an scheme/algebraic space $X$ (resp. algebraic stack $\meX$), we work in the small \'etale site over $X$ (resp. lisse-\'etale site over $\meX$). As such, all unadorned cohomology groups should be interpreted as \'etale cohomology and the phrase `\define{\'etale sheaf}' should be interpreted as a sheaf on the lisse-\'etale (resp. small \'etale) site of an algebraic stack $\meX$ (resp. algebraic space $X$).
    \begin{itemize}
        \item On occasion, we will make use of the big fppf site on $X$ (resp. flat-fppf site on $\meX$). In such circumstances, we will make use of the subscript $\fppf{}$, writing e.g. $\fhom$ for fppf cohomology.
        \item We will often use, without reference, the fact \cite[Theorem 11.7]{Gr3} that \'etale cohomology and fppf cohomology agree with coefficients in a smooth group scheme.
    \end{itemize}
    \item Given an algebraic stack $\meX$, we write $\Sch_\meX$ to denote the 1-category of schemes over $\meX$ (this is equivalent to the underlying category of $\meX$).
    \item Given an algebraic stack $\meX$ and an object $x\in\meX(S)$ over some scheme $S$, we write $\ul\Aut(\meX,x)$ (or $\ul\Aut_\meX(x)$) for the automorphism group functor of $x$. When $\meX$ is clear from context, we will sometimes denote this simply as $\ul\Aut(x)$.
    \item A `\define{geometric point}' of a scheme $S$ is a map $\bar s\colon\spec\Omega\to S$ from the spectrum of a \important{separably} closed field $\Omega$. We will often simply write $\bar s\to S$ and use $\kappa(\bar s)$ to denote the implicitly chosen field $\Omega$.
    \item By a `\define{gerbe}' over an algebraic stack $\meX$, we mean an `fppf gerbe', i.e. a morphism $\meY\to\meX$ of algebraic stacks realizing $\meY$ as a gerbe over $\Sch_\meX$ equipped with the flat-fppf topology.
    \item By a `\define{group}' we generally mean a `flat, finitely presented group scheme'. When we want to emphasize that we mean a `group in the classical sense of group theory', we will call such a thing an `\define{abstract group}'.
    \begin{itemize}
        \item By a `\define{finite group}' we generally mean a `finite, flat, finitely presented group scheme'.
        \item If $G$ is a finite group over a base scheme $S$, then $G\simeq\Spec_S\msE$ for some locally free $\msO_S$-algebra $\msE$. The rank of this algebra is denoted $\#G$ and is called the order, or cardinality, of $G$.
    \end{itemize}
    \item Given a group $G$ and an integer $N$, we write $G[N]$ for its $N$-torsion subgroup.
    \item Given an abstract group $G$ and a prime $p$, we write $G\{p\}$ for its $p$-primary torsion subgroup.
    \item Given a complex $\msC^\bullet$ with differentials $\d^n\colon\msC^n\to\msC^{n+1}$, we set
    \[\pi_n(\msC^\bullet)\coloneqq\frac{\ker\d^n}{\im\d^{n-1}}\]
    and call these the `\define{homotopy groups}' of the complex. This notation/terminology is chosen to distinguish them from the (hyper)cohomology groups of $\msC^\bullet$ which are computed as the homotopy groups of the relevant derived functor applied to $\msC^\bullet$.
    \item We write $\abs\meX$ for the underlying topological space of an algebraic stack, and we write $\abs\meX_1$ for its set of codimension $1$ points.
    \item We call a scheme $S$ `\define{local}' if $S\cong\spec R$ for a local ring $R$, and we call it (and also $R$) `\define{strictly local}' if furthermore $R$ is strictly henselian.
    \item For us, a `\define{stacky curve}' is a separated, noetherian algebraic stack which is of pure dimension $1$ and has finite inertia.
    \item Let $k$ be a field. We say a $k$-scheme $X$ is \define{nice} if it is smooth, projective, and geometrically connected.
\end{itemize}

\section{\bf Background \& preliminaries}\label{sect:background}
\subsection{Hochschild cohomology}

We will make use of Hochschild cohomology throughout this work, so here we collect some of its basic definitions and properties. See \cite[Chapter 15]{milne-alg-grps} for more information.
\begin{defn}\label{defn:hochschild-cohomology}
    Let $S$ be a scheme (or algebraic stack), let $G$ be a group-valued functor on $\Sch_S\op$, and let $M$ be a $G$-module, i.e. a commutative group functor on which $G$ acts by group homomorphisms. We define a complex $C^\bullet(G,M)$ with $C^n(G,M)\coloneqq\Nat(G^n,M)$, the set of natural transformations (of \important{set-valued} functors) from $G^n$ to $M$ and whose differential $\d^n\colon C^n(G,M)\to C^{n+1}(G,M)$ is defined as usual (see \cite[Section 15.b]{milne-alg-grps}). The \define{Hochschild cohomology} of $G\actson M$ is defined to be the homotopy of this complex:
    \[\ghom^n(G,M)\coloneqq\pi_n\p{C^\bullet(G,M)}\coloneqq\frac{\ker\d^n}{\im\d^{n-1}}.\qedhere\]
\end{defn}
\begin{rem}
    In low degrees, Hochschild cohomology has interpretations analogous to those used for classical group cohomology. Let $G,M,S$ be as in \cref{defn:hochschild-cohomology}.
    \begin{itemize}
        \item $\ghom^0(G,M)=M^G(S)=M(S)^G$ computes $G$-invariants.
        \item $\ghom^1(G,M)=\{$crossed homomorphisms $G\to M\}/\{$principal crossed homomorphisms$\}$, with (principal) crossed homomorphisms defined completely analogously as they are in the classical group cohomology setting.
        \begin{ex}
            If $G\actson M$ trivially, then $\ghom^1(G,M)=\Hom(G,M)$ is the group of homomorphisms from $G$ to $M$.
        \end{ex}
        \item $\ghom^2(G,M)$ classifies isomorphisms classes of \define{Hochschild extensions}, i.e. exact sequences
        \[0\too M\too E\too G\too0\]
        of \important{group functors/presheaves} such that there exists a map $s\colon G\to E$ of \important{set-valued} functors splitting the sequence on the right and such that the conjugation action of $G\actson M$ is the given action of $G\actson M$.
        \qedhere
    \end{itemize}
\end{rem}
\begin{rem}
    If $Q$ is an abstract group with associated constant sheaf $\ul Q$ over a scheme (or algebraic stack) $S$, then, for any $\ul Q$-module $M$, $\ghom^n(\ul Q,M)$ is the classical group cohomology of the abstract group $Q$ acting on $M(S)$. Indeed, the complex $C^\bullet(\ul Q,M)$ of \cref{defn:hochschild-cohomology} is simply the bar complex of $Q\actson M(S)$. For this reason, we will continue to use $\ghom$ to denote classical group cohomology, writing e.g. $\ghom^n(\ul Q,M)\simeq\ghom^n(Q,M(S))$.
\end{rem}
For us, Hochschild cohomology will most commonly come up in the context of computing the (\'etale) cohomology of quotient stacks, so the modules $M$ we encounter will usually be of the following form.
\begin{notn}
    Let $S$ be a scheme, let $X/S$ be an $S$-scheme equipped with a sheaf $\msF$. For any $j\ge0$, we define the functor
    \[\mapdesc{\ul\hom^j(X/S,\msF)}{\Sch_S\op}\Set{T/S}{\hom^j(X_T,\msF_T),}\]
    where $X_T\coloneqq X\by_ST$ and $\msF_T\coloneqq\msF\vert_{X_T}$. We also similarly define $\ul{\fhom^j}(X/S,\msF)\colon\Sch\op_S\to\Set$.
\end{notn}
\begin{lemma}[Descent spectral sequence]\label{lem:descent-ss}
    Let $S$ be a scheme, and let $G/S$ be an $S$-group scheme, so $G$ is flat and finitely presented over $S$ by convention. Suppose that $G$ acts on some $S$-scheme $X$, and set $\meX\coloneqq[X/G]$. Then, for any flat sheaf $\msF$ on $\meX$, there is a spectral sequence
    \[E_2^{ij}=\ghom^i(G,\ul{\fhom^j}(X/S,\msF_X))\implies\fhom^{i+j}(\meX,\msF).\]
\end{lemma}
\begin{proof}
    This is a special case of the `spectral sequence of the covering' $X\to\meX$; see \cite[2.4.26, 9.2.4, and/or 11.6.3]{olsson}. Writing $X_n\coloneqq X\by_\meX\dots\by_\meX X$ ($n+1$ factors), this is the spectral sequence
    \[E_1^{ij}=\fhom^j(X_i,\msF_i)\implies\fhom^{i+j}(\meX,\msF),\]
    where $\msF_j$ is simply the pullback of $\msF$ to $X_j$. Note that, for any $n\ge0$, $X_n\simeq G\by_S\by\dots\by_SG\by_SX$ (with $n$ copies of $G$), so the $E_1$-page of our spectral sequence is
    \[E_1^{ij}=\fhom^j(X_i,\msF_i)=\fhom^j(G^{\by i}_S\by_SX,\msF_i)=\Nat\p{G^{\by i}_S,\ul\hom^j(X/S,\msF_X)}\eqqcolon C^i(G,\ul\hom^j(X/S,\msF_X)).\]
    The differentials on this page match those used to define Hochschild cohomology, so $E_2^{ij}=\pi_i\p{E_1^{\bullet j}}=\ghom^i(G,\ul{\fhom^j}(X/S,\msF_X))$.
\end{proof}
\begin{rem}
    In the context of \cref{lem:descent-ss}, if $\msF$ is representable by a smooth group scheme (e.g. $\msF=\G_m$), then the descent spectral sequence equivalently takes place entirely within \'etale cohomology:
    \[E_2^{ij}=\ghom^i(G,\ul\hom^j(X/S,\msF_X))\implies\hom^{i+j}(\meX,\msF).\]
    We will make implicit use of this fact throughout the paper. An exact sequence of the above form (using \'etale cohomology) also holds if $G$ is smooth and $\msF$ is any \'etale sheaf.
\end{rem}
\begin{lemma}\label{lem:bar-complex-triv}
    Let $R$ be a strictly Henselian local ring, let $X$ be a finite $R$-scheme, and let $\msF$ be an \'etale sheaf on $X$. Then, for any finite $R$-scheme $Y$,
    \[\Nat(Y^{\by i}_R,\ul\hom^j(X/R,\msF))=0\]
    for any $i\ge0$ and $j\ge1$.
\end{lemma}
\begin{proof}
    By Yoneda, $\Nat(Y^{\by i}_R,\ul\hom^j(X/R,\msF))=\ul\hom^j(X/R,\msF)(Y^{\by i}_R)=\hom^j(X\by_R Y^{\by i}_R,\msF)$. Note that $X\by_RY^{\by i}_R$ is still finite over $R$ (by definition, $Y^{\by0}_R\coloneqq\spec R$) and so \cite[\href{https://stacks.math.columbia.edu/tag/03QJ}{Tag 03QJ}]{stacks-project} tells us that $X\by_RY^{\by i}_R$ is a finite disjoint union of $\spec$'s of strictly henselian local rings, and so $\hom^j(X\by_RY^{\by i}_R,\msF)=0$.
\end{proof}
\begin{cor}\label{cor:local-coh}
    Let $R$ be a strictly henselian local ring, let $G/R$ be a finite group scheme, and suppose $G$ acts on a finite $R$-scheme $X$. Then,
    \[\hom^n([X/G],\G_m)\iso\ghom^n(G,\G_{m,X/R})\tforany n\ge0.\]
\end{cor}
\begin{proof}
    Set $\meX\coloneqq[X/G]$. Consider the cohomological descent spectral sequence of \cref{lem:descent-ss}:
    \[E_2^{ij}=\ghom^i(G,\ul\hom^j(X/R,\G_m))\implies\hom^{i+j}(\meX,\G_m).\]
    It follows from \cref{lem:bar-complex-triv} that $E_2^{ij}=0$ whenever $j>0$. Thus, the spectral sequence is concentrated in the $j=0$ row, from which the corollary follows.
\end{proof}

\subsection{Tame stacks}

We review here the basics of tame stacks in the sense of \cite{AOV}. 
\begin{defn}
    Let $S$ be a scheme. A group scheme $G/S$ is \define{linear reductive} if it is flat and finitely presented over $S$, and the functor $\QCoh^G(S)\to\QCoh(S)$, $F\mapsto F^G$, from $G$-equivariant quasi-coherent sheaves to all quasi-coherent sheaves, is exact.
\end{defn}
\begin{defnprop}[{\cite[Definition 3.1 and Theorem 3.2]{AOV}}]\label{defn:tame}
    Let $S$ be a scheme. An algebraic stack $\meX/S$ is \define{tame} if it is locally of finite presentation and has finite inertia over $S$ -- so it admits a coarse moduli space $c\colon\meX\to X$ \cite{keel-mori,conrad-km} -- and either of the following equivalent conditions hold:
    \begin{itemize}
        \item The functor $\push c\colon\QCoh(\meX)\to\QCoh(X)$ is exact.
        \item For any object $\xi\in\meX(k)$ over an algebraically closed field $k$, the $k$-group scheme $\ul\Aut(\meX,\xi)$ is linearly reductive.
        \qedhere
    \end{itemize}
\end{defnprop}
Recall that a finite locally free commutative group scheme $G/S$ is \define{diagonalizable} if its Cartier dual $\dual G\coloneqq\ul\Hom(G,\G_m)$ is constant. In \cite[Section 2]{AOV}, the authors classify finite linearly reductive group schemes (i.e. the automorphism/stabilizer groups of tame stacks). We will implicitly make use of their classification throughout this paper:
\begin{prop}[{\cite[Theorem 2.19]{AOV}}]\label{prop:lin-red-classification}
    Let $S$ be a scheme, and let $G/S$ be a finite flat group scheme of finite presentation. Then, the following are equivalent
    \begin{alphabetize}
        \item $G$ is linearly reductive.
        \item There exists an an fpqc cover $\{S_i\to S\}_{i\in I}$ such that, for all $i$, the group scheme $G\by_SS_i$ is a semidirect product $\Delta\rtimes\ul Q$, where $\Delta$ is diagonalizable and $\ul Q$ is a tame constant group (i.e. $\#Q\in\Gamma(S_i,\msO_{S_i})^\by$).
        \item The fibers of $G\to S$ are linearly reductive.
    \end{alphabetize}
\end{prop}
\begin{cor}\label{cor:G-dual-etale}
    Let $S$ be a scheme, and let $G/S$ be a finite linearly reductive group. Then, $\dual G\coloneqq\ul\Hom_S(G,\G_m)$ is \'etale over $S$.
\end{cor}
\begin{proof}
    This can be checked fpqc-locally on $S$, so WLOG, we may and do assume that there is an exact sequence
    \[0\too\Delta\too G\too Q\too0,\]
    where $\Delta$ is diagonalizable and $Q$ is tame \'etale (i.e. $\#Q\in\Gamma(S,\msO_S)^\by$). Applying $\ul\Hom(-,\G_m)$, we arrive at the exact sequence
    \[0\too\dual Q\too\dual G\too\dual\Delta.\]
    Above, $\dual Q$ is once again tame \'etale, because $\#\dual Q=\#\pdual{\ab Q}=\#\ab Q$ divides $\#Q$ and $\dual\Delta$ is also \'etale, because $\Delta$ is diagonalizable. Thus, every subgroup of $\dual\Delta$ is \'etale, and so $\dual G$, being an extension of \'etale groups, is \'etale too.
\end{proof}
\begin{ex}
    Let $R$ be a strictly local ring. It follows from \cref{prop:lin-red-classification} that, for any finite linearly reductive group $G/R$, its connected--\'etale sequence \cite[Section (3.7)]{tate-ffgs} exhibits $G$ as an extension of a tame constant group scheme $\ul Q$ by a diagonalizable group scheme.
\end{ex}
\begin{lemma}\label{lem:tame-stack-etale-loc}
    Let $\meX$ be a tame algebraic stack over a scheme $S$, with coarse moduli space $c\colon\meX\to X$. Fix a point $x\in X$. Then, there exists all of the following data.
    \begin{itemize}
        \item A finite separable field extension $k/\kappa(x)$ and a point $x'\in\meX(k)$ lifting $x\in X$.
        \item An \'etale neighborhood $(U,u)\to(X,x)$ of $x$, with $u\in U(k)$.
        \item A finite linearly reductive group $G/U$ such that $G_u\cong\ul\Aut(\meX,x')$ as $k$-group schemes.
        \item A finite and finitely presented $U$-scheme $V/U$ equipped with a $G$-action and an isomorphism
        \[\meX\by_XU\cong[V/G]\]
    \end{itemize}
\end{lemma}
\begin{proof}
    This is a consequence of the proof, though not the statement, of \cite[Theorem 3.2(d)]{AOV}. First, \cite[Proposition 3.7]{AOV} guarantees the existence of $k$ and $x'\in\meX(k)$ as in the first bullet point. Now, \cite[\href{https://stacks.math.columbia.edu/tag/02LF}{Tag 02LF}]{stacks-project} guarantees the existence of an \'etale neighborhood $(U,u)\to(X,x)$ as in the second bullet point. By replacing $X$ with $U$ (and $\meX$ with $\meX\by_XU$), we may and do assume that $x$ lifts to some $x'\in\meX(\kappa(x))$. At this point, the rest of the lemma follows from \cite[\textsc{Step} 1 of the proof of Proposition 3.6]{AOV}; in particular, the third bullet point is guaranteed by their use of \cite[Proposition 2.18]{AOV} to obtain the linearly reductive group $G$ appearing in their argument.
\end{proof}
\begin{cor}\label{cor:tame-stack-sh}
    Let $\meX$ be a tame algebraic stack over a scheme $S$, with coarse moduli space $c\colon\meX\to X$. Let $\spec\msO_{X,\bar x}\to X$ be the (strictly henselian) local ring at some geometric point $\bar x\to X$. Then,
    \[\meX\by_X\msO_{X,\bar x}\cong[\spec R/G]\]
    for some strictly henselian local ring $R$ which is finite and finitely presented over $\msO_{X,\bar x}$ and is acted upon by some finite linearly reductive group $G/\msO_{X,\bar x}$ such that $G_{\bar x}\cong\ul\Aut(\meX,\bar x)$. Furthermore, $G_{\bar x}$ acts trivially on the residue field of $R$.
\end{cor}
\begin{proof}
    By \cref{lem:tame-stack-etale-loc}, by passing to an \'etale neighborhood of $\bar x$, we may and do assume that $\meX=[V/G]$ for some finite and finitely presented $V/X$ and some finite linearly reductive group $G/X$ such that $G_{\bar x}\cong\ul\Aut(\meX,\bar x)$. Write $\spec R\coloneqq V\by_X\msO_{X,\bar x}$. The claim will follow as soon as we show that $R$ is local (equivalently, $\spec R$ is connected). 

    Note that $R$ is a finite product of local rings, so $\spec R$ is connected if its special fiber $Y\coloneqq(\spec R)_{\bar x}$ is connected. Furthermore, $\meX\by_X\bar x\simeq[Y/G_{\bar x}]$ is connected because \cite[Corollary 3.3(a)]{AOV} shows its coarse space is $\spec\kappa(\bar x)$, so $G_{\bar x}$ must act transitively on the connected components of $Y$. After possibly extending the field of definition $\kappa(\bar x)$ of $\bar x$, we may lift it to a geometric point $\bar y\to Y$, and then observe that
    \[G_{\bar x}=\ul\Aut(\meX,\bar x)\simeq\Stab_{G_{\bar x}}(\bar y)\le G_{\bar x}\]
    from which we conclude that $G_{\bar x}$ acts trivially on $\bar y$. Since we saw earlier that $G_{\bar x}$ acts transitively on $Y$'s connected components, we deduce that $Y$ (and so also $\spec R$) must be connected.
\end{proof}

\subsection{Brauer groups}

We now set our conventions concerning Brauer groups. In particular, we differentiate between the (Azumaya) Brauer group $\Br$, the cohomological Brauer group $\Br'$, and $\hom^2(-,\G_m)$.
\begin{defn}
    Fix an algebraic stack $\meX$. An \define{Azumaya algebra} over $\meX$ is a quasi-coherent $\msO_\meX$-algebra $\msA$ which is \'etale-locally isomorphism to $M_n(\msO_\meX)$, the sheaf of $n\by n$-matrices over $\msO_\meX$, for some $n\ge1$. Two Azumaya algebras $\msA,\msB$ are \define{Brauer equivalent} if there exists vector bundles $\msE,\msE'$ on $\meX$ such that
    \[\msA\otimes\sEnd(\msE)\simeq\msB\otimes\sEnd(\msE').\]
    The \define{(Azumaya) Brauer group} $\Br(\meX)$ is the group of Brauer equivalence classes of Azumaya algebras, under tensor products.
\end{defn}
\begin{defn}
    The \define{cohomological Brauer group} of an algebraic stack $\meX$ is $\Br'(\meX)\coloneqq\hom^2(\meX,\G_m)_{\t{tors}}$.
\end{defn}
\begin{rem}
    For any algebraic stack $\meX$, there always exists a natural injection
    \[\alpha_\meX\colon\Br\meX\into\Br'\meX;\]
    see \cite[(2.1)]{Gr1} and/or \cite[(2.2.1)]{shin}.
\end{rem}

For comparing the various groups $\Br\meX,\Br'\meX,\hom^2(\meX,\G_m)$, one has the following results.
\begin{lemma}\label{lem:alpha-finite-cover}
    Let $f\colon\meX\to\meY$ be a finite, flat, finitely presented surjective morphism of algebraic stacks. Then,
    \begin{enumerate}
        \item a class $\beta\in\hom^2(\meY,\G_m)$ is in the image of $\alpha_\meY$ if and only if $\pull f\beta\in\hom^2(\meX,\G_m)$ is in the image of $\alpha_\meX$.
        \item if $\alpha_\meX$ is an isomorphism, so is $\alpha_\meY$.
        \item if $\Br'(\meX)=\hom^2(\meX,\G_m)$, so too does $\Br'(\meY)=\hom^2(\meY,\G_m)$.
    \end{enumerate}
\end{lemma}
\begin{proof}
    Parts \bp{1,2} are \cite[Proposition 2.5 + Corollary 2.6]{shin}. For part \bp3, the norm map $\push f\G_{m,X}\to\G_{m,Y}$ gives a morphism $\hom^2(X,\G_m)=\hom^2(Y,\push f\G_{m,X})\to\hom^2(Y,\G_m)$ such that the composition
    \[\hom^2(Y,\G_m)\xtoo{\pull f}\hom^2(X,\G_m)\too\hom^2(Y,\G_m)\]
    is multiplication by $\deg f$. Hence, if the middle group is torsion, the outer group must be as well.
\end{proof}
\begin{thm}[Gabber, {\cite[Theorem 4.2.1]{CT-S}}]\label{thm:gabber}
    Let $X$ be a (quasi-compact, separated) scheme which admits an ample line bundle. Then, $\alpha_X$ is an isomorphism.
\end{thm}

At times, we will implicitly use the following facts about Brauer groups.
\begin{lemma}\label{lem:Br=Br'-stacky-curves}
    Let $\meX$ be a regular stacky curve over a field $k$. Then, $\Br\meX\iso\hom^2(\meX,\G_m)$.
\end{lemma}
\begin{proof}
    Let $c\colon\meX\to X$ be its coarse space map; note that $X$ is normal (and so $k$-smooth) because $\meX$ is. By \cite[Theorem 2.7]{Br-quot}, there exists a finite surjection $Z\to\meX$ from a (not necessarily separated) scheme $Z$. We may replace $Z$ with its normalization and then by any one of its (connected) components which dominates $X$ in order to assume that it is normal and irreducible. The composition $Z\to\meX\xto cX$, being both proper and quasi-finite, is therefore finite, so $Z$ is a $k$-curve. Since $Z$ was assumed normal, it is in fact regular. Now, $Z\to\meX$ is a finite morphism from a regular scheme to a regular algebraic stack of the same dimension. It follows from \cite[Theorem 23.1]{matsumura} (applied to $Z_U\to U$ for $U\to\meX$ a smooth cover by a scheme) that $Z\to\meX$ must be flat as well. Finally, $Z$ admits an ample line bundle \cite[\href{https://stacks.math.columbia.edu/tag/09NZ}{Tag 09NZ}]{stacks-project} and $\Br'(Z)=\hom^2(Z,\G_m)$ \cite[Proposition 6.6.7]{poonen-rat-pts}, so we conclude by \cref{thm:gabber} and \cref{lem:alpha-finite-cover}.
\end{proof}
\begin{prop}\label{prop:brauer-injects}
    Let $\meX$ be a regular, noetherian algebraic stack, and let $\meU\subset\meX$ be a dense open substack. Then, the restriction map $\hom^2(\meX,\G_m)\to\hom^2(\meU,\G_m)$ is injective.
\end{prop}
\begin{proof}
    One can argue complete analogously to \cite[Lemma 3.1.3.3]{lieblich-twisted}. For the sake of completeness, we include details below.
    
    Let $\meG\to\meX$ be a $\G_m$-gerbe, and suppose that $\meG_\meU$ is trivial. Then, there exists a $1$-twisted line bundle $\msL_\meU$ on $\meG_\meU$ \cite[12.3.10]{olsson}. Let $i\colon\meG_\meU\into\meG$ be the natural inclusion. Then, $\push i\msL_\meU$ is a quasi-coherent 1-twisted sheaf on $\meG$. Since any subsheaf of a 1-twisted sheaf is 1-twisted, \cite[Proposition 15.4]{champs-algebriques} shows that $\push i\msL_\meU$ is a colimit of its coherent 1-twisted subsheaves. Consequently, there exists some coherent 1-twisted sheaf $\msL$ on $\meG$ extending $\msL_\meU$; replacing $\msL$ by $\pdual{\dual\msL}$, if necessary, we may and do suppose that $\msL$ is reflexive. Now, any reflexive sheaf of rank 1 on a regular algebraic stack is invertible (follows from \cite[Proposition 1.9]{hart-stable-reflex}), so $\msL$ is a 1-twisted line bundle. Since $\meG$ supports a 1-twisted line bundle, it follows from \cite[Proposition 12.3.11]{olsson} that $[\meG]=0\in\hom^2(\meX,\G_m)$.
\end{proof}
\begin{cor}
    Let $\meX$ be a regular, noetherian algebraic stack with coarse moduli space $c\colon\meX\to X$. If $c$ is an isomorphism over an open $U\subset X$ such that $Z\coloneqq X\sm U$ has codimension $\ge2$ in $X$, then $\pull c\colon\Br X\to\Br\meX$ is an isomorphism.
\end{cor}
\begin{proof}
    Let $U,Z$ be as in the proposition statement, and consider the commutative diagram
    \[\begin{tikzcd}
        \Br X\ar[r, "\pull c"]\ar[d, "\sim" sloped]&\Br\meX\ar[d, hook]\\
        \Br U\ar[r, "\sim"]&\Br\meX_U,
    \end{tikzcd}\]
    whose left arrow is an isomorphism because $\codim_XZ\ge2$ (see \cite[Theorem 1.1]{cesnavicius-brauer} and/or \cite[Theorem 3.7.6]{CT-S}) and whose bottom arrow is an isomorphism because $\meX_U\iso U$. The claim follows from commutativity of this diagram.
\end{proof}
\begin{prop}\label{prop:iso-away-from-aut}
    Let $\meX$ be either
    \begin{itemize}
        \item a tame algebraic stack; or
        \item a separated DM stack.
    \end{itemize} 
    In either case, let $c\colon\meX\to X$ denote its coarse space. Let $\l$ be a prime such that $\l\nmid\#\Aut(\meX,x)$ for every $x\in\abs\meX$. Then, for any $i\ge1$, the pullback map
    \[\pull c\colon\hom^i(X,\G_m)\otimes\zloc\l\too\hom^i(\meX,\G_m)\otimes\zloc\l\]
    is an isomorphism. In particular, $\Br'(X)\{\l\}\iso\Br'(\meX)\{\l\}$.
\end{prop}
\begin{proof}
    By considering the Leray spectral sequence, we see it suffices to show that $\homR^i\push c\G_m\otimes\zloc\l=0$ for all $i\ge1$. Since this can be verified on the level of stalks, appealing to the local structure theorem for tame stacks (see \cref{cor:tame-stack-sh}) or for separated DM stacks (see \cite[Th\'eor\`eme 6.2]{champs-algebriques} and/or \cite[Theorem 2.12]{olsson-hom-stacks}), we may assume that $X=\spec R$ is strictly local ant that $\meX=[\spec A/G]$ for some finite $R$-algebra $A$ and some finite $R$-group $G$ for which $\l\nmid\#G$. In this case, it suffices to compute that $\hom^i(\meX,\G_m)\otimes\zloc\l=0$ for all $i\ge1$. As in the beginning of the proof of \cref{prop:diag-quot-Br-vanish}, the existence of the $G$-torsor $\spec A\to\meX$ with $\hom^i(A,\G_m)=0$ for all $i\ge1$ (because $A$ is a product of strictly local rings) suffices to deduce that $\hom^i(\meX,\G_m)$ is $\#G$-torsion for all $i\ge1$. Since $\l\nmid\#G$, we conclude that $\hom^i(\meX,\G_m)\otimes\zloc\l=0$ for all $i\ge1$.
\end{proof}
\begin{rem}
    If $\meX$ in \cref{prop:iso-away-from-aut} is DM, then the conclusion
    \[\pull c\colon\hom^i(X,\push c\msF)\otimes\zloc\l\isoo\hom^i(\meX,\msF)\otimes\zloc\l\tforall i\ge1\]
    holds for any \'etale sheaf $\msF$. Indeed, once one reduces to the case of $\meX=[\spec A/G]$ with $G$ an abstract group (because $\meX$ is DM) of cardinality not divisible by $\l$, the Hochschild--Serre/descent spectral sequence
    \[E_2^{ij}=\ghom^i(G,\hom^j(A,\msF))\implies\hom^{i+j}(\meX,\msF)\]
    shows that $\hom^i(\meX,\msF)$ is $\#G$-torsion for all $i\ge1$ (so $\hom^i(\meX,\msF)\otimes\zloc\l=0$).

    To get this more general conclusion in the case of algebraic $\meX$ in \cref{prop:iso-away-from-aut}, it would suffice have a positive answer to \cref{qn:rand}, at least for linearly reductive $G$.
\end{rem}
\begin{qn}\label{qn:rand}
    Let $S$ be a base scheme, and let $G$ be a finite, flat, finitely presented $S$-group scheme. For a $G$-module $M$, is it necessarily the case that $\ghom^i(G,M)$ is $\#G$-torsion for all $i\ge1$?
\end{qn}

\section{\bf Brauer groups of classifying stacks for locally diagonalizable groups}\label{sect:Blocally-diagonalizable}
In this section, we compute the Brauer (and Picard) groups of classifying stacks for certain finite group schemes.
\begin{rec}
    A finite, flat, finitely presented commutative group scheme $G/S$ is \define{locally diagonalizable} if its Cartier dual $\dual G\coloneqq\ul\Hom(G,\G_m)$ is $S$-\'etale.
\end{rec}
\begin{defn}
    Let $G/S$ be an \'etale group scheme. Call $G$ \define{cyclic} if, \'etale-locally on $S$, it is isomorphic to $\ul{\zmod n}$ for some $n\in\Z$. Call a locally diagonalizable group $G/S$ \define{cyclic} if $\dual G$ is cyclic.
\end{defn}
\begin{defn}
    We say a finite group scheme $G/S$ is \define{relatively connected} if, for every geometric point $\bar s\to S$, the scheme $G_{\msO_{S,\bar s}}$ is connected.
\end{defn}
\begin{ex}
    $\mu_n$ (for any $n$) is a cyclic locally diagonalizable group. Any commutative group $G/S$ of order invertible on $S$ is locally diagonalizable.
\end{ex}
\begin{lemma}\label{lem:mun-extension-comm}
    Let $R$ be a strictly henselian local ring, fix some $n\ge1$, and let
    \[1\too\G_m\too E\too\mu_n\too1\]
    be an extension of abelian fppf sheaves on $R$. Then, $E$ is commutative and represented by a scheme.
\end{lemma}
\begin{proof}
    It is clear that $E$ is a $\G_m$-torsor over $\mu_n$, so effective descent for affine schemes shows that $E$ is represented by a scheme. Hence, we only need to show that $E$ is commutative, which one can do by arguing as in \cite{73470}. The commutator pairing $\mu_n\by\mu_n\to\G_m$ gives rise to a map $\phi\colon\mu_n\to\ul\Hom(\mu_n,\G_m)$. Since the target is \'etale, $\phi$ uniquely factors through the maximal \'etale quotient $\mu_n^\tet$ of $\mu_n$. Note that $\#\mu_m$ is coprime to $\#\mu_n^\circ$, the order of the identity component of $\mu_n$. For any $\zeta\in\mu_n$, $\phi(\zeta)\in\ul\Hom(\mu_n,\G_m)$ is killed by $\#\mu_n^\tet$ (since $\phi$ factors through $\mu_n^\tet$), so $\phi(\zeta)$ must kill $\mu_n^\circ$; that is, $\phi$ factors as
    \[\phi\colon\mu_n\onto\mu_n^\tet\to\ul\Hom(\mu_n^\tet,\G_m)\into\ul\Hom(\mu_n,\G_m).\]
    In other words, we may assume that $\mu_n=\mu_n^\tet$ is \'etale. Now, $\phi\colon\mu_n\to\ul\Hom(\mu_n,\G_m)$ comes from the commutator pairing $\mu_n\by\mu_n\to\G_m$, so for any $R$-scheme $S$ and any $\zeta\in\mu_n(S)$, we have $\phi(\zeta)(\zeta)=1$. Since $\mu_n$ is \'etale, $\mu_n(S)$ is cyclic for any $S/R$, so by choosing $\zeta\in\mu_n(S)$ to be a generator, we conclude that $\phi=0$.
\end{proof}
\begin{lemma}\label{lem:Ext-G-Gm}
    Let $G/S$ be a finite locally free commutative group over a scheme $S$, and let $\meX/S$ be an algebraic stack. Then, there are canonical isomorphisms
    \[\fhom^n(\meX,\dual G)\iso\Ext^n_\meX(G,\G_m),\]
    where the latter group is $\Ext$ in the category of abelian sheaves on the big fppf site $\fppf\meX$.
\end{lemma}
\begin{proof}
    This was proven by Shatz \cite{Shatz-PHS} (and independently by Waterhouse \cite{waterhouse-PHS} in the case $n=1$) when $\meX$ is a scheme, and the general case follows from his argument. In brief, one has the local global spectral sequence $\fhom^i(\meX,\sExt^j_\meX(G,\G_m))\implies\Ext^{i+j}_\meX(G,\G_m)$. If $X\to\meX$ is a smooth cover by a scheme, then \cite{Shatz-PHS} shows that $\sExt^j_X(G,\G_m)=\sExt^j_\meX(G,\G_m)\vert_X$ vanishes for all $j>0$, so also $\sExt^j_\meX(G,\G_m)=0$. Thus, the sequence immediately collapses and gives isomorphisms
    \[\fhom^n(\meX,\dual G)=\fhom^n(\meX,\sHom_\meX(G,\G_m))\iso\Ext^n_\meX(G,\G_m),\]
    as claimed.
\end{proof}
\begin{lemma}\label{lem:BG-Gm-stalk}
    Let $R$ be a strictly henselian local ring, and let $G/R$ be a finite group. Then,
    \begin{align*}
        \hom^0(BG_R,\G_m) &= \G_m(R) \\
        \hom^1(BG_R,\G_m) &= \dual G(R). \\
        \intertext{If, furthermore, $G$ is locally diagonalizable and either cyclic or connected, then also}
        \hom^2(BG_R,\G_m) &= 0.
    \end{align*}
\end{lemma}
\begin{proof}
    The first equality holds because $\spec R$ is $BG_R$'s coarse space. For the second, \cref{cor:local-coh} shows that $\hom^1(BG_R,\G_m)\simeq\ghom^1(G,\G_m)=\Hom_{\GrpSch_R}(G,\G_m)$. Now assume that $G$ is locally diagonalizable and either cyclic or connected. \cref{cor:local-coh} shows that $\hom^2(BG_R,\G_m)\simeq\ghom^2(G,\G_m)$ which classifies central extensions
    \[1\too\G_m\too E\too G\too1.\]
    For any such extension, $E$ is commutative either by \cref{lem:mun-extension-comm} if $G$ is cyclic or by \cite[Theorem 15.39]{milne-alg-grps} (see also \cite[Lemma 2.14]{AOV}) if $G$ is connected. Thus, $\ghom^2(G,\G_m)\simeq\Ext^1(G,\G_m)$ ($\Ext$ here is computed in the category of abelian fppf sheaves on $\spec R$). Finally, \cref{lem:Ext-G-Gm} shows that $\Ext^1(G,\G_m)\simeq\fhom^1(R,\dual G)$, but because $\dual G$ is \'etale, $\fhom^1(R,\dual G)\simeq\hom^1(R,\dual G)=0$.
\end{proof}
\begin{prop}\label{prop:BG-pushforwards}
    Let $G/S$ be a finite group scheme, and let $c\colon BG_S\to S$ be the structure map. Then,
    \[\push c\G_m\simeq\G_m\tand\homR^1\push c\G_m\simeq\dual G.\]
    If, furthermore, $G$ is locally diagonalizable and either cyclic or relatively connected, then also $\homR^2\push c\G_m=0$.
\end{prop}
\begin{proof}
    The first of these holds because $c$ is a coarse space map, the second holds because \cref{lem:BG-Gm-stalk} shows that the natural map $\homR^1\push c\G_m\to\dual G$ (induced from the maps $\hom^1(BG_T,\G_m)\to\dual G(T)$, for any scheme $T/S$, sending a $G$-equivariant line bundle on $T$ to the character corresponding to its $G$-action) is an isomorphism on stalks, and the last holds because \cref{lem:BG-Gm-stalk} shows that the stalks of $\homR^2\push c\G_m$ all vanish in this case.
\end{proof}
\begin{prop}\label{prop:BG-Gm-coh}
    Let $G/S$ be a finite locally diagonalizable group over some base scheme $S$, and consider the classifying stack $c\colon BG_S\to S$. Assume that $G$ is either cyclic or connected. Then, there are split exact sequences
    \begin{alignat}{10}
        0 \too& \centermathcell{\Pic(S)} &&\xtoo{\pull c}&& \centermathcell{\Pic(BG_S)} &&\too&& \centermathcell{\dual G(S)} &&\too&& 0\\
        0 \too& \hom^2(S,\G_m) &&\xtoo{\pull c}&& \hom^2(BG_S,\G_m) &&\too&& \hom^1(S,\dual G) &&\too\,&& 0 \label{ses:H2-BG}
        .
    \end{alignat}
    The latter sequence restricts to analogous split exact sequences when $\hom^2(-,\G_m)$ is replaced by $\Br'$ or $\Br$. 
\end{prop}
\begin{proof}
    Consider the Leray spectral sequence $E_2^{ij}=\hom^i(S,\homR^j\push\pi\G_m)\implies\hom^{i+j}(BG_S,\G_m)$, pictured in \cref{fig:leray-BG} with terms computed using \cref{prop:BG-pushforwards}.
    \begin{figure}[h]
        \centering
        \[\begin{tikzcd}
            0\\
            \dual G(S)\ar[rrd] & \hom^1(S, \dual G)\ar[rrd] \\
            \G_m(S) & \Pic(S) & \hom^2(S,\G_m) & \hom^3(S,\G_m)
        \end{tikzcd}\]
        \caption{The $E_2$-page of the Leray spectral sequence for $BG_S\to S$.}
        \label{fig:leray-BG}
    \end{figure}
    Let $\pi\colon S\to BG_S$ be the section of $c$ given by the universal torsor. Then, $\pull\pi$ splits all the edge maps $\hom^i(S,\G_m)\to\hom^i(BG_S,\G_m)$, so the displayed differentials in \cref{fig:leray-BG} must both vanish. Thus, the spectral sequence gives rise to the claimed (split) exact sequences. Because \cref{ses:H2-BG} is split, it remains exact when passing to torsion subgroups (so remains split exact with $\hom^2(-,\G_m)$ replaced by $\Br'$), and by \cref{lem:alpha-finite-cover}, for any $\alpha\in\Br'(BG_S)$, we have $\pull\pi\alpha\in\Br(S)\iff\alpha\in\Br(BG_S)$, so \cref{ses:H2-BG} remains split exact with $\hom^2(-,\G_m)$ replaced by $\Br$.
\end{proof}
\begin{rem}
    In the case that $G=\mu_n$ and $S=\spec L$ for a field $L$ of characteristic not dividing $n$, $\hom^2(B\mu_{n,L},\G_m)$ was computed earlier by Lieblich \cite[Proposition 4.1.4]{lieblich2011period}.
\end{rem}
\begin{rem}
    The conclusion of \cref{prop:BG-Gm-coh} -- i.e. that $\hom^2(BG_S,\G_m)\cong\hom^2(S,\G_m)\oplus\fhom^1(S,\dual G)$ -- holds for more finite group schemes than its statement accounts for. For example, \cite[Proposition 3.2]{AM} proves the same statement when $G=\ul{\zmod n}$ is a constant, cyclic group scheme. Note that, in this case, $G$ is \important{not} locally diagonalizable if $n$ is not invertible on the base. As another example, if $k$ is a separably closed field of characteristic $p$, then
    \[\hom^2(k,\G_m)\oplus\fhom^1(k,\dual\alpha_p)\cong\fhom^1(k,\alpha_p)\cong k/k^p\overset*\cong\ghom^2(\alpha_p,\G_m)\overset{\cref{cor:local-coh}}\cong\hom^2(B\alpha_{p,k},\G_m),\]
    where the isomorphism labelled with $*$ holds by \cite[Propositions 15.35 and 15.36]{milne-alg-grps}. Finally, \cref{rem:BG-Gm-coh} generalizes \cref{prop:BG-Gm-coh} to certain additional finite linearly reductive groups $G$.

    However, if $G=\mu_\l\by\mu_\l$, for some prime $\l$, over a field $k$ with $\Char k\neq\l$, then this statement fails for $G$ \cite[Proposition 4.3.2]{lieblich2011period}. Additionally, if $G$ is an abelian variety over a field, then the statement generally fails \cite[Lemma 4.2]{Shin-Br=Br'}.
\end{rem}
Next, we construct a right-splitting of \cref{ses:H2-BG}, with an argument inspired by \cite[Proof of Theorem 1.4]{descent-open}.

\subsection{Constructing a right-splitting of \cref{ses:H2-BG}}\label{sect:split-BG}
\begin{set}
    Let $G/S$ be a finite locally diagonalizable group over some base scheme $S$, and assume that $G$ is cyclic or connected. Let $f\colon BG_S\to S$ denote its classifying stack, and let $\sigma\in\fhom^1(BG_S,G)$ denote the class of a universal $G$-torsor $\pi\colon S\to BG_S$.
\end{set}
\begin{rem}
    Above, we refer to $\sigma$ as the class of ``a'' universal $G$-torsor instead of ``the'' universal $G$-torsor because, in general, $BG_S$ can support multiple non-isomorphic universal $G$-torsors. For example, if $\sigma$ is a universal $G$-torsor, then so is $\sigma\cdot\pull f(\tau)$ for any $\tau\in\hom^1(S,G)$ as is $\push\lambda(\sigma)$ for any $\lambda\in\Aut_S(G)$.
\end{rem}
Our goal in this section is to show that the map
\[\mapdesc{}{\hom^1(S,\dual G)}{\hom^2(BG_S,\G_m)}\alpha{\sigma\smile\pull f\alpha}\]
is a right-splitting to \cref{ses:H2-BG}.

In the remainder of this section, we work in fppf topology instead of the \'etale topology since we will need to cohomology with coefficients in $G$. Let $\msC\coloneqq(\tau_{\le2}\fhomR\push f\G_m)[1]$ be the indicated shift of the indicated truncation of the derived pushforward of $\G_m$. Note that, by \cref{prop:BG-pushforwards}, $\msC\simeq(\tau_{\le1}\fhomR\push f\G_m)[1]$ and sits in a distinguished triangle
\begin{equation}\label{tri:BG}
    \G_m[1]\too \msC\too\dual G\xtoo{+1}
    .
\end{equation}
Furthermore, the above distinguished triangle \cref{tri:BG} is split; indeed, the map $\G_m[1]\to(\tau_{\le2}\fhomR\push f\G_m)[1]=\msC$ is split by (the shift of the derived pushforward) of the map $\G_{m,BG_S}\to\fhomR\push\pi\G_{m,S}$ coming from the section $\pi$ of $f$. Now, applying $\Hom_S(\dual G,-)$ to \cref{tri:BG}, we obtain a short exact sequence
\[0\too\Hom_S(\dual G,\G_m[1])\too\Hom_S(\dual G,\msC)\xtoo\chi\Hom_S(\dual G,\dual G)\too0.\]
Using that, in general, $\Hom(\msA,\msB[n])=\Ext^n(\msA,\msB)$ along with \cref{lem:Ext-G-Gm} allows us to see that $\Hom(\dual G,\G_m[1])\simeq\Ext^1(\dual G,\G_m)\simeq\fhom^1(S,G)$. Similarly,
\begin{align}\label{eqn:HomC=H1G}
    \Hom_S(\dual G,\msC)
    &= \Hom_S\p{\dual G,(\tau_{\le2}\fhomR\push f\G_m)[1]} \nonumber\\
    &\simeq \Ext^1_S(\dual G,\tau_{\le2}\fhomR\push f\G_m) \nonumber\\
    &\simeq \Ext^1_S(\dual G,\fhomR\push f\G_m) \nonumber\\
    &\simeq \Ext^1_{BG_S}(\dual G,\G_m) \nonumber\\
    &\simeq \fhom^1(BG_S, G)
    .
\end{align}
Thus, our earlier short exact sequence can be rewritten as
\begin{equation}\label{ses:H1-BG-G}
    0\too\fhom^1(S,G)\xtoo{\pull f}\fhom^1(BG_S,G)\xtoo{\chi'}\Hom_S(\dual G,\dual G)\too0
    .
\end{equation}
\begin{lemma}\label{lem:univ-id}
    Let $\tau\in\fhom^1(BG_S,G)$ be a class mapping to the identity $\id_{\dual G}\in\Hom_S(\dual G,\dual G)$ under the right map in \cref{ses:H1-BG-G}. Then, $\tau$ is (the class of) a universal $G$-torsor.
\end{lemma}
\begin{proof}
    Consider the commutative diagram
    \begin{equation}\label{diag:big-Hom-BG}
        \begin{tikzcd}
            \Hom_S(G,G)\by\fhom^1(BG_S,G)\ar[r, "\id\by\chi'"]\ar[d, "\sim" sloped]\ar[ddd, bend right = 90, "{(\lambda,\alpha)\mapsto\push\lambda(\alpha)}"']&\Hom_S(G,G)\by\Hom_S(\dual G,\dual G)\ar[d, equals]\\
            \Hom_S(G,G)\by\Hom_S(\dual G,\msC)\ar[r, "\id\by\chi"]\ar[d, "{(\lambda,\phi)\mapsto\phi\circ\dual\lambda}"']&\Hom(G,G)\by\Hom(\dual G,\dual G)\ar[d, "{(\lambda,\psi)\mapsto\psi\circ\dual\lambda}"]\\
            \Hom_S(\dual G,\msC)\ar[d, "\sim" sloped]\ar[r, "\chi"]&\Hom_S(\dual G,\dual G)\ar[d, equals]\\
            \fhom^1(BG_S,G)\ar[r, "\chi'"]&\Hom_S(\dual G,\dual G)
            .
        \end{tikzcd}
    \end{equation}
    Let $\lambda\coloneqq\dual{\chi'(\sigma)}\in\Hom_S(G,G)$. Then commutativity of \cref{diag:big-Hom-BG} implies that $\chi'(\push\lambda(\tau))=\id_G\circ\dual\lambda=\chi'(\sigma)$. By \cref{ses:H1-BG-G}, possibly after twisting $\tau$ by an element of $\pull f\fhom^1(S,G)$, we may assume wlog that $\push\lambda(\tau)=\sigma$. Note that $\lambda$ may also be regarded as an $S$-morphism $\lambda\colon BG_S\to BG_S$; indeed this is the morphism sending a $G$-torsor $T$ over an $S$-scheme $X$ to the $G$-torsor $\push\lambda(T)$ over $X$. Let $g_\tau\colon BG_S\to BG_S$ denote the unique (up to isomorphism) $S$-morphism such that $\pull g_\tau(\sigma)=\tau$, and define $h$ as the composition $h\colon BG_S\xto{g_\tau}BG_S\xto\lambda BG_S$.
    Note that, by definition of $\lambda$, we have 
    \[\pull h(\sigma)=\push\lambda\p{\pull g_\tau(\sigma)}=\push\lambda(\tau)=\sigma.\]
    It follows from this that $h\simeq\id_{BG_S}$, so $g_\tau$ is an isomorphism. Thus, by transfer of structure, $\tau=\pull g_\tau(\sigma)$ must be universal as well.
\end{proof}
\begin{prop}\label{prop:BG-Br-split}
    There exists a choice $\sigma\in\fhom^1(BG_S,G)$ of class of universal $G$-torsor such that the map
    \[\mapdesc s{\hom^1(S,\dual G)}{\hom^2(BG_S,\G_m)}\alpha{\pull f\alpha\smile\sigma}\]
    is a right-splitting to \cref{ses:H2-BG}.
\end{prop}
\begin{proof}
    By \cref{lem:univ-id}, we may and do choose $\sigma$ so that $\chi'(\sigma)=\id_{\dual G}$. Let $\Sigma\in\Hom_S(\dual G,\msC)$ be the morphism corresponding to $\sigma$ under the identification \cref{eqn:HomC=H1G}. Note that the short exact sequence
    \[0\too\hom^2(S,\G_m)\too\hom^2(BG_S,\G_m)\xtoo r\hom^1(S,\dual G)\too0\]
    obtained from applying $\hom^1(S,-)$ to the split distinguished triangle \cref{tri:BG} (note $\hom^1(S,\msC)\simeq\hom^2(BG_S,\G_m)$ by definition of $\msC$) is the sequence \cref{ses:H2-BG} obtained from the Leray spectral sequence. Write $\rho\colon\msC\to\dual G$ for the map in \cref{tri:BG}, and consider the commutative diagram (commutativity follows from \cite[Proposition V.1.20]{milne-et})
    \[\begin{tikzcd}[column sep=large]
        \hom^1(S,\dual G)\by\Hom_S(\dual G,\msC)\ar[d, "\sim" sloped]\ar[r]&\fhom^1(S,\msC)\ar[r, "\push\rho"]\ar[d, "\sim" sloped]&\hom^1(S,\dual G)\ar[d, equals]\\
        \hom^1(S,\dual G)\by\fhom^1(BG_S,G)\ar[r, "\pull f(-)\smile(-)"]&\hom^2(BG_S,\G_m)\ar[r, "r"]&\hom^1(S,\dual G)
        .
    \end{tikzcd}\]
    Commutativity of this diagram shows that
    \[r(s(\alpha))=r\p{\pull f(\alpha)\smile\sigma}=\push\rho(\push\Sigma(\alpha))=\ppush{\rho\circ\Sigma}(\alpha)=\push{\chi(\Sigma)}(\alpha)=\ppush{\id_{\dual G}}(\alpha)=\alpha,\]
    where the third- and second-to-last equalities hold by definition of $\chi$ and our assumption on $\sigma$.
\end{proof}

\section{\bf A cohomological vanishing result}\label{sect:vanishing-result}

In this section, we aim to prove a version of \cite[Theorem 6]{meier-cs} which holds for tame algebraic stacks (see \cref{thm:R2Gm-vanishes}). As an immediate application, we then use it to compute the Brauer group of the moduli stack $\meX(1)$ of generalized elliptic curves over any noetherian strictly Henselian local ring (see \cref{ex:X(1)R}). 

To prove \cref{thm:R2Gm-vanishes}, one needs to guarantee the vanishing of groups of the form $\hom^2([\spec R/G],\G_m)$, where $R$ is a strictly henselian local ring and $G$ is a finite linearly reductive group. Since such $G$ are built from tame \'etale groups and connected locally diagonalizable groups, the first steps are to understand such cohomology groups when $G$ falls into one of these categories.
\begin{lemma}\label{lem:abs-grp-loc-Gm-coh}
    Let $R$ be a strictly henselian local ring with residue field $k$, and let $G$ be a finite abstract group acting on $R$. Assume that $p\coloneqq\Char k\nmid\#G$ and set $\meX\coloneqq[\spec R/G]$. Then, $\hom^n(\meX,\G_m)\simeq\ghom^n(G,\G_m(R))$ for all $n\ge1$ and $\ghom^n(G,\G_m(R))\simeq\ghom^n(G,\G_m(k))$ for all $n\ge2$. This latter isomorphism holds for $n=1$ as well if $R$ has no nontrivial $p$-power roots of unity.
\end{lemma}
\begin{proof}
    This follows from the proof of \cite[Lemma 9]{meier-cs} (which assumed that $R$ was a domain, but did not use this assumption in its proof) coupled with the observation that $\hom^n(\meX,\G_m)=\ghom^n(G,\G_m(R))$ is $\#G$-torsion. To assure the reader that $R$ need not be a domain, we include a full proof below.

    First note that \cref{cor:local-coh} shows that $\hom^n(\meX,\G_m)\simeq\ghom^n(G,\G_m(R))$ for all $n\ge0$. Now, let $M\coloneqq\ker\p{\G_m(R)\to\G_m(k)}$. Because $R$ is strictly Henselian, Hensel's lemma shows that $M$ is a $\zloc p$-module. Thus, $M[1/p]$ (if $p=0$, by this, we simply mean $M$) is a $\Q$-vector space, so taking cohomology of the exact sequence $0\to M[1/p]\to\G_m(R)[1/p]\to\G_m(k)[1/p]\to0$ shows that
    \[\ghom^n(G,\G_m(R)[1/p])\iso\ghom^n(G,\G_m(k)[1/p])\tforall n\ge1.\]
    
    Finally, because $p\nmid\#G$, for any $G$-module $N$, we have $\ghom^n(G,N)\iso\ghom^n(G,N[1/p])$ for all $n\ge2$ (because the kernel and cokernel of $N\to N[1/p]$ are both $p$-power torsion) and furthermore $\ghom^1(G,N)\iso\ghom^1(G,N[1/p])$ if $N\{p\}=0$ (because then $N\to N[1/p]$ is injective).
\end{proof}

\subsection{Cohomology of quotients by locally diagonalizable groups}
\begin{lemma}\label{lem:tame-red-Gm}
    Let $\meX$ be a noetherian tame algebraic stack with coarse space morphism $c\colon\meX\to X$. Assume that $X$ is an affine scheme. Then, for every $i\ge1$, the natural map
    \[\hom^i(\meX,\G_m)\too\hom^i(\red\meX,\G_m)\]
    is an isomorphism.
\end{lemma}
\begin{proof}
    We can factor $\red\meX\into\meX$ into a sequence
    \[\red\meX=\meX_0\into\meX_1\into\dots\into\meX_n=\meX\]
    of square-zero thickenings, so there are quasi-coherent ideal sheaves $\msI_i\subset\msO_{\meX_i}$, for $i=1,\dots,n$, such that $\msI_i^2=0$ and $\msO_{\meX_i}/\msI_i\simeq\msO_{\meX_{i-1}}$. For each $i$, the map $\msI_i\to\G_{m,\meX_i}$, $x\mapsto1+x$, sits in the following exact sequence
    \begin{equation}\label{ses:nilpotent-Gm}
        0\too\msI_i\too\G_{m,\meX_i}\too\push\iota\G_{m,\meX_{i-1}}\too1  
    \end{equation}
    of sheaves on $\liset{\p{\meX_i}}$, where $\iota\colon\meX_{i-1}\into\meX_i$ is the natural immersion. Now, by considering the long exact sequence in cohomology associated to \cref{ses:nilpotent-Gm} and inducting on $i$, one sees that, to prove the lemma, it suffices to show that $\hom^j(\meX_i,\msI_i)=0$ for all $j\ge1$ and all $i=1,\dots,n$. For fixed $i,j$, let $r\colon\meX_i\into\meX$ be the natural closed immersion. Because $r$ is a closed immersion $\push r$ is exact. At the same time, because $\meX$ is tame, $\push c$ is exact on qcoh sheaves (recall \cref{defn:tame}), so
    \[\hom^j(\meX_i,\msI_i)\simeq\hom^j(\meX,\push r\msI_i)\simeq\hom^j(X,\push c\push r\msI_i),\]
    but this latter group vanishes because $X$ is affine (and $j\ge1$).
\end{proof}
\begin{prop}\label{prop:diag-quot-Br-vanish}
    Let $R$ be a strictly henselian noetherian local ring with residue field $k$, let $\Delta/R$ be a connected diagonalizable group, suppose $\Delta$ acts on some finite, connected $R$-scheme $X$, and set $\meY\coloneqq[X/\Delta]$. Write $X=\spec A$ for a strictly henselian local ring $A$, say with residue field $k'$. Assume that $\Delta_k$ acts trivially on $k'$. Then, 
    \begin{itemize}
        \item $\hom^i(\meY,\G_m)$ is $\#\Delta$-torsion for all $i\ge1$;
        \item $\hom^i(\meY,\G_m)\iso\hom^i(B\Delta_{k'},\G_m)$ for all $i\ge2$; and
        \item $\Pic\meY\onto\Pic B\Delta_{k'}=\dual\Delta(k')$ is surjective.
    \end{itemize}
    It follows from the second bullet point and \cref{lem:BG-Gm-stalk} that $\hom^2(\meY,\G_m)=0$.
\end{prop}
\begin{proof}
    Let $p=\Char k$ and write $\#\Delta=p^n$. Fix some $i\ge1$. The existence of the $\Delta$-torsor $f\colon X\to\meY$ with $\hom^i(X,\G_m)=0$ is enough to deduce that $\hom^i(\meY,\G_m)=\hom^i(\meY,\G_m)[p^n]$; indeed, the norm map $\push f\G_{m,X}\to\G_{m,\meY}$ gives rise to a composition
    \[\hom^i(\meY,\G_m)\too\underbrace{\hom^i(\meY,\push f\G_{m,X})}_{\simeq\hom^i(X,\G_m)=0}\too\hom^i(\meY,\G_m)\]
    which equals multiplication by $p^n$ (see \cite[Section 3.8]{CT-S} for more details on this norm map). With this in mind, define $\Gamma$ via the follow exact sequence of sheaves on $\liset\meY$:
    \[1\too\mu_{p^n}\too\G_m\xtoo{p^n}\G_m\too\Gamma\too1.\]
    Let $\meY_0$ denote the special fiber of $\meY\to\spec R$, and consider the two term complexes $\msC\coloneqq[\G_{m,\meY}\xto{p^n}\G_{m,\meY}]$ and $\msC_0\coloneq[\G_{m,\meY_0}\xto{p^n}\G_{m,\meY_0}]$. The distinguished triangle $\mu_{p^n}\to\msC\to\Gamma[-1]\xto{+1}$ (along with its analogue over $\meY_0$) gives rise to the following homomorphism of exact sequences:
    \[\begin{tikzcd}
        \hom^{i-2}(\meY,\Gamma)\ar[r]\ar[d]&\hom^i(\meY,\mu_{p^n})\ar[r]\ar[d]&\hom^i(\meY,\G_m\xto{p^n}\G_m)\ar[r]\ar[d]&\hom^{i-1}(\meY,\Gamma)\ar[r]\ar[d]&\hom^{i+1}(\meY,\mu_{p^n})\ar[d]\\
        \hom^{i-2}(\meY_0,\Gamma)\ar[r]&\hom^i(\meY_0,\mu_{p^n})\ar[r]&\hom^i(\meY_0,\G_m\xto{p^n}\G_m)\ar[r]&\hom^{i-1}(\meY_0,\Gamma)\ar[r]&\hom^{i+1}(\meY_0,\mu_{p^n}).
    \end{tikzcd}\]
    Since $\mu_{p^n}$ and $\Gamma$ are both torsion sheaves and $\meY\to\spec R$ is proper, proper base change \cite[Theorem 1.3]{olsson-proper} (see also \cite[Corollary VI.2.7]{milne-et}) followed by an application of the five lemma implies that all vertical maps above are isomorphisms. Now, the distinguished triangle $\G_m[-1]\to\msC\to\G_m\xto{+1}$ (along with its analogue over $\meY_0$) gives rise to the following homomorphism of short exact sequences:
    \begin{equation}\label{homses:wild-arg}
        \begin{tikzcd}
            0\ar[r]&\hom^{i-1}(\meY,\G_m)/p^n\ar[r]\ar[d]&\hom^i(\meY,\G_m\xto{p^n}\G_m)\ar[d, "\sim" sloped]\ar[r]&\hom^i(\meY,\G_m)[p^n]\ar[r]\ar[d]&0\\
            0\ar[r]&\hom^{i-1}(\meY_0,\G_m)/p^n\ar[r]&\hom^i(\meY_0,\G_m\xto{p^n}\G_m)\ar[r]&\hom^i(\meY_0,\G_m)[p^n]\ar[r]&0.
        \end{tikzcd}
    \end{equation}
    Noting that $\red{\p{\meY_0}}\simeq B\Delta_{k'}$, we have a surjection
    \[\hom^i(\meY,\G_m)=\hom^i(\meY,\G_m)[p^n]\onto\hom^i(\meY_0,\G_m)[p^n]=\hom^i(\meY_0,\G_m)\underset{\cref{lem:tame-red-Gm}}\simeq\hom^i(B\Delta_{k'},\G_m).\]
    When $i=1$, this (combined with \cref{lem:BG-Gm-stalk}) proves the third bullet point of the claim. Assume now that $i\ge2$. \cref{lem:tame-red-Gm} allows us to rewrite \cref{homses:wild-arg} as
    \begin{equation}\label{seshom:Br-pn}
        \begin{tikzcd}
            0\ar[r]&\hom^{i-1}(\meY,\G_m)/p^n\ar[r]\ar[d]&\hom^2(\meY,\G_m\xto{p^n}\G_m)\ar[d, "\sim" sloped]\ar[r]&\hom^2(\meY,\G_m)[p^n]\ar[r]\ar[d]&0\\
            0\ar[r]&\hom^{i-1}(B\Delta_{k'},\G_m)/p^n\ar[r]&\hom^2(\meY_0,\G_m\xto{p^n}\G_m)\ar[r]&\hom^2(B\Delta_{k'},\G_m)[p^n]\ar[r]&0.
        \end{tikzcd}
    \end{equation}
    We claim the left vertical arrow in \cref{seshom:Br-pn} is surjective. Indeed, applying the argument so far to $i-1$ in place of $i$ shows that $\hom^{i-1}(\meY,\G_m)$ is $p^n$-torsion (as is $\hom^{i-1}(B\Delta_{k'},\G_m)$) and that
    \[\hom^{i-1}(\meY,\G_m)/p^n=\hom^{i-1}(\meY,\G_m)=\hom^{i-1}(\meY,\G_m)[p^n]\onto\hom^{i-1}(B\Delta_{k'},\G_m)[p^n]=\hom^{i-1}(B\Delta_{k'},\G_m)=\hom^{i-1}(B\Delta_{k'},\G_m)/p^n\]
    is surjective. Now, \cref{seshom:Br-pn} shows that $\hom^i(\meY,\G_m)[p^n]\to\hom^i(B\Delta_{k'},\G_m)[p^n]$ is injective and so an isomorphism. Both $\hom^i(\meY,\G_m)$ and $\hom^i(B\Delta_{k'},\G_m)$ are $p^n$-torsion, so we conclude the second bullet of the claim.
\end{proof}

\subsection{Vanishing of $\homR^2\push c\G_m$}
\begin{prop}\label{prop:tame-loc-Gm-coh}
    Let $R$ be a strictly henselian noetherian local ring with residue field $k$. Let $G/R$ be a finite linearly reductive group, suppose $G$ acts on some finite $R$-scheme $X$, and set $\meX\coloneqq[X/G]$. Suppose further that 
    \begin{itemize}
        \item $X$ is the spectrum of a strictly henselian local ring $A$ with residue field $k'$; and
        \item $G_k$ acts trivially on $\spec k'$.
    \end{itemize}
    Write
    \[0\too\Delta\too G\too\ul Q\too0\]
    for $G$'s connected-\'etale sequence, with $Q$ an abstract group.  Then, there is an exact sequence
    \[0\to\ghom^1(Q,\G_m(k'))\to\Pic\meX\to\ghom^1(\Delta,\G_{m,X/R})^Q\to\ghom^2(Q,\G_m(k'))\to\hom^2(\meX,\G_m)\to0.\]
    In particular, if $\ghom^2(Q,\G_m(k'))=0$, then $\hom^2(\meX,\G_m)=0$.
\end{prop}
\begin{proof}
    Set $\meY\coloneqq[X/\Delta]$, so $\meY\to\meX$ is a $Q$-torsor, and consider the Hochschild--Serre spectral sequence 
    \[E_2^{ij}=\ghom^i(Q,\hom^j(\meY,\G_m))\implies\hom^{i+j}(\meX,\G_m).\]
    Below, we calculate many of the terms of this spectral sequence. In every argument below, keep in mind that $Q$ is a tame abstract group -- so $p\coloneqq\Char k\ge0$ does \important{not} divide $\#Q$ -- and $\Delta$ is a finite connected diagonalizable $R$-group.
    \begin{enumerate}
        \item Claim: The ring $A^\Delta$ of $\Delta$-invariants is strictly henselian and has $k'$ as its residue field. Furthermore,
        \[E_2^{n0}=\ghom^n(Q,\G_m(\meY))\simeq\ghom^n(Q,\G_m(k'))\tforall n\ge1.\]
        
        Note that, since $R$ is noetherian, $A^\Delta\subset A$ is a finite $R$-algebra as well. It is therefore a product of strictly henselian local rings \cite[\href{https://stacks.math.columbia.edu/tag/03QJ}{Tag 03QJ}]{stacks-project}. At the same time, being a subset of $A$, it contains no nontrivial idempotents, so $A^\Delta$ must itself be strictly henselian local.
        Furthermore, because $\Delta$ is linear reductive, taking $\Delta$-invariants is an exact functor (on the category of $R$-modules with $\Delta$-action) and so commutes with quotients. It follows that $A^\Delta$ has $(k')^\Delta=k'$ as a quotient, so $k'$ must be its residue field. Now, $\G_m(\meY)=\G_m(A^\Delta)$ by construction, and \cref{lem:abs-grp-loc-Gm-coh} tells us that $\ghom^n(Q,\G_m(A^\Delta))\simeq\ghom^n(G,\G_m(k'))$ for all $n\ge1$, as claimed.
        
        \item Claim: $E_2^{01}=\ghom^0(Q,\Pic\meY)\simeq\ghom^1(\Delta,\G_{m,X/R})^Q$ and $E_2^{ij}=0$ for all $i,j\ge1$.

        By \cref{cor:local-coh}, $\Pic\meY\simeq\ghom^1(\Delta,\G_{m,X/R})$ so the claimed computation of $E_2^{01}$ holds. Fix now some $i,j\ge1$. Then, \cref{prop:diag-quot-Br-vanish} shows that $\hom^j(\meY,\G_m)$ is $\#\Delta$-torsion. Thus, $E_2^{ij}=\ghom^i(Q,\hom^j(\meY,\G_m))$ is both $\#\Delta$-torsion and $\#Q$-torsion. Since $\gcd(\#\Delta,\#Q)=1$, we conclude that $E_2^{ij}=0$.
        
        \item Claim: $\hom^2(\meY,\G_m)\simeq0$, so $E_2^{02}=\ghom^0(Q,\hom^2(\meY,\G_m))\simeq0$.

        This follows from \cref{prop:diag-quot-Br-vanish}.
    \end{enumerate}
    The claimed exact sequence now follows from considering the low degree exact sequence associated to our spectral sequence $E_2^{ij}$.
\end{proof}
\begin{defn}\label{defn:brauerless}
    Let $R$ be a strictly local ring, and let $G/R$ be a finite group scheme with connected-\'etale sequence
    \begin{equation}\label{ses:conn-etale}
        0\too \Delta\too G\too\ul Q\too0.
    \end{equation}
    We say that $G/R$ is \define{Brauerless} if $\Delta$ is diagonalizable and $\ghom^3(Q,\Z)=0$, when $Q$ acts trivially on $\Z$. More generally, if $S$ is a scheme, we will say a finite group scheme $G/S$ is \define{Brauerless} if $G_{\msO_{S,\bar s}}$ is Brauerless for every geometric point $\bar s\to S$.
\end{defn}
\begin{prop}\label{prop:brauerless-equiv/R}
    Let $R$ be a strictly local ring with residue field $k$, and let $G/R$ be a finite linearly reductive $R$-group scheme. Then, the following are equivalent
    \begin{alphabetize}
        \item $G/R$ is Brauerless.
        \item For every map local map $R\to A$ to a strictly local ring $A$, $G_A$ is Brauerless.
        \item $\ghom^3(G,\ul\Z)=0$.
        \item $G_k/k$ is Brauerless.
        \item $\hom^3(\pi_0(G),\Z)=0$, where $\pi_0(G)$ is the abstract group of connected components of $G$.
        \item The abstract group $\pi_0(G)$ of connected components of $G$ is Brauerless, when viewed as the constant group scheme $\ul{\pi_0(G)}_R$ over $R$.
    \end{alphabetize}
\end{prop}
\begin{proof}
    \hfill\begin{itemize}[align=left]
        \item[(\bp a$\iff$\bp b)] Since the connected-\'etale sequence \cref{ses:conn-etale} on $R$ restricts to the one over $A$, this equivalence is clear.
        \item[($\bp a\iff\bp c$)] Use notation as in \cref{ses:conn-etale} for $G$'s connected-\'etale sequence. Since the group scheme $\ul\Z$ is \'etale, every map $G\to\ul\Z$ (of schemes) is constant on connected components of $G$ and so factors through $\ul Q$. A completely analogous statement holds with $G$ replaced by $G^n$; hence,
        \[C^n(\ul Q,\ul\Z)=\Nat(\ul Q^n,\ul\Z)\iso\Nat(G^n,\ul\Z)=C^n(G,\ul\Z)\tforall n,\]
        and so $\ghom^n(\ul Q,\ul\Z)\iso\ghom^n(G,\ul\Z)$ for all $n$. Thus, $G$ is Brauerless if and only if $\ghom^3(G,\ul\Z)=0$.
        \item[$(\bp a\iff\bp d)$] The connected-\'etale sequence \cref{ses:conn-etale} of $G$ restricts to the connected-\'etale sequence of $G_k$ from which one sees that $G_k$ is Brauerless if and only if $G$ is. 
        \item[($\bp a\iff\bp e$)] The group $\pi_0(G)$ of connected components is precisely the abstract group $Q$ appearing in \cref{ses:conn-etale}, so \bp e is a restatement of \bp a.
        \item[($\bp e\iff\bp f$] This is a matter of expanding definitions.
        \qedhere
    \end{itemize}
\end{proof}
\begin{cor}\label{cor:brauerless-equiv}
    Let $G/S$ be a finite linearly reductive $S$-group scheme. Then, the following are equivalent
    \begin{itemize}
        \item $G$ is Brauerless.
        \item For every map $\spec R\to S$ from a strictly local scheme, $G_R$ is Brauerless.
        \item For every geometric point $\bar s\to S$, $\ghom^3(G_{\bar s},\ul\Z)=0$.
        \item For every geometric point $\bar s\to S$, $\ghom^3(\pi_0(G_{\bar s}),\Z)=0$, where $\pi_0(G_{\bar s})$ is the abstract group of connected components of $G_{\bar s}$.
    \end{itemize}
\end{cor}
\begin{rem}
    In the case that $G$ is an abstract group, Meier \cite[Definition 3]{meier-cs} used the term `poor' for essentially the same notion, but I felt that `Brauerless' was more evocative of the utility of this definition in this context. 
\end{rem}
\begin{ex}\label{ex:mun-bruaerless}
    For every $n$, both $\zmod n$ and $\mu_n$ are Brauerless.
\end{ex}
\begin{rem}\label{rem:comm-Brauerless=cyclic}
    An abstract finite abelian group is Brauerless if and only if it is cyclic; this is a consequence of the K\"unneth formula \cite[Exercise 6.1.8]{weibel}. Consequently, a finite commutative group scheme $G$ over a strictly local ring $R$ is Brauerless if and only if its maximal \'etale quotient is cyclic.
\end{rem}
\begin{ex}
    By \cite[Example 4]{meier-cs}, any geometric automorphism group of the moduli space $\meY(1)$ of elliptic curves is Brauerless.
\end{ex}
\begin{lemma}\label{lem:brauerless-Gm-vanish}
    If a finite abstract group $Q$ is Brauerless, then for any separably closed field $k$ of characteristic $p\nmid\#Q$, $\ghom^2(Q,\G_m(k))=0$ when $Q$ acts trivially on $\G_m(k)$.
\end{lemma}
\begin{proof}
    This was essentially proven in \cite[Lemma 5]{meier-cs}, except $k$ there was assumed algebraically closed. Below, we show how to modify Meier's argument.

    Let $M=\G_m(k)[1/p]$ (by which, we simply mean $\G_m(k)$ if $p=0$), and note that $\ghom^2(Q,M)\simeq\ghom^2(Q,\G_m(k))$ since $p\nmid\#Q$. Because $k$ is separably closed of characteristic $p$, $M$ is divisible. It follows from the structure theorem for divisible groups \cite[Theorem 23.1]{fuchs-i} that $M$ is a subgroup (and so direct summand) of some group of the form $(\Q/\Z)^{\oplus I}\oplus\Q^{\oplus J}$ for some sets $I,J$. Thus, it suffices to show $\ghom^2(Q,\Q/\Z)=0$ and $\ghom^2(Q,\Q)=0$. The latter of these holds simply because $Q$ is finite. For the former, the exact sequence $0\to\Z\to\Q\to\Q/\Z\to0$ identifies $\ghom^2(Q,\Q/\Z)\simeq\ghom^3(Q,\Z)$ which vanishes by assumption.
\end{proof}
\begin{defn}\label{defn:locally-Brauerless}
    We say a tame algebraic stack $\meX$ is \define{locally Brauerless} if all of its geometric automorphism groups are Brauerless (see \cref{defn:brauerless}).
\end{defn}
\begin{thm}\label{thm:R2Gm-vanishes}
    Let $\meX$ be a locally Brauerless tame algebraic stack, with coarse moduli map $c\colon\meX\to X$. Then, $\homR^2\push c\G_m=0$.
\end{thm}
\begin{proof}
    Let $\bar x\to X$ be a geometric point. By \cref{cor:tame-stack-sh}, $\meX\by_X\msO_{X,\bar x}\simeq[\spec R/G]$ for some strictly henselian local ring $R$ acted upon by some finite linearly reductive group $G/\msO_{X,\bar x}$ satisfying $G_{\bar x}\cong\ul\Aut(\meX,\bar x)$. Thus, $(\homR^2\push c\G_m)_{\bar x}\simeq\hom^2([\spec R/G],\G_m)$. Write $\ul Q_{\msO_{X,\bar x}}$ for the maximal \'etale quotient of $G$, and write $k$ for $R$'s residue field. By \cref{prop:tame-loc-Gm-coh}, $\hom^2([\spec R/G],\G_m)=0$ if $\ghom^2(Q,\G_m(k))$ vanishes. To see this, first note that since $G_{\bar x}$ is Brauerless (by assumption on $\meX$), \cref{prop:brauerless-equiv/R} (in particular, the equivalence between statements $\bp d$ and $\bp f$) shows that $Q$ is Brauerless as well and so $\ghom^2(Q,\G_m(k))=0$  by \cref{lem:brauerless-Gm-vanish}. We conclude that every stalk of $\homR^2\push c\G_m$ vanishes.
\end{proof}
\begin{rem}
    The DM case of \cref{thm:R2Gm-vanishes}, which was essentially proven in \cite[Theorem 6]{meier-cs}, only requires \cref{lem:abs-grp-loc-Gm-coh} instead of the more difficult \cref{prop:tame-loc-Gm-coh}.
\end{rem}
\begin{cor}\label{cor:R2Gm-vanishes-sequence}
    Let $\meX$ be as in \cref{thm:R2Gm-vanishes}. Then, there is an exact sequence
    \[\begin{tikzcd}
        0\ar[r]&\Pic X\ar[r, "\pull c"]\ar[d, phantom, ""{coordinate, name=Z}]&\Pic\meX\ar[r]&\hom^0(X,\homR^1\push c\G_m)
        \ar[dll, rounded corners, to path = { -- ([xshift=2ex]\tikztostart.east)
                                                  |- (Z) [near end]\tikztonodes
                                                  -| ([xshift=-2ex]\tikztotarget.west)
                                                  -- (\tikztotarget)}, "\d_2^{0,1}"']\\
        &\hom^2(X,\G_m)\ar[r, "\pull c"]&\hom^2(\meX,\G_m)\ar[r]&\hom^1(X,\homR^1\push c\G_m)\ar[r, "\d_2^{1,1}"]&\ker\p{\hom^3(X,\G_m)\xto{\pull c}\hom^3(\meX,\G_m)}
        .
    \end{tikzcd}\]
\end{cor}
\begin{proof}
    This is the exact sequence of low degree terms in the Leray spectral sequence $E_2^{ij}=\hom^i(X,\homR^j\push c\G_m)\implies\hom^{i+j}(\meX,\G_m)$.
\end{proof}
In order to ease difficulties which may arise in attempting to compute the cohomology groups $\hom^i(X,\homR^1\push c\G_m)$ appearing above, it may be most useful to apply \cref{cor:R2Gm-vanishes-sequence} only to stacks over separably closed fields (or, more generally, strictly Henselian local rings). Despite this, it can still serve as one ingredient in a larger computation of Brauer groups over more general bases; for example, below we will compute $\Br\meX(1)_R$ when $R$ is a strictly Henselian local ring, and then later extend this computation to other bases in \cref{sect:X(1)-example}.

\begin{ex}[$\Br\meX(1)_R$]\label{ex:X(1)R}
    Let $R$ be a noetherian strictly local ring, and let $\meX=\meX(1)/R$ be the moduli space of generalized elliptic curves, sometimes also denoted $\bar\meM_{1,1}$, and let $c\colon\meX\to\P^1$ be its coarse moduli space. Assume that $6\in\units R$, so $\meX/R$ is tame. We will apply \cref{cor:R2Gm-vanishes-sequence} in order to compute that $\Br\meX=\Br\P^1_R=0$. First, one produces the exact sequence
    \begin{equation}\label{ses:R1-X(1)}
        0\too \ul{\zmod3}_0\oplus\ul{\zmod2}_{1728} \too \homR^1\push c\G_m \too \ul{\zmod 2} \too 0
    \end{equation}
    of \'etale sheaves on $\P^1_R$, where $\ul{\zmod n}_x$ denotes the pushforward $i_{x,*}\ul{\zmod n}$ along the inclusion $i_x\colon\spec R\to\P^1_R$ of $x\in\P^1(R)$. To keep this example relatively tidy, we postpone a derivation of \cref{ses:R1-X(1)} for now; it will later be a direct consequence of \cref{lem:two-step-push-es} (as will be spelled out in the proof of \cref{lem:push-f-X(1)}).
    
    The long exact sequence in cohomology associated to \cref{ses:R1-X(1)} immediately shows that $\hom^1(\P^1_R,\homR^1\push c\G_m)=0$. Hence, the exact sequence in \cref{cor:R2Gm-vanishes-sequence} yields
    \[\hom^0(\P^1_R,\homR^1\push c\G_m)\xtoo{\d_2^{0,2}}\hom^2(\P^1_R,\G_m)\xtoo{\pull c}\hom^2(\meX(1)_R,\G_m)\too0.\]
    Above, note that $\im\d_2^{0,2}\subset\Br'(\P^1_R)=\hom^2(\P^1_R,\G_m)$, but $\Br'(\P^1_R)=\Br'(R)=0$ by \cite[Chapter II, Theorem 2]{gabber:thesis}, so $\d_2^{0,2}=0$, i.e. $\pull c$ above is injective.\footnote{Alternatively, at least when $R$ is regular, one can instead directly show that $\pull c\colon\hom^2(\P^1,\G_m)\to\hom^2(\meX(1),\G_m)$ is injective as follows. There exists an open $U\subset\P^1$ which admits a section $s\colon U\to\meX(1)_U$ of $c$ (one can take $U=\P^1\sm\{0,1728,\infty\}$; see e.g. \cite[The proof of Proposition III.1.4(c)]{silverman}). Hence, $\hom^2(U,\G_m)\xto{\pull c}\hom^2(\meX(1)_U,\G_m)$ is injective. By \cref{prop:brauer-injects}, $\hom^2(\P^1,\G_m)\into\hom^2(U,\G_m)$ and $\hom^2(\meX(1),\G_m)\into\hom^2(\meX(1)_U,\G_m)$ are both injective as well. It follows that $\pull c\colon\hom^2(\P^1,\G_m)\to\hom^2(\meX(1),\G_m)$ must be injective.} Thus, $\hom^2(\P^1_R,\G_m)\iso\hom^2(\meX(1)_R,\G_m)$. 
\end{ex}
\begin{ex}[$\Br\meY_0(2)_R$]\label{ex:Y0(2)}
    Let $R$ be a \important{regular} noetherian strictly local $\Z[1/2]$-algebra, and let $\meX=\meY_0(2)/R$ be the moduli space of elliptic curves equipped with a cyclic subgroup of order $2$. Let $c\colon\meX\to\A^1_R\sm\{0\}$ denote its coarse moduli space (see \cite[Corollary 3.15]{achenjang2024brauer}) and note that $\meX$ is tame (see \cite[Lemma 3.14]{achenjang2024brauer}). It was shown in \cite[Lemma 7.5]{achenjang2024brauer} that $\meX$ is a $\zmod2$-gerbe over $\A^1\sm\{0,-1/4\}$ while points above $-1/4$ have automorphism group $\mu_4$. Using this, one can produce an exact sequence
    \begin{equation}\label{ses:R1-Y0(2)}
        0\too\ul{\zmod2}_{-1/4}\too\homR^1\push c\G_m\too\ul{\zmod2}\too0
    \end{equation}
    of \'etale sheaves on $\A^1\sm\{0\}$, analogous to \cref{ses:R1-X(1)}. Again, we postpone giving the details of the construction of this sequence until later (see \cref{sect:Y0(2)-example}).
    From this exact sequence, one immediately sees that 
    \[\hom^1(\A^1\sm\{0\},\homR^1\push c\G_m)\simeq\hom^1(\A^1\sm\{0\},\zmod2)=\hom^1(\A^1\sm\{0\},\mu_2)\simeq\G_m(\A^1\sm\{0\})/2\simeq\zmod2,\] 
    and that $\#\hom^0(\A^1\sm\{0\},\homR^1\push c\G_m)=4$. One can check that the Hodge bundle on $\meY_0(2)$ generates $\hom^0(\A^1\sm\{0\},\homR^1\push c\G_m)$, so $\Pic\meX\to\hom^0(\A^1\sm\{0\},\homR^1\push c\G_m)$ is surjective and $\hom^0(\A^1\sm\{0\},\homR^1\push c\G_m)\simeq\zmod4$. Hence, the exact sequence in \cref{cor:R2Gm-vanishes-sequence} yields
    \[0\too\hom^2(\A^1_R\sm\{0\},\G_m)\xtoo{\pull c}\hom^2(\meY_0(2)_R,\G_m)\too\zmod2\xto{\d_2^{1,1}}\hom^3(\A^1_R\sm\{0\},\G_m).\]
    We claim that the differential $\d_2^{1,1}$ above vanishes, so that we have an exact sequence
    \begin{equation}\label{ses:BrY0(2)R}
        0\too\hom^2(\A^1_R\sm\{0\},\G_m)\too\hom^2(\meY_0(2)_R,\G_m)\too\zmod2\too0.
    \end{equation}
    Because $\im\d_2^{1,1}\subset\hom^3(\A^1\sm\{0\},\G_m)[2]$, the Kummer sequence $0\to\mu_2\to\G_m\to\G_m\to0$ implies that it suffices to show that $\hom^3(\A^1_R\sm\{0\},\mu_2)=0$. Gabber's absolute cohomological purity theorem \cite[Theorem 2.1.1]{fujiwara-gabber} applied to the closed immersion $Y\coloneqq\spec R\sqcup\spec R\xto{0\sqcup\infty}\P^1_R$ implies that
    \[\hom^r(R,\zmod2)\oplus\hom^r(R,\zmod2)=\hom^r(Y,\zmod 2(-1))\simeq\hom^{r+2}_Y(\P^1_R,\zmod 2)=\hom_Y^{r+2}(\P^1_R,\mu_2)\]
    for all $r$. In particular, since $R$ is strictly henselian, we see that $\hom^n_Y(\P^1_R,\mu_2)=0$ for all $n\ge3$. It then follows from the usual long exact sequence of cohomology with supports in $Y$ that $\hom^n(\P^1_R,\mu_2)\iso\hom^n(\A^1_R\sm\{0\},\mu_2)$ for all $n\ge3$. Now, by the proper base change theorem \cite[Corollary VI.2.7]{milne-et}, $\hom^3(\P^1_k,\mu_2)\simeq\hom^3(\P^1_R,\mu_2)\simeq\hom^3(\A^1_R\sm\{0\},\mu_2)$, where $k$ is the (separably closed) residue field of $R$. Finally,  $\hom^3(\P^1_k,\mu_2)=0$, e.g. by \cite[Theorem VI.1.1]{milne-et}. This completes the derivation of \cref{ses:BrY0(2)R}.
\end{ex}

\section{\bf Brauer groups of gerbes}\label{sect:gerbes}

In this section, we use \cref{thm:R2Gm-vanishes} to study the Brauer groups of (tame, locally Brauerless) gerbes.
\begin{rec}
    Recall from \cref{sect:conventions} that by `gerbe' we always mean an `fppf gerbe'. Consequently, in this section, we use the flat-fppf site over algebraic stacks (resp. big fppf site over schemes) instead of the lisse-\'etale site (resp. small \'etale). We signify this by making use of the subscript $\fppf{}$, where appropriate.
\end{rec}
\begin{lemma}\label{lem:H2(R G)}
    Let $R$ be a strictly local ring, let $G/R$ be a finite linearly reductive group, and let $c\colon\meH\to\spec R$ be a $G$-gerbe over $R$. Then, $\meH(R)\neq\emptyset$ (so $\meH\cong BG_R$).
\end{lemma}
\begin{proof}
    Let $0\to\Delta\to G\to\ul Q\to0$ be $G$'s connected-\'etale sequence, and let $\meM\coloneq\meH\thickslash\Delta$, the rigidification as in \cite[Appendix A]{AOV}. Then, $\meM$ is a $Q$-gerbe over $R$ and $\meH$ is a $\Delta$-gerbe over $\meM$. Note that $\meM(R)\neq\emptyset$ since $Q$ is \'etale and $R$ is strictly local; fix some section $s\colon\spec R\to\meM$. Set $\meH'\coloneqq\meH\by_{\meM,s}R$, so it suffices to show that $\meH'(R)\neq\emptyset$. This $\meH'$ is a $\Delta$-gerbe over $R$; noting that $\Delta$ is commutative (even diagonalizable), it suffices to show that $\fhom^2(R,\Delta)=0$. One argues as in \cite[Lemma 3.13]{AOV} (which assume that $R$ is a separably closed field); briefly, $\Delta$ is a product of groups of the form $\mu_n$, for various values of $n$, and the Kummer sequence $1\to\mu_n\to\G_m\to\G_m\to1$ shows that $\fhom^2(R,\mu_n)=0$. 
\end{proof}
\begin{lemma}\label{lem:gerbe-pushforwards}
    Let $S$ be a scheme, and let $G/S$ be a finite linearly reductive group scheme. Let $c\colon\meH\to S$ be a $G$-gerbe over $S$. Then,
    \[\push c\G_m\simeq\G_m\tand \homR^1\push c\G_m\simeq\dual G.\]
    If, furthermore, $G$ is Brauerless, then also $\homR^2\push c\G_m=0$.
    Note here that $\dual G\coloneqq\ul\Hom(G,\G_m)$ is commutative even if $G$ is not.
\end{lemma}
\begin{proof}
    $\push c\G_m\simeq\G_m$ simply because $c$ is a coarse space map, and $\homR^2\push c\G_m=0$ (if $G$ is Brauerless) by \cref{thm:R2Gm-vanishes}. To compute $\homR^1\push c\G_m$, we first construct a map $\homR^1\push c\G_m\to\dual G$. For this, it suffices to construct functorial maps $\hom^1(\meH_T,\G_m)\to\dual G(T)$ for any scheme $T/S$. To ease notation in describing such maps, we may as well assume $T=S$ and so construct a map $\Pic\meH=\hom^1(\meH,\G_m)\to\dual G(S)$. Such a map is constructed in \cite[Beginning of Section 6]{lopez2023picard} (essentially sending a line bundle $\msL$ to the induced action of $G$ on fibers of $\msL$). Thus, we get a map $\homR^1\push c\G_m\to\dual G$, which we claim is an isomorphism on stalks. Indeed, if $S=\spec R$ is the spectrum of a strictly henselian local ring, then $\meH\simeq BG_R$ by \cref{lem:H2(R G)}. Thus, $\hom^1(BG_R,\G_m)\iso\dual G(R)$ by \cref{lem:BG-Gm-stalk}.
\end{proof}
\begin{rem}
    In the setting of \cref{lem:gerbe-pushforwards}, one also has $\fhomR^1\push c\G_m\simeq\dual G$ and $\fhomR^2\push c\G_m=0$ (if $G$ is Brauerless) since these are the fppf-sheafifications of the corresponding higher \'etale pushforwards.
\end{rem}
\begin{set}\label{set:gerbe}
    Fix some choice of $S$, $G/S$, and $c\colon\meH\to S$ as in \cref{lem:gerbe-pushforwards}. Furthermore, let $\ab G/S$ denote the abelianization of $G$.
\end{set}
\begin{rem}
    Recall that \cref{cor:G-dual-etale} shows that $\dual G$ is \'etale, so $\ast\hom(-,\dual G)=\ast\fhom(-,\dual G)$. On the other hand, $G$ is not necessarily \'etale, so $\ast\hom(-,G)\neq\ast\fhom(-,G)$ in general. In particular, while $\meH$ represents a class in $\fhom^2(S,G)$, it may not represent a class in $\hom^2(S,G)$ (i.e. it may not be \'etale-locally trivial). Because of this, we will need to carefully keep track of \'etale vs. fppf cohomology in this section.
\end{rem}
\begin{notn}
    Let $\ab\meH\to S$ denote the abelianization of $\meH/S$; this is the rigidification $\meH\thickslash\der G$, where $\der G=[G,G]$ is the derived subgroup of $G$, and it represents the image of $[\meH]$ under the natural map $\fhom^2(S,G)\to\fhom^2(S,\ab G)$; see \cite[Section 6]{lopez2023picard}.
\end{notn}
We study the fppf-Leray spectral sequence
\begin{equation}\label{ss:leray-gerbe-Gm}
    E_2^{ij}=\fhom^i(S,\fhomR^j\push c\G_m)\implies\hom^{i+j}(\meH,\G_m),
\end{equation}
pictured in \cref{fig:leray-gerbe} (the objects on the $E_2$-page are identified by using \cref{lem:gerbe-pushforwards}).
\begin{figure}[h]
    \centering
    \[\begin{tikzcd}
        0\\
        \dual G(S)\ar[rrd, "d_2^{0,1}"] & \hom^1(S, \dual G)\ar[rrd, "d_2^{1,1}"] \\
        \G_m(S) & \Pic(S) & \hom^2(S,\G_m) & \hom^3(S,\G_m)
    \end{tikzcd}\]
    \caption{The $E_2$-page of the Leray spectral sequence for a $G$-gerbe $\meH\to S$.}
    \label{fig:leray-gerbe}
\end{figure}
We are in particular interested in computing the differentials $\d_2^{i,1}$ for $i\ge0$. We claim that these come from cupping with the class $[\meH]\in\hom^2(S,G)$ of the gerbe $\meH\to S$ (even better put, cupping with $[\ab\meH]\in\hom^2(S,\ab G)$). To prove this, we exploit the multiplicative structure relating \cref{ss:leray-gerbe-Gm} to the Leray spectral sequences
\[F_2^{ij}=\fhom^i(S,\fhomR^j\push c\ab G)\implies\fhom^{i+j}(\meH,\ab G)\tand{}'F_2^{ij}=\fhom^i(S,\fhomR^j\push c\dual G)\implies\hom^{i+j}(\meH,\dual G)\]
for $\ab G$ and $\dual G$.
\begin{notn}
    We use $\d_2=\d_2^{ij}$ to denote the differentials on $E_2^{ij}$, but use $\delta_2=\delta_2^{ij}$ and ${}'\delta_2={}'\delta_2^{ij}$ to denote the differentials on $F_2^{ij},{}'F_2^{ij}$, respectively.
\end{notn}
The natural evaluation map $G\by\dual G\to \G_m$ (which factors through $\ab G\by\dual G\to\G_m$) induces a product $\smile\colon F_2^{ij}\by{}'F_2^{i'j'}\too E_2^{i+i',j+j'}$ compatible with differentials in the sense that they satisfy the Leibniz rule
\begin{equation}\label{eqn:leibniz}
    \d_2\p{\alpha\smile\beta}=\delta_2\alpha\smile\beta+(-1)^{i+j}\alpha\smile{}'\delta_2\beta
\end{equation}
for all $\alpha\in F_2^{ij}$ and $\beta\in{}'F_2^{i'j'}$; see \cite[Section IV.6.8]{bredon-sheaf-theory} for details on the construction of this product.
\begin{warn}\label{warn:cup}
    As explained in \cite[Section IV.6.8]{bredon-sheaf-theory}, the product $\smile\colon F_2^{ij}\by{}'F_2^{i'j'}\too E_2^{i+i',j+j'}$ alluded to above is \important{not} simply the cup product
    \[\fhom^i(S,\fhomR^j\push c\ab G)\by\fhom^{i'}(S,\fhomR^{j'}\push c\dual G)\xtoo\cup\fhom^{i+i'}(S,\fhomR^j\push c\ab G\otimes\fhomR^{j'}\push c\dual G)\too\fhom^{i+i'}(S,\fhomR^{j+j'}\push c\G_m)\]
    induced by $\ab G\by\dual G\to\G_m$, but is instead $(-1)^{ji'}$ times the above composition.
\end{warn}
\begin{lemma}\label{lem:G-easy-push}
    $\push c\dual G\simeq\dual G$, $\push c\ab G\simeq\ab G$, and $\fhomR^1\push c\ab G\simeq\ul\Hom(G,\ab G)\simeq\ul\End(\ab G)$.
\end{lemma}
\begin{proof}
    The first two of these holds simply because $c$ is a coarse space map. For the last, we construct a map $\fhomR^1\push c\ab G\to\ul\Hom(G,\ab G)$ and verify that it is locally an isomorphism. As in the proof of \cref{lem:gerbe-pushforwards}, it suffices to construct functorial maps $\hom^1(\meH_T,\G_m)\to\ul\Hom(G,\ab G)(T)=\Hom_T(G_T,\ab G_T)$ for schemes $T/S$, and even to just construct a suitable map $\hom^1(\meH,\ab G)\to\Hom_S(G,\ab G)$. Fix a $\ab G$-torsor $P$ on $\meH$. For any pair $(T/S,t\in\meH(T))$, $\pull tP$ is a $\ab G$-torsor on $T$ and so one obtains a homomorphism
    \[G(T)=\Aut_\meH(t)\too\Aut(\pull tP)=\ab G(T).\]
    Ranging over such pairs, this defines a homomorphism $G\to\ab G$ (say initially defined over some cover of $S$, but which then descends to one defined over $S$). This defines the map $\homR^1\push cG\to\ul\Hom(G,\ab G)$. To see that this is an isomorphism locally, we may assume that $\fhom^1(S,\ab G)=0$ and that $\meH\simeq BG_S$ (since both of these hold fppf-locally on $S$). Then, a $\ab G$-torsor on $\meH=BG_S$ is the data of a (necessarily trivial) $\ab G$-torsor on $S$ equipped with an equivariant $G$-action, i.e. is the data of a homomorphism $G\to\ab G$. Thus, $\hom^1(BG_S,\ab G)\iso\Hom_S(G,\ab G)$, concluding the proof.
\end{proof}
\begin{lemma}\label{lem:actual-SS-diff-comp}
    Viewing $\id=\id_{\ab G}\in\End_S(\ab G)=F_2^{01}$, $\delta_2(\id)=[\ab\meH]\in\fhom^2(S,\ab G)=F_2^{20}$.
\end{lemma}
\begin{proof}
    To simplify notation in this proof, rename $\meH\coloneqq\ab\meH$ and $G\coloneqq\ab G$. By \cite[Proposition V.3.2.1]{giraud}, $\delta_2(\id)$ is represented by the gerbe $D/S$ whose fiber category $D(T)$ over an $S$-scheme $T/S$ is the groupoid of $G$-torsors over $\meH_T\coloneqq\meH\by_ST$ whose image under the map
    \[\fhom^1(\meH_T,G)\too\hom^0(T,\fhomR^1\push cG)\simeq\End_T(G_T)\]
    is the identity endomorphism.
    To show that $D\simeq\meH$, it suffices to construct a morphism $\meH\to D$ of $G$-gerbes over $S$. Consider an arbitrary $T/S$ and $t\in\meH(T)$. Define the two morphisms
    \[\alpha\coloneqq\id_{\meH_T}\colon\meH_T\to\meH_T\tand\beta\colon\meH_T=\meH\by_ST\xto{\pr_2}T\xto{(t,\id)}\meH\by_ST=\meH_T.\]
    Then, $P\coloneqq\ul\Isom(\alpha,\beta)\to\meH_T$ is a $G$-torsor over $\meH_T$. The $G$-torsor $\pull tP\to T$ is trivial (as $\alpha,\beta$ both pull back to the identity map on $T$) and, by construction, $G(T)=\Aut_\meH(t)$ acts on it via the identity map $G(T)\to G(T)$. In other words, the $\Aut_\meH(t)$-torsor action on $\ul\Isom(\pull t\alpha,\pull t\beta)$ is via composition on the right. Thus, $P\in D(T)$ and this defined our desired morphism $\meH\to D$.
\end{proof}
\begin{prop}\label{prop:gerbe-Gm-es}
    Continuing to use notation as in \cref{set:gerbe}, there is an exact sequence
    \begin{equation}\label{es:gerbe-Gm}
        0\to\Pic(S)\xto{\pull c}\Pic(\meH)\to\hom^0(S,\dual G)\xto{[\ab\meH]\cup}\hom^2(S,\G_m)\xto{\pull c}\hom^2(\meH,\G_m)\to\hom^1(S,\dual G)\xto{-[\ab\meH]\cup}\hom^3(S,\G_m).
    \end{equation}
    Above, the last term can be replaced with $\ker\p{\hom^3(S,\G_m)\xto{\pull c}\hom^3(\meH,\G_m)}$.
\end{prop}
\begin{proof}
    The exact sequence of low degree terms in the Leray spectral sequence \cref{ss:leray-gerbe-Gm}, pictured in \cref{fig:leray-gerbe}, is
    \[0\to\Pic S\xto{\pull c}\Pic\meH\to\dual G(S)\xto{\d_2^{0,1}}\hom^2(S,\G_m)\xto{\pull c}\hom^2(\meH,\G_m)\to\hom^1(S,\dual G)\xto{\d_2^{1,1}}\ker\p{\hom^3(S,\G_m)\xto{\pull c}\hom^3(\meH,\G_m)},\]
    so it suffices to compute the differentials $\d_2^{0,1},\d_2^{1,1}$. Fix any $i\ge0$. The multiplicative structure \cref{eqn:leibniz} shows that
    \[\d_2^{i,1}\p{\alpha\smile\beta}=\delta_2\alpha\smile\beta+(-1)^{i+1}\alpha\smile{}'\delta_2\beta=\delta_2\alpha\smile\beta\]
    for any $\alpha\in F_2^{01}=\hom^0(S,\fhomR^1\push c\ab G)\simeq\End_S(\ab G)$ (\ref{lem:G-easy-push}) and $\beta\in{}'F_2^{i0}=\hom^i(S,\dual G)$. In particular, taking $\alpha=\id_{\ab G}\in\End_S(\ab G)$ (and $\beta\in\hom^i(S,\dual G)\simeq E_2^{i1}$ arbitrary), and letting $\cup$ denote the usual cup product (see \cref{warn:cup}), we see that
    \[\d_2(\beta)=\d_2(\id_{\ab G}\cup\beta)=(-1)^i\d_2(\id_{\ab G}\smile\beta)=(-1)^i\delta_2(\id_{\ab G})\smile\beta=(-1)^i\delta_2(\id_{\ab G})\cup\beta,\]
    so it suffices to compute $\delta_2(\id_{\ab G})\in\hom^2(S,G)=F_2^{20}$. We conclude by \cref{lem:actual-SS-diff-comp}.
\end{proof}
\begin{rem}\label{rem:extend-lopez}
    Using notation as in \cref{prop:BG-Gm-coh}, by definition of $\dual G$ and of cup products, one has a commutative diagram
    \[\begin{tikzcd}[column sep = large]
        \hom^0(S,\dual G)\ar[r, "{[\ab\meH]\cup}"]\ar[d, equals]&\hom^2(S,\G_m)\ar[d, equals]\\
        \Hom_S(G,\G_m)\ar[r, "{(-)_*[\meH]}"]&\hom^2(S,\G_m)
        .
    \end{tikzcd}\]
    Hence, \cref{prop:gerbe-Gm-es} extends an exact sequence constructed by Lopez in \cite[Theorem 1.2]{lopez2023picard}. 
\end{rem}
\begin{rem}\label{rem:BG-Gm-coh}
    Say $G/S$ is as in \cref{prop:gerbe-Gm-es} and $\meH=BG_S$. Then, the section $S\to BG_S$ shows that the differentials in the spectral sequence \cref{fig:leray-gerbe} vanish and that one has
    \[\Pic(BG_S)\cong\Pic S\oplus\hom^0(S,\dual G)\tand\hom^2(BG_S,\G_m)\cong\hom^2(S,\G_m)\oplus\hom^1(S,\dual G),\]
    generalizing \cref{prop:BG-Gm-coh}.
\end{rem}
\begin{cor}[of \cref{prop:gerbe-Gm-es}]
    Assume that $S$ is regular and noetherian. If there exists a dense open $U\subset S$ over which $\meH$ trivializes (i.e. there exists a section $U\to\meH$), then the exact sequence of \cref{prop:gerbe-Gm-es} splits into two exact sequences:
    \begin{alignat*}{8}
        0 &\too&& \centermathcell{\Pic(S)} &&\xtoo{\pull c} &&\centermathcell{\Pic(\meH)} &&\too&& \hom^0(S,\dual G) &&\too &&\; 0 \\
        0 &\too && \hom^2(S,\G_m) &&\xtoo{\pull c}&& \hom^2(\meH,\G_m) &&\too&& \hom^1(S,\dual G) &&\xtoo{-[\ab\meH]\cup} && \ker\p{\hom^3(S,\G_m)\xto{\pull c}\hom^3(\meH,\G_m)}
    \end{alignat*}
\end{cor}
\begin{proof}
    It suffices to show that $\hom^0(S,\dual G)\xto{[\ab\meH]\cup}\hom^2(S,\G_m)$ is the zero map. Note that we have a commutative square
    \[\commsquare{\hom^0(S,\dual G)}{[\ab\meH]\cup}{\hom^2(S,\G_m)}{}{}{\hom^0(U,\dual G)}{[\ab\meH_U]\cup}{\hom^2(U,\G_m).}\]
    The claim now follows from the facts that $[\meH_U]=0$ (so $[\ab\meH_U]=0$), by assumption, and that $\hom^2(S,\G_m)\to\hom^2(U,\G_m)$ is injective, by \cref{prop:brauer-injects}. 
\end{proof}
\begin{lemma}
    Assume that $G=\mu_{n,S}$ for some $n\ge1$. Then, $\pull c\colon\hom^2(S,\G_m)\to\hom^2(\meH,\G_m)$ is injective if and only if $[\meH]\in\im\p{\hom^1(S,\G_m)\to\fhom^2(S,\mu_n)}$, i.e. $\meH$ is the gerbe of $n$th roots of some line bundle on $S$.
\end{lemma}
\begin{proof}
    By \cref{prop:gerbe-Gm-es} (and \cref{rem:extend-lopez}), injectivity of $\hom^2(S,\G_m)\to\hom^2(\meH,\G_m)$ is equivalent to 
    \[\Hom(\mu_n,\G_m)\xto{(-)_*[\meH]}\hom^2(S,\G_m)\]
    being the zero map. Note that $\Hom(\mu_n,\G_m)\cong\zmod n$, generated by the natural inclusion $\mu_n\into\G_m$ so the displayed map is the zero map if and only if $[\meH]\mapsto0$ under $\hom^2(S,\mu_n)\to\hom^2(S,\G_m)$. The Kummer sequence $1\to\mu_n\to\G_m\to\G_m\to1$ shows that this is the case if and only if $[\meH]\in\im\p{\hom^1(S,\G_m)\to\fhom^2(S,\mu_n)}$.
\end{proof}

\section{\bf Brauer groups of root stacks}\label{sect:root-stacks}
In this section, we compute the Brauer groups of root stacks. To begin, we recall their construction, following \cite[Section 10.3]{olsson}. One can also see \cite{cadman} for more information on root stacks.
\begin{defn}
    Let $\meX$ be an algebraic stack. A \define{generalized effective Cartier divisor} on $\meX$ is a pair $(\msL,\rho)$ where $\msL$ is a line bundle on $\meX$ and $\rho\colon\msL\to\msO_\meX$ is a morphism of $\msO_\meX$-modules. An isomorphism between generalized effective Cartier divisors $(\msL',\rho')$ and $(\msL,\rho)$ is an isomorphism $\msL'\iso\msL$ of line bundles which identifies $\rho'$ with $\rho$.
\end{defn}
\begin{ex}
    Let $\meD\into\meX$ be an effective Cartier divisor, with ideal sheaf $\msO_\meX(-\meD)$. Then, the inclusion $\msO_\meX(-\meD)\into\msO_\meX$ is a generalized effective Cartier divisor.
\end{ex}
\begin{prop}
    Let $\G_m\actson\A^1$ with the usual scaling action $\lambda\cdot t=\lambda t$. Then, the quotient stack $[\A^1/\G_m]$ parameterizes generalized effective Cartier divisors; i.e., given a scheme $S$, the category $[\A^1/\G_m](S)$ is equivalent to the groupoid of generalized effective Cartier divisors on $S$.
\end{prop}
\begin{proof}
    This is \cite[Proposition 10.3.7]{olsson}.
\end{proof}
\begin{defn}
    Let $\meX$ be an algebraic stack, and let $(\msL,\rho)$ be a generalized effective Cartier divisor on $\meX$. For an integer $n\ge1$, the \define{$n$th root stack} $\sqrt[n]{(\msL,\rho)/\meX}$ is the fibered category over $\Sch$ whose objects are tuples $\p{f\colon S\to\meX,(\msM,\lambda),\sigma}$ where $f$ is a morphism from a scheme $S$ to $\meX$, $(\msM,\lambda)$ is a generalized effective Cartier divisor on $S$, and $\sigma\colon(\msM^{\otimes n},\lambda^n)\iso(\pull f\msL,\pull f\rho)$ is an isomorphism of generalized effective Cartier divisors on $\meX$. Morphisms in $\sqrt[n]{(\msL,\rho)/\meX}$ are as expected; see \cite[10.3.9]{olsson}.
\end{defn}
\begin{notn}
    If $\meD\into\meX$ is an effective Cartier divisor with ideal sheaf $\msO_\meX(-\meD)$ and if $n\ge1$ is an integer, we'll write $\sqrt[n]{\meD/\meX}\coloneqq\sqrt[n]{\p{\msO_\meX(-\meD),\msO_\meX(-\meD)\into\msO_\meX}/\meX}$.
\end{notn}
\begin{prop}\label{prop:root-stack-summary}
    Let $\meX$ be an algebraic stack, let $(\msL,\rho)$ be a generalized effective Cartier divisor on $\meX$, fix some $n\ge1$, and set $\meX_n\coloneqq\sqrt[n]{(\msL,\rho)/\meX}$. Then,
    \begin{enumerate}
        \item $\meX_n$ is an algebraic stack, and is in fact the fiber product
        \[\commsquare{\meX_n}{}{[\A^1/\G_m]}{}{(-)^n}\meX{(\msL,\rho)}{[\A^1/\G_m].}\]
        \item If $\msL=\msO_\meX$ with $\rho$ corresponding to $f\in\Gamma(\meX,\msO_\meX)$, then
        \[\meX_n\simeq\sq{\frac{\Spec_\meX\p{\msO_\meX[T]/(T^n-f)}}{\mu_n}},\]
        where $\mu_n$ acts trivially on $\msO_\meX$ and acts on $T$ via $\zeta\cdot T=\zeta T$.
        \item The natural morphism $\pi\colon\meX_n\to\meX$ is an isomorphism over the open locus $\meU\subset\meX$ where $\rho$ is an isomorphism.
        \item If $n$ is invertible on $\meX$, then $\pi\colon\meX_n\to\meX$ is a DM morphism in the sense of \cite[\href{https://stacks.math.columbia.edu/tag/04YW}{Tag 04YW}]{stacks-project}.
    \end{enumerate}
\end{prop}
\begin{proof}
    See \cite[Theorem 10.3.10]{olsson} for the case the $\meX$ is a scheme, to which the general case reduces.
\end{proof}
\begin{set}\label{set:root-induction}
    For the reminder of the section, we fix the following notation.
    \begin{enumerate}
        \item $\meX$ is a tame algebraic stack over some scheme $S$. We write $c\colon\meX\to X$ for its coarse space map.

        We further assume that $\meX$ is locally Bruaerless, and that there exists an open locus $U\subset X$, flat over $S$, above which $c$ is an isomorphism.
        \item Fix an effective Cartier divisor $\iota\colon\meD\into\meX$ which is flat over $S$.

        We further assume that $\meD$ lives in $\inv c(U)$, so $c$ maps $\meD$ isomorphically onto its image in $X$.
        \item We fix an integer $n\ge1$, and write $\meX_n\coloneqq\sqrt[n]{\meD/\meX}$. We let $\pi\colon\meX_n\to\meX$ denote the natural morphism.
    \end{enumerate}
\end{set}
\begin{lemma}\label{lem:root-R2}
    $\homR^2\push\pi\G_m=0$.
\end{lemma}
\begin{proof}
    By \cref{prop:root-stack-summary}\bp3, $\pi$ is an isomorphism away from $\meD$, so $\homR^2\push\pi\G_m$ is supported along $\meD$, i.e. $\homR^2\push\pi\G_m=\push\iota \homR^2\ppush{\pi\vert_{\inv\pi(\meD)}}\G_m$. At the same time, by \cref{set:root-induction}, above $\meD$, $\pi$ is equivalent to the composition $\inv\pi(\meD)\xto\pi\meD\xto cX$. By construction, this composition is the coarse space map of $\inv\pi(\meD)$ and all geometric automorphism groups of $\inv\pi(\meD)$ are of the form $\mu_n$ for some $n$ (consequence of \cref{prop:root-stack-summary}\bp2) and so are all Brauerless (\cref{ex:mun-bruaerless}). Thus, we conclude by \cref{thm:R2Gm-vanishes}.
\end{proof}
\begin{lemma}\label{lem:root-R1}
    $\homR^1\push\pi\G_m\simeq\push\iota\ul{\zmod n}$.
\end{lemma}
\begin{proof}
    As in the proof of \cref{lem:root-R2}, as a consequence of \cref{prop:root-stack-summary}\bp3, we need only compute the higher pushforward along the composition
    \[\pi'\colon\meD'\coloneqq\inv\pi(\meD)\xto\pi\meD\xiso cc(\meD)\eqqcolon D.\]
    Since $\pi$ is an $n$th root stack along $\meD$, it follows that $\meD'$ is an infinitesimal extension of a stack $\meH$ which is a $\mu_n$-gerbe over $D$, say via $\rho\colon\meH\to D$ (see e.g. \cite[Below Example 2.4.3]{cadman}). We claim that the natural map $\phi\colon\homR^1\push\pi'\G_m\to\homR^1\push\rho\G_m$ is an isomorphism. This can be checked at the level of stalks, where it follows from \cref{lem:tame-red-Gm}. Finally, the claim follows from \cref{lem:gerbe-pushforwards} which shows that $\homR^1\push\rho\G_m\simeq\ul{\zmod n}$. 
\end{proof}
\begin{rem}
    When $\meX$ is a curve, an alternative proof of \cref{lem:root-R1} is given by \cref{cor:alt-R1Gm-root}.
\end{rem}
\begin{prop}\label{prop:root-Gm-coh}
    There are exact sequences
    \begin{alignat}{9}
        0 &\too&& \centermathcell{\Pic(\meX)} &&\xtoo{\pull\pi} &&\centermathcell{\Pic(\meX_n)} &&\too&& \hom^0(\meD,\ul{\zmod n}) &&\too &&\; 0 \label{ses:Pic-root}\\
        0 &\too && \hom^2(\meX,\G_m) &&\xtoo{\pull\pi}&& \hom^2(\meX_n,\G_m) &&\too&& \hom^1(\meD,\ul{\zmod n}) &&\xtoo\delta && \ker\p{\hom^3(\meX,\G_m)\xto{\pull\pi}\hom^3(\meX_n,\G_m)} \label{es:Br-root}
    \end{alignat}
\end{prop}
\begin{proof}
    Note that $\push\pi\G_m=\G_m$, either by a local computation using \cref{prop:root-stack-summary}\bp2 or by realizing that $\meX_n$ and $\meX$ share $X$ as a coarse space. Consider the Leray spectral sequence $E_2^{ij}=\hom^i(\meX,\homR^j\push\pi\G_m)\implies\hom^{i+j}(\meX_n,\G_m)$; by making use of \cref{lem:root-R1,lem:root-R2}, where appropriate, we see that its low degree terms fit in to the exact sequence
    \[0\to\Pic(\meX)\xto{\pull\pi}\Pic(\meX_n)\to\hom^0(\meD,\zmod n)\xto{d_2^{0,1}}\hom^2(\meX,\G_m)\xto{\pull\pi}\hom^2(\meX_n,\G_m)\to\hom^1(\meD,\ul{\zmod n})\to\hom^3(\meX,\G_m).\]
    It suffices to show that $\ker\d_2^{0,1}=0$, equivalently, to show that $\Pic(\meX_n)\to\hom^0(\meD,\zmod n)$ is surjective. Perhaps the quickest way to see this is to observe that \cite[Corollary 3.1.2]{cadman} shows that $\#\Pic(\meX_n)/\Pic(\meX)=n^{\#\pi_0(\meD)}=\#\hom^0(\meD,\zmod n)$.
\end{proof}
\begin{rem}
    The Picard groups of root stacks have previously been computed in both \cite[Section 3.1]{cadman} and \cite[Section 3]{lopez2023picard}.
\end{rem}
\begin{ex}[${\sqrt[n]{\infty/\P^1}}$]
    We construct an example showing that the morphism $\delta$ in \cref{es:Br-root} can be nonzero. Fix a field $k$ and an integer $n>1$. Let $\meX\coloneqq\sqrt[n]{\infty/\P^1_k}\xtoo c\P^1_k$. Then, \cref{prop:root-Gm-coh} yields an exact sequence
    \[0\too\hom^2(\P^1_k,\G_m)\xtoo{\pull c}\hom^2(\meX,\G_m)\too\hom^1(k,\zmod n)\xtoo\delta\hom^3(\P^1_k,\G_m).\]
    We will later show (see \cref{prop:stacky-Faddeev}) that $\pull c\colon\hom^2(\P^1_k,\G_m)\to\hom^2(\meX,\G_m)$ is an isomorphism. Thus, in this example, $\delta\colon\hom^1(k,\zmod n)\into\hom^3(\P^1_k,\G_m)$ must be injective.
\end{ex}

\section{\bf Brauer groups of stacky curves}\label{sect:stacky-curve}

In the present section, we combine the work from the previous two in order to study Brauer groups of stacky curves, especially over algebraically closed fields. Since stacky curves often arise as gerbes over root stacks, we begin with a general result above such spaces.
\begin{lemma}
    Let $S$ be a base scheme. Let $\meY/S$ be an algebraic $S$-stack with coarse moduli space $\rho\colon\meY\to X$. Let $G/S$ be a finite linearly reductive group, and let $\pi\colon\meX\to\meY$ be a $G$-gerbe, so we have a commutative diagram
    \[\compdiag\meX\pi\meY\rho{X,}c\]
    where $c$ is $\meX$'s coarse space map.
\end{lemma}
\begin{lemma}\label{lem:two-step-push-es}
    There is an exact sequence
    \[0\too\homR^1\push\rho\G_m\too\homR^1\push c\G_m\too\dual G\too\homR^2\push\rho\G_m,\]
    of \'etale sheaves on $X$.
\end{lemma}
\begin{proof}
    We use the Grothendieck spectral sequence $E_2^{ij}=\homR^i\push\rho\p{\homR^j\push\pi\G_m}\implies\homR^{i+j}\push c\G_m$. Using \cref{lem:gerbe-pushforwards} to compute $\homR^1\push\pi\G_m\simeq\dual G$, we see that its exact sequence of low degree terms begins with
    \[0\too\homR^1\push\rho\G_m\too\homR^1\push c\G_m\too\push\rho\dual G\too\homR^2\push\rho\G_m.\]
    Finally, $\push\rho\dual G=\dual G$ (as \'etale sheaves on $X$) since $\dual G$ is a scheme and $\rho$ is a coarse space map.
\end{proof}
We now restrict to the curve setting.

\subsection{Over perfect fields}
\begin{set}
    Let $k$ be a field. Let $\meY/k$ be a regular stacky curve which is generically schemey, and let $\rho\colon\meY\to X$ be its coarse space map. Let $G/k$ be a finite linearly reductive group, and let $\pi\colon\meX\to\meY$ be a $G$-gerbe. Consider the commutative diagram
    \[\compdiag\meX\pi\meY\rho {X,}c\]
    where $c$ is $\meX$'s coarse space map.
\end{set}
\begin{ex}\label{ex:iterated-root-stack}
    One could take $\meY$ above to be a multiply rooted stack over $X$, i.e. one could fix distinct closed points $x_1,\dots,x_r\in X$ as well as integers $e_1,\dots,e_r>1$ and then set
    \begin{equation}\label{eqn:iterated-root-stack}
        \meY=\sqrt[e_1]{x_1/X}\by_X\dots\by_X\sqrt[e_r]{x_r/X}\xtoo\rho X.  
    \end{equation}
    Assuming that $X$ is a regular curve, this is always a tame regular generically schemey stacky curve.
\end{ex}
\begin{rem}
    It is well known that every tame DM stacky curve over a field can be written as a gerbe over a multiply rooted stack; see \cite[Proposition 1.5]{lopez2023picard} and/or \cite{GS}.
\end{rem}
\begin{ex}
    Over a field $k$ of characteristic not $2$ or $3$, one can take $\meX=\meY(1)$ with $G=\ul{\zmod2}$ and $\meY\simeq\sqrt[3]{0/\A^1}\by_{\A^1}\sqrt[2]{1728/\A^1}$.
\end{ex}
\begin{ex}
    Over a field $k$ of characteristic not $2$, one can take $\meX=\meY_0(2)$ with $G=\ul{\zmod2}$ and $\meY\simeq\sqrt[2]{\left.-\frac14\right/(\A^1\sm\{0\})}$. This follows from \cite[Section 7.1]{achenjang2024brauer}.
\end{ex}
\begin{lemma}\label{lem:kbar-stacky-curve-push}
    Assume that $k$ is algebraically closed. Then, there is a skyscraper sheaf $\msF$ and an exact sequence
    \[0\to\hom^0(X,\homR^1\push\rho\G_m)\to\hom^0(X,\homR^1\push c\G_m)\to\hom^0(X,\dual G)\to\hom^0(X,\msF)\to\hom^1(X,\homR^1\push c\G_m)\to\hom^1(X,\dual G)\to0.\]
    Furthermore, $\hom^i(X,\homR^1\push c\G_m)\simeq\hom^i(X,\dual G)$ for all $i\ge2$.
\end{lemma}
\begin{proof}
    This follows from the exact sequence
    \[0\too\homR^1\push\rho\G_m\too\homR^1\push c\G_m\too\dual G\too\msF\too0\twhere\msF\coloneqq\im\p{\dual G\to\homR^2\push\rho\G_m}\]
    of \cref{lem:two-step-push-es} along with the fact that $\homR^i\push\rho\G_m$, for $i=1,2$, is supported on a finite subscheme of $X$ and hence they (and all their subsheaves) are acyclic.
\end{proof}
\begin{rem}
    If $\meY$ is as in \cref{ex:iterated-root-stack}, then $\homR^1\push\rho\G_m$ in \cref{lem:kbar-stacky-curve-push} is $\homR^1\push\rho\G_m\simeq\bigoplus_{i=1}^r\ul{\zmod{e_i}}_{x_i}$; this follows from repeated use of \cref{lem:root-R1} (or one use of \cref{cor:alt-R1Gm-root} per connected component of $X$). Furthermore, in this case, $\homR^2\push\rho\G_m=0$ by \cref{thm:R2Gm-vanishes}, so one has the exact sequence
    \[0\too\bigoplus_{i=1}^r\ul{\Zmod{e_i}}_{x_i}\too\homR^1\push c\G_m\too\dual G\too0.\]
    Compare this with the exact sequence of \cite[Proposition 6.9]{AMS} obtained for $\meX=\meY(1)$.
\end{rem}
\begin{assump}
    Assume from now on that $\meY$ (and so also $\meX$) is tame.
\end{assump}
\begin{prop}\label{prop:stacky-tsen}
    Assume that $k$ is algebraically closed and that $\meX$ is locally Brauerless. Then, there is a surjection
    \[\Br\meX=\hom^2(\meX,\G_m)\onto\hom^1(X,\dual G)\]
    which is an isomorphism if $\meY$ is locally Brauerless as well.
\end{prop}
\begin{proof}
    We use the Leray spectral sequence $E_2^{ij}=\hom^i(X,\homR^j\push c\G_m)\implies\hom^{i+j}(\meX,\G_m)$. Note that $E_2^{i0}=\hom^i(X,\G_m)=0$ for $i\ge2$, by \cite[Theorem 7.2.7]{fu-etale}, and $E_2^{02}=0$ by \cref{thm:R2Gm-vanishes}. Thus, one concludes that $\hom^2(\meX,\G_m)\iso\hom^1(X,\homR^1\push c\G_m)$. Now, \cref{lem:kbar-stacky-curve-push} provides the surjection $\hom^2(\meX,\G_m)\iso\hom^1(X,\homR^1\push c\G_m)\onto\hom^1(X,\dual G)$ and shows that it is an isomorphism if $\hom^0(X,\homR^2\push\rho\G_m)=0$ (e.g. if $\meY$ is locally Brauerless). Finally, \cref{lem:Br=Br'-stacky-curves} shows that $\Br\meX=\hom^2(\meX,\G_m)$.
\end{proof}
\begin{cor}\label{cor:stacky-curve-Br-field}
    Assume that $k$ is a perfect field and that both $\meX$ and $\meY$ are locally Brauerless. Then, there is an exact sequence
    \[\begin{tikzcd}
        0\ar[r]&\hom^1(k,\G_m(\bar X))\ar[r]\ar[d, phantom, ""{coordinate, name=Z}]&\Pic\meX\ar[r]&\p{\Pic\bar\meX}^{G_k}
        \ar[dll, rounded corners, to path = { -- ([xshift=2ex]\tikztostart.east)
                                                  |- (Z) [near end]\tikztonodes
                                                  -| ([xshift=-2ex]\tikztotarget.west)
                                                  -- (\tikztotarget)}]\\
        &\hom^2(k,\G_m(\bar X))\ar[r]&\ker\p{\hom^2(\meX,\G_m)\to\hom^1(\bar X,\dual G)}\ar[r]&\hom^1(k,\Pic\bar\meX)\ar[r]&\hom^3(k,\G_m(\bar X))
        .
    \end{tikzcd}\]
    where $\bar\meX\coloneqq\meX_{\bar k}$, $\bar X\coloneqq X_{\bar k}$, and $G_k=\Gal(\bar k/k)$.
\end{cor}
\begin{proof}
    This is the exact sequence of low degree terms in the spectral sequence $E_2^{ij}=\hom^i(k,\hom^j(\bar\meX,\G_m))\implies\hom^{i+j}(\meX,\G_m)$.
\end{proof}
\begin{cor}\label{cor:curve-num-field}
    Use tha notation and assumptions of \cref{cor:stacky-curve-Br-field}. Furthermore, assume that $k$ is a number field and $X$ is proper, geometrically reduced, and geometrically connected. Then, there is an exact sequence
    \[0\too\Pic\meX\too(\Pic\bar\meX)^{G_k}\too\Br k\too\ker\p{\hom^2(\meX,\G_m)\to\hom^1(\bar X,\dual G)}\too\hom^1(k,\Pic\bar\meX)\too0.\]
\end{cor}
\begin{proof}
    For such $X$, $\G_m(\bar X)=\G_m(\bar k)$. Furthermore, it is known that $\hom^3(k,\G_m)=0$ if $k$ is a number field; see \cite[Remark 6.7.10]{poonen-rat-pts}. Apply \cref{cor:stacky-curve-Br-field}.
\end{proof}
\begin{rem}
    In \cite{AM,achenjang2024brauer}, it was shown that, for $k$ of characteristic not $2$,
    \[\Br\meY(1)_{\bar k}=0\twhile\Br\meY_0(2)_{\bar k}=\zmod2.\]
    \cref{prop:stacky-tsen} helps explain why one of these groups vanishes while the other does not; note that\footnote{Technically, \cref{prop:stacky-tsen} only proves $\Br\meY(1)_k\simeq\hom^1(\A^1_k,\mu_2)$ when $\Char k\nmid6$. However, if you replaces its use of \cref{thm:R2Gm-vanishes} with \cref{cor:char3-R2Gm=0}, then one can deduce that this equality holds even in characteristic $3$.}
    \[\hom^1(\A^1_{\bar k},\mu_2)=0\twhile\hom^1(\A^1_{\bar k}\sm\{0\},\mu_2)=\zmod2.\qedhere\]
\end{rem}
\begin{rem}
    To justify the tameness assumption in \cref{prop:stacky-tsen}, we remark that $\Br\meY(1)_{\bar\F_2}=\zmod2$, as was shown by Shin \cite{shin}, but $\fhom^1(\A^1_{\bar\F_2},\mu_2)=0$. Note that $\meY(1)_{\bar\F_2}$ is nowhere tame.
\end{rem}
\begin{rem}
    Furthermore, even for tame $\meX$ over $k=\bar k$, some additional condition is necessary for one to have $\hom^2(\meX,\G_m)\simeq\hom^1(X,\dual G)$. For example, if $\meY\coloneqq\sqrt[2]{0/\A^1_k}$, where $k$ is an algebraically closed field of characteristic not $2$, and $\meX\coloneqq B\mu_{2,\meY}$, then
    \[\hom^2(\meX,\G_m)\simeq\hom^1(\meY,\mu_2)\simeq\zmod2\not\simeq0\simeq\hom^1(\A^1_k,\mu_2).\]
    Notice that, in this example, $\ul\Aut(\meX,0)\simeq\mu_{2,k}\by\mu_{2,k}$ is \important{not} Brauerless (see \cref{rem:comm-Brauerless=cyclic}).
\end{rem}
\begin{qn}
    Use the notation and assumptions of \cref{cor:stacky-curve-Br-field}. When is the map $\hom^2(\meX,\G_m)\to\hom^2(\bar\meX,\G_m)=\hom^1(\bar X,\dual G)$ surjective (onto the $\Gal(\bar k/k)$-invariant classes)?
\end{qn}
\begin{qn}
    Use the notation and assumptions of \cref{prop:stacky-tsen} and assume that $\meY$ is locally Brauerless. Can one make the isomorphism $\hom^1(X,\dual G)\iso\hom^2(\meX,\G_m)$ explicit, e.g. by writing down an Azumaya algebra over $\meX$ representing a given $\alpha\in\hom^1(X,\dual G)$?
\end{qn}

\subsection{Partial results on proper stacky curves over general bases}\label{sect:stacky-curve-prop}
In the previous section, we had to restrict attention to stacky curves over \important{perfect} fields. This was ultimately because of our use of Tsen's theorem that schemey curves over \important{algebraically closed} fields have trivial Brauer groups. However, for proper curves, Grothendieck \cite[Corollaire 5.8]{Gr3} has shown that their Brauer groups vanish once their base field is only \important{separably closed}. In the present section, we push this observation to allow us study Brauer groups of proper stacky curves over fairly general bases; informally, we go from curves over separably closed fields to those over strictly local rings to those over general base schemes. As will quickly become apparent from the argument, many of the specifics on this computation depend on subtly on the particular curve being analyzed; hence, to be able to give a reasonable statement in the end, we will eventually need to place various assumptions on our curve.
\begin{set}
    Let $S$ be a noetherian base scheme. Let $X/S$ be a smooth, projective $S$-curve with geometrically connected fibers, let $\meY/S$ be a tame algebraic $S$-stack with coarse space morphism $\rho\colon\meY\to X$. Let $G/S$ be a finite linearly reductive group, and let $\pi\colon\meX\to\meY$ be a $G$-gerbe. Consider the commutative diagram
    \[\compdiag\meX\pi\meY\rho {X,}c\]
    where $c$ is $\meX$'s coarse space map.
\end{set}
\begin{assump}
    Throughout this section, we further assume that
    \begin{itemize}
        \item $\rho\colon\meY\to X$ is an isomorphism over some dense open $U\subset X$ which surjects onto $S$; and
        \item both $\meX$ and $\meY$ are locally Brauerless.
        \qedhere
    \end{itemize}
\end{assump}
\begin{rem}
    For any $s\in S$, the fiber $U_s\subset X_s$ is nonempty and open, so dense. It follows from \cite[Proposition 11.10.10]{egaiv.3} that, for any $T/S$, $U_T$ is dense in $X_T$. We will use this implicitly in many places.
\end{rem}
\begin{rem}
    The next few results all make use of some ideas from the proof of \cref{prop:diag-quot-Br-vanish}; in particular, its reliance on proper base change: \cite[Theorem 1.3]{olsson-proper} and \cite[Corollary VI.2.7]{milne-et}.
\end{rem}
\begin{lemma}\label{lem:Br'-prop-red-i}
    Assume that $S=\spec R$ is strictly local, say with residue field $k$. If $\Pic\meX\to\Pic\meX_k$ is surjective, then $\Br'(\meX)\iso\Br'(\meX_k)$.
\end{lemma}
\begin{proof}
    Fix any $n\ge1$. As a consequence of the proper base change theorem \cite[Theorem 1.3]{olsson-proper}, we have
    \[\hom^i\p{\meX,\G_m\xto n\G_m}\iso\hom^i\p{\meX_k,\G_m\xto n\G_m}\tforall i\ge0.\]
    Taking $i=2$, this leads us to the following homomorphism of short exact sequences:
    \[\begin{tikzcd}
        0\ar[r]&\Pic(\meX)/p^n\ar[r]\ar[d, hook]&\hom^2(\meX,\G_m\xto n\G_m)\ar[d, "\sim" sloped]\ar[r]&\hom^2(\meX,\G_m)[n]\ar[r]\ar[d, two heads]&0\\
        0\ar[r]&\Pic(\meX_k)/p^n\ar[r]&\hom^2(\meX_k,\G_m\xto n\G_m)\ar[r]&\hom^2(\meX_k,\G_m)[n]\ar[r]&0.
    \end{tikzcd}\]
    The claim now follows from the snake lemma, recalling that $\Br'(\meX)=\bigcup_n\hom^2(\meX,\G_m)[n]$.
\end{proof}
\begin{lemma}\label{lem:Br'-prop-red-ii}
    Assume that $S=\spec R$ is strictly local, say with residue field $k$. Then, $\Pic\meX\to\Pic\meX_k$ is surjective. From \cref{lem:Br'-prop-red-i}, we deduce that $\Br'(\meX)\iso\Br'(\meX_k)$.
\end{lemma}
\begin{proof}
    Throughout the proof, we use that formation of $\homR^1\push c\G_m$ and of $\homR^1\push\rho\G_m$ commute with arbitrary base change, by the proper base change theorem \cite[Theorem 1.3]{olsson-proper}. By comparing the Leray spectral sequence for $\meX\to X$ with that for $\meX_k\to X_k$, we obtain the following commutative diagram
    \begin{equation}\label{cd:Pic-surj}
        \begin{tikzcd}
            0\ar[r]&\Pic X\ar[r]\ar[d, two heads]&\Pic\meX\ar[d]\ar[r]&\hom^0(X,\homR^1\push c\G_m)\ar[d]\ar[r, "\d_2^{0,1}"]&\hom^2(X,\G_m)\ar[d]\\
            0\ar[r]&\Pic X_k\ar[r]&\Pic\meX_k\ar[r]&\hom^0(X_k,\homR^1\push c\G_m)\ar[r]&\hom^2(X_k,\G_m)
        \end{tikzcd}
    \end{equation}
    with exact rows. Above, the map $\Pic X\to\Pic X_0$ is surjective by \cite[Corollaire 21.9.12]{egaiv.4}. Furthermore, every element of $\hom^0(X,\homR^1\push c\G_m)$ is torsion and so, by \cref{lem:Br'-prop-red-i} applied to $X$, the rightmost vertical map in \cref{cd:Pic-surj} -- i.e. $\hom^2(X,\G_m)\to\hom^2(X_k,\G_m)$ -- is injective when restricted to $\im\d_2^{0,1}\subset\hom^2(X,\G_m)_{\mrm{tors}}$. Thus, by the five lemma, in order to conclude that $\Pic\meX\to\Pic\meX_k$ is surjective, it suffices to show that $\hom^0(X,\homR^1\push c\G_m)\to\hom^0(X_k,\homR^1\push c\G_m)$ is surjective.

    By \cref{lem:two-step-push-es} applied to $c=\rho\circ\pi$ along with \cref{thm:R2Gm-vanishes} applied to $\rho$, we have an exact sequence
    \[0\too\homR^1\push\rho\G_m\too\homR^1\push c\G_m\too\dual G\too0.\]
    Note that $\homR^1\push\rho\G_m$ is supported on some finite $R$-scheme $Z\subset X$, so $\hom^1(X,\homR^1\push\rho\G_m)=0$. Thus, by the snake lemma, in order to show that $\hom^0(X,\homR^1\push c\G_m)\to\hom^0(X_k,\homR^1\push c\G_m)$ is surjective, it suffices to show that $\dual G(X)\onto\dual G(X_k)$ and that $\hom^0(X,\homR^1\push\rho\G_m)\onto\hom^0(X_k,\homR^1\push\rho\G_m)$. The first of these holds simply because, by \cref{cor:G-dual-etale}, $\dual G$ is \'etale (and so constant). Therefore, to finish the proof, it suffices to show that $\hom^0(X,\homR^1\push\rho\G_m)\to\hom^0(X_k,\homR^1\push\rho\G_m)$ is surjective.

    Note that $Z$, the support of $\homR^1\push\rho\G_m$, is a finite $R$-scheme and so is a disjoint union of (spectra of) strictly henselian local rings. Thus, it suffices to show that, for any strictly local $\spec A\to X$, the natural map
    \[\Pic(\meY\by_XA)=\hom^0(\spec A,\homR^1\push\rho\G_m)\to\hom^0(\spec A\otimes_Rk,\homR^1\push\rho\G_m)=\Pic\p{(\meY\by_XA)_k}\]
    is surjective. As a consequence of \cref{cor:tame-stack-sh}, we have that $\meY\by_XA\simeq[\spec B/H]$ for some strictly local $B$ and finite linearly reductive $H/A$ whose action is trivial on $B$'s residue field $k_B$. Hence, by \cref{lem:tame-red-Gm},
    \[\Pic\p{(\meY\by_XA)_k}\simeq\Pic\p{(\meY\by_XA)_{k,\mrm{red}}}=\Pic(BH_{k_B})\simeq\dual H(k_B).\]
    Now, the maps $(\meY\by_XA)_k\to\meY\by_XA\simeq[\spec B/H]\to[\spec A/H]=BH_A$ induce
    \[\dual H(A)=\Pic(BH_A)\too\Pic(\meY\by_XA)\too\Pic\p{(\meY\by_XA)_k}\simeq\dual H(k_B)\]
    on Picard groups. \cref{cor:G-dual-etale} shows that $\dual H/A$ is constant, so the above composition is surjective, finishing the proof.
\end{proof}
\begin{prop}\label{prop:Br-prop-local}
    Assume that $S=\spec R$ is strictly local. Then, there is an exact sequence
    \[0\too\hom^2(X,\G_m)\too\hom^2(\meX,\G_m)\too\hom^1(X,\dual G)\too0,\]
    and $\hom^2(X,\G_m)$ is torsion-free (i.e. $\Br'(X)=0$).
\end{prop}
\begin{proof}
    We use the Leray spectral sequence $E_2^{ij}=\hom^i(X,\homR^j\push c\G_m)\implies\hom^{i+j}(\meX,\G_m)$. Since $\homR^2\push c\G_m=0$ by \cref{thm:R2Gm-vanishes}, its low degree terms fit into the following exact sequence:
    \begin{equation}\label{es:prop-stacky-curve-local-i}
        \hom^0(X,\homR^1\push c\G_m)\xtoo{\d_2^{0,1}}\hom^2(X,\G_m)\too\hom^2(\meX,\G_m)\too\hom^1(X,\homR^1\push c\G_m)\xtoo{\d_2^{1,1}}\hom^3(X,\G_m).
    \end{equation}
    We claim that $\d_2^{0,1}=0$ and that $\d_2^{1,1}=0$. For the former, \cref{lem:Br'-prop-red-ii} (or, more accurately, \cref{lem:Br'-prop-red-i} + \cite[Corollaire 21.9.12]{egaiv.4}) applied to $X$ shows that $\Br'(X)\simeq\Br'(X_k)$. At the same time, since $X_k$ is a proper curve over a separably closed field, \cite[Corollaire 5.8]{Gr3} shows that $\Br'(X_k)=0$. Hence, $\im\d_2^{0,1}\subset\hom^2(X,\G_m)_{\mrm{tors}}=\Br'(X)=0$.

    We next show that $\d_2^{1,1}=0$ as well. It suffices to show that $\hom^3(X,\G_m)_{\mrm{tors}}=0$. Pick an arbitrary $n\ge1$ and note that $\hom^3(X,\G_m\xto n\G_m)\onto\hom^3(X,\G_m)[n]$, so it suffices to show that $\hom^3(X,\G_m\xto n\G_m)=0$. Again, as a consequence of proper base change \cite[Theorem 1.3]{olsson-proper}, $\hom^3(X,\G_m\xto n\G_m)\simeq\hom^3(X_k,\G_m\xto n\G_m)$, so we only need to show the latter vanishes. This group sits in the exact sequence
    \[\hom^3(X_k,\mu_n)\too\hom^3(X_k,\G_m\xto n\G_m)\too\hom^2(X_k,\Gamma),\twhere\Gamma\coloneqq\G_m/n.\]
    Note that $\hom^3(X_k,\mu_n)=0$ by \cite[Theorem VI.1.1]{milne-et} since $2\dim X_k=2<3$. Finally, letting $p\coloneqq\Char k$, $\Gamma$ is $p^{v_p(n)}$-torsion, so $\hom^2(X_k,\Gamma)=0$ by \cite[Remark VI.1.5(b)]{milne-et} since $k$ is separably closed and $\dim X_k=1<2$.

    To finish the proof, we show that $\hom^1(X,\homR^1\push c\G_m)\simeq\hom^1(X,\dual G)$. By \cref{lem:two-step-push-es} applied to $c=\rho\circ\pi$ along with \cref{thm:R2Gm-vanishes} applied to $\rho$, we have an exact sequence $0\to\homR^1\push\rho\G_m\too\homR^1\push c\G_m\too\dual G\too0$. Furthermore, $\homR^1\push\rho\G_m$ is supported on a finite $R$-scheme and so is acyclic. Thus, taking cohomology, we see that $\hom^1(X,\homR^1\push c\G_m)\iso\hom^1(X,\dual G)$, finished the proof.
\end{proof}
\begin{rem}
    Note that \cref{ex:X(1)R} is a special case of \cref{prop:Br-prop-local}.
\end{rem}
\begin{cor}
    The exact sequence of \cref{prop:Br-prop-local} always restricts to an isomorphism $\Br'(\meX)\simeq\hom^1(X,\dual G)$.
\end{cor}
\begin{proof}
    One needs to show that the map $\Br'(\meX)\subset\hom^2(\meX,\G_m)\to\hom^1(X,\dual G)$ is surjective. Consider the exact sequence
    \[0\too\hom^2(X_k,\G_m)\too\hom^2(\meX_k,\G_m)\too\hom^1(X_k,\dual G)\too0\]
    over $k$. Since $X_k$ is regular, $\hom^2(X_k,\G_m)=\Br'(X_k)$ (see \cite[Proposition 6.6.7]{poonen-rat-pts}) and so vanishes. Thus, $\hom^2(\meX_k,\G_m)\iso\hom^1(X_k,\dual G)$. Combining this with \cref{lem:Br'-prop-red-ii} applied to $\meX$ and proper base change \cite[Corollary VI.2.7]{milne-et} applied to $\dual G$ on $X$ shows
    \[\Br'(\meX)\iso\Br'(\meX_k)\xinto\sim\hom^2(\meX_k,\G_m)\iso\hom^1(X_k,\dual G)\fiso\hom^1(X,\dual G),\]
    as desired.
\end{proof}
\begin{cor}\label{cor:Br-prop-reg-local}
    With notation as in \cref{prop:Br-prop-local}, if $S$ is furthermore regular, then $\hom^2(\meX,\G_m)=\Br'(\meX)\simeq\hom^1(X,\dual G)$.
\end{cor}
\begin{proof}
    If $S$ is regular, then so is $X$ and then $\hom^2(X,\G_m)=\Br'(X)=0$, using \cite[Proposition 6.6.7]{poonen-rat-pts} for the first equality. Consequently, $\hom^2(\meX,\G_m)\iso\hom^1(X,\dual G)$.
\end{proof}
When $S$ is not strictly local, we consider the commutative diagram
\[\mapover\meX cXfg{S,}\]
and proceed by comparing the Leray spectral sequence for $f$ to the one for $g$, i.e., we wish to compare
\begin{equation}\label{SS:leray-prop-comp}
    E_2^{ij}=\hom^i(S,\homR^j\push f\G_m)\implies\hom^{i+j}(\meX,\G_m)\twith F_2^{ij}=\hom^i(S,\homR^j\push g\G_m)\implies\hom^{i+j}(X,\G_m).
\end{equation}
\begin{lemma}
    $\G_m\iso\push g\G_m\iso\push f\G_m$
\end{lemma}
\begin{proof}
    The first isomorphism follows from the fact that $g$ is proper with geometrically reduced and irreducible fibers. The second is a consequence of the fact that $X$ if $\meX$'s coarse space.
\end{proof}
\begin{lemma}\label{lem:prop-R-comp}
    There are short exact sequences
    \begin{alignat}{9}
        0 &\too\;&& \centermathcell{\homR^1\push g\G_m} &&\too &&\centermathcell{\homR^1\push f\G_m} &&\too&&\; \push g\homR^1\push c\G_m &&\too &&\; 0 \nonumber\\
        0 &\too && \homR^2\push g\G_m &&\too&&\; \homR^2\push f\G_m &&\too&& \centermathcell{\homR^1\push g\dual G} &&\too &&\; 0. \label{es:prop-R2}
    \end{alignat}
    Furthermore, $\homR^2\push g\G_m=0$ (and so $\homR^2\push f\G_m\iso\homR^1\push g\dual G$) if $S$ is regular.
\end{lemma}
\begin{proof}
    We begin with the exact sequence of low degree terms in the Grothendieck spectral sequence $\homR^i\push g\p{\homR^j\push c\G_m}\implies\homR^{i+j}\push f\G_m$. Using that $\push c\G_m=\G_m$ and that $\homR^2\push c\G_m=0$ (by \cref{thm:R2Gm-vanishes}), this takes the form
    \[0\to\homR^1\push g\G_m\to\homR^1\push f\G_m\to\push g\homR^1\push c\G_m\to\homR^2\push g\G_m\to\homR^2\push f\G_m\to\homR^1\push g\homR^1\push c\G_m.\]
    Passing to the level of stalks, (the proof of) \cref{prop:Br-prop-local} shows that $\homR^2\push g\G_m\to\homR^2\push f\G_m$ is injective, that $\homR^1\push g\homR^1\push c\G_m\iso\homR^1\push g\dual G$, and that $\homR^2\push f\G_m\to\homR^1\push g\homR^1\push c\G_m\simeq\homR^1\push g\dual G$ is surjective. This gives the two claimed short exact sequences. Finally, if $S$ is regular, then \cref{cor:Br-prop-reg-local} shows that $\homR^2\push g\G_m=0$.
\end{proof}
\begin{notn}
    We denote the relative Picard sheaves:
    \[\Pic_{\meX/S}\coloneqq\homR^1\push f\G_m\tcomma\Pic_{X/S}\coloneqq\homR^1\push g\G_m,\tand\Pic_{\meX/X}\coloneqq\homR^1\push c\G_m.\]
    Consequently, the first short exact sequence of \cref{lem:prop-R-comp} can be rewritten as
    \begin{equation}\label{es:prop-R1}
        0\too\Pic_{X/S}\too\Pic_{\meX/S}\too\push g\Pic_{\meX/X}\too0.
    \end{equation}
\end{notn}
As previously alluded, we want to consider the morphism $F_2^{ij}\to E_2^{ij}$ between the two spectral sequences in \cref{SS:leray-prop-comp}. Consider the low degree terms of each sequence, this gives rise to the following commutative diagram with exact columns:
\begin{equation}\label{es:ladder}
    \begin{tikzcd}
        0\ar[d]&0\ar[d]\\
        \Pic S\ar[r, equals]\ar[d]&\Pic S\ar[d]\\
        \Pic X\ar[r, hook]\ar[d]&\Pic\meX\ar[d]\\
        \Pic_{X/S}(S)\ar[r, hook]\ar[d]&\Pic_{\meX/S}(S)\ar[d]\\
        \hom^2(S,\G_m)\ar[r, equals]\ar[d]&\hom^2(S,\G_m)\ar[d]\\
        \ker\p{\hom^2(X,\G_m)\to\hom^0(S,\homR^2\push g\G_m)}\ar[r]\ar[d]&\ker\p{\hom^2(\meX,\G_m)\to\hom^0(S,\homR^2\push f\G_m)}\ar[d]\\
        \hom^1(S,\Pic_{X/S})\ar[r]\ar[d]&\hom^1(S,\Pic_{\meX/S})\ar[d]\\
        \ker\p{\hom^3(S,\G_m)\to\hom^3(X,\G_m)}\ar[r, hook]&\ker\p{\hom^3(S,\G_m)\to\hom^3(\meX,\G_m)}
    \end{tikzcd}
\end{equation}
Above, the map $\Pic_{X/S}(S)\to\Pic_{\meX/S}(S)$ (and hence also $\Pic X\to\Pic\meX$) is injective by \cref{es:prop-R1}.
\begin{assump}
    Assume from now on that $\meX(S)\neq\emptyset$.
\end{assump}
This assumption guarantees that $\hom^i(S,\G_m)\to\hom^i(\meX,\G_m)$ is injective for all $i$ and so lets us split \cref{es:ladder} into the two homomorphisms of (split) short exact sequences
\begin{equation}\label{homses:prop-R1}
    \begin{tikzcd}
        0\ar[r]&\Pic S\ar[d, equals]\ar[r, "\pull g"]&\Pic X\ar[d, "\pull c"]\ar[r]&\Pic_{X/S}(S)\ar[r]\ar[d]&0\\
        0\ar[r]&\Pic S\ar[r, "\pull f"]&\Pic\meX\ar[r]&\Pic_{\meX/S}(S)\ar[r]&0
    \end{tikzcd}
\end{equation}
and
\begin{equation}\label{homses:prop-R2}
    \begin{tikzcd}
        0\ar[r]&\hom^2(S,\G_m)\ar[d, equals]\ar[r, "\pull g"]&\ker\p{\hom^2(X,\G_m)\to\hom^0(S,\homR^2\push g\G_m)}\ar[d, "\pull c"]\ar[r]&\hom^1(S,\Pic_{X/S})\ar[r]\ar[d]&0\\
        0\ar[r]&\hom^2(S,\G_m)\ar[r, "\pull f"]&\ker\p{\hom^2(\meX,\G_m)\to\hom^0(S,\homR^2\push f\G_m)}\ar[r]&\hom^1(S,\Pic_{\meX/S})\ar[r]&0.
    \end{tikzcd}
\end{equation}
\begin{prop}
    There is an exact sequence
    \[0\too\Pic X\xtoo{\pull c}\Pic\meX\too\Pic_{\meX/X}(X)\too\hom^1(S,\Pic_{X/S})\too\hom^1(S,\Pic_{\meX/S}).\]
\end{prop}
\begin{proof}
    The snake lemma applied to \cref{homses:prop-R1} shows that $\coker\pull c\simeq\coker\p{\Pic_{X/S}(S)\to\Pic_{\meX/S}(S)}$. From \cref{es:prop-R1}, we see that this latter cokernel is the kernel of $\Pic_{\meX/X}(X)=\hom^0(S,\push g\Pic_{\meX/X})\to\hom^1(S,\Pic_{X/S})$ and that we can continue the just-established sequence $0\to\Pic X\to\Pic\meX\to\Pic_{X/X}(X)\to\hom^1(S,\Pic_{X/S})$ as indicated in the proposition statement.
\end{proof}
\begin{thm}\label{thm:Br-prop-best statement I could muster}
    Assume that $S$ is regular, that $\meX(S)\neq\emptyset$, and that $\hom^1(S,\Pic_{X/S})\to\hom^1(S,\Pic_{\meX/S})$ is injective.\footnote{Equivalently, $\Pic\meX\to\Pic_{\meX/X}(X)$ is surjective.} Then, $\hom^2(\meX,\G_m)$ admits a filtration whose graded pieces are
    \begin{itemize}
        \item $\hom^2(X,\G_m)$;
        \item $\coker\p{\hom^1(S,\Pic_{X/S})\to\hom^1(S,\Pic_{\meX/S})}$; and
        \item a subgroup of $\hom^0(S,\homR^1\push g\dual G)$.
    \end{itemize}
\end{thm}
\begin{proof}
    Because $S$ is regular, $\homR^2\push g\G_m=0$ and $\homR^2\push f\G_m\simeq\homR^1\push g\dual G$; see \cref{lem:prop-R-comp}. Therefore, the snake lemma applied to \cref{homses:prop-R1} shows that $\hom^2(X,\G_m)\xto{\pull c}\hom^2(\meX,\G_m)$ is injective and that
    \[\coker\p{\hom^1(S,\Pic_{X/S})\to\hom^1(S,\Pic_{\meX/S})}\simeq\frac{\ker\p{\hom^2(\meX,\G_m)\to\hom^0(S,\homR^1\push g\dual G)}}{\hom^2(X,\G_m)}.\]
    This finishes the proof. Namely, the claimed filtration is
    \[0\subset\hom^2(X,\G_m)\subset\ker\p{\hom^2(\meX,\G_m)\to\hom^0(S,\homR^1\push g\dual G)}\subset\hom^2(\meX,\G_m).\qedhere\]
\end{proof}
\begin{ex}
    Say $S$ is a regular $\Z[1/6]$-scheme and $\meX=\meX(1)$ is the moduli space of generalized elliptic curves. In this case, $\Pic_{\meX/S}=\ul\Z$ by \cite{FO}, so $\hom^1(S,\Pic_{\meX/S})=\hom^1(S,\Z)=0$ since $S$ is regular. Furthermore, $G=\ul{\zmod2}$ and the stalks of $\homR^1\push g\dual G$ are of the form $\hom^1(\P^1_R,\mu_2)=0$ (for some strictly local $R$), so $\homR^1\push g\dual G=0$ as well. Thus, \cref{thm:Br-prop-best statement I could muster} shows that
    \[\hom^2(\meX(1)_S,\G_m)\simeq\hom^2(\P^1_S,\G_m)\]
    in this case. This computation will be generalized in \cref{sect:examples}, where we prove \cref{thma:mod-curve}.
\end{ex}
\begin{rem}
    The conditions in \cref{thm:Br-prop-best statement I could muster} are made out of convenience instead of necessity. If one is interested in computing the Brauer group of a specific (possibly non-proper) stacky curve $\meX/S$ not satisfying all these assumptions, then this can likely still be done by running through the strategy employed in this section with some extra work. In \cref{sect:examples}, we do this for the modular curves appearing in \cref{thma:mod-curve}.
\end{rem}

\section{\bf Residue exact sequence}\label{sect:res-exact}
One highly useful tool for understanding Brauer groups of schemes is the residue exact sequence; see, for example, \cite[Theorem 6.8.3]{poonen-rat-pts} or \cite[Proposition 2.1]{Gr3}. Our goal in the current section is to explore the extension of this sequence to the stacky setting. One such extension has been obtained previously in \cite[Proposition 2]{hassett2015stable}, at least for 2-dimensional DM stacks. In this section, we first use a stacky version of Artin's relative cohomological purity theorem to construct residue maps for stacks which are smooth over a scheme. We then explore the case of 1-dimensional stacks, where residue maps can more easily be constructed using the usual divisor exact sequence as in \cite[Example III.2.22]{milne-et} or \cite[(2) in Section 1]{Gr2}.
\begin{prop}
    Let $S$ be an affine regular, noetherian scheme of dimension $\le1$, and fix a prime $\l$ which is invertible on $S$. Let $i\colon\meD\into\meX$ be a closed immersion of smooth algebraic $S$-stacks which is pure of codimension $1$. Let $j\colon\meU\into\meX$ be the open complement of $\meD$. Then, there is an exact sequence
    \[0\too\Br'(\meX)\{\l\}\too\Br'(\meU)\{\l\}\too\hom^1(\meD,\Q_\l/\Z_\l)\too\hom^3(\meX,\G_m)\{\l\}\too\hom^3(\meU,\G_m)\{\l\}.\]
\end{prop}
\begin{proof}
    One can argue as in \cite[Proof of Proposition 2.14]{AM}, using Laszlo and Olsson's stacky version of Artin's relative cohomological purity theorem \cite[Proposition 4.9.1]{LO-six} (which was proven only for smooth $S$-stacks with $S$ as in the proposition statement) in place of \cite{AM}'s use of Gabber's aboslute purity theorem.
\end{proof}
When $\meX$ is a `stacky curve', we can and will say more. Residues of Brauer classes on regular curves can be more readily understood using the usual divisor exact sequence \cref{ses:divisor-stack} described below.
\begin{set}
    Let $\meX$ be a regular, integral noetherian algebraic stack. Note that every point $x\in\abs\meX$ in $\meX$'s underlying topological space has a residual gerbe which is in particular a reduced algebraic stack \cite[\href{https://stacks.math.columbia.edu/tag/0H22}{Tag 0H22}]{stacks-project}.
\end{set}
\begin{notn}
    Write $j\colon\meG\into\meX$ for the residual gerbe of the generic point of $\meX$. Write $\abs\meX_1$ for the set of codimension one points in $\meX$; for each $x\in\abs\meX_1$, write $i_x\colon\meG_x\into\meX$ for the inclusion of its residual gerbe.
\end{notn}
\begin{rec}
    Recall that, in this paper, we work in the lisse-\'etale (resp. small \'etale) site over a given stack $\meX$ (resp. given scheme $X$). Most of what we do is largely agnostic to the difference between this and the big \'etale site, but the exact sequence in \cref{prop:stacky-divisor-ses} does not hold on the big sites. Since this section is particularly sensative to our choice of the lisse-\'etale site, we will often be extra explicit by adding the subscript $\liset{}$ to some of our notation.
\end{rec}
\begin{prop}\label{prop:stacky-divisor-ses}
    There is an exact sequence
    \begin{equation}\label{ses:divisor-stack}
        0\too\G_{m,\meX}\too\push j\G_{m,\meG}\too\bigoplus_{x\in\abs\meX_1}i_{x,*}\ul\Z\too0
    \end{equation}
    of sheaves on $\liset\meX$.
\end{prop}
\begin{proof}
    The sequence \cref{ses:divisor-stack} is characterized by the property that for any smooth cover $U\to\meX$ by a connected scheme $U$, the pullback of \cref{ses:divisor-stack} along this cover is the usual such exact sequence over a regular locally noetherian scheme, as in \cite[Section 1]{Gr2}. This simultaneously defines all the maps in the sequence and shows that it is exact, since it is known to be exact when restricted to the small \'etale site $\et U$ for any smooth cover by a scheme.
\end{proof}
The exact sequence \cref{ses:divisor-stack} is most useful whenever $\meX$ is a curve, so that the inclusions $i_x\colon\meG_x\into\meX$ are closed immersions, as follows from \cite[\href{https://stacks.math.columbia.edu/tag/0H27}{Tag 0H27}]{stacks-project}.

\subsection{The stacky curve case}
\begin{assump}
    From now on, assume $\meX$ is a `{stacky curve}' (see \cref{sect:conventions}) in the sense that it is a one-dimensional algebraic stack with finite inertia, on top of still being regular, integral, and noetherian. It follows from Keel--Mori \cite{keel-mori,conrad-km} that $\meX$ has a coarse moduli space $c\colon\meX\to X$.
\end{assump}
\begin{rem}
    This is a somewhat more general usage of the phrase `stacky curve' than is commonly used. For example,
    \begin{itemize}
        \item Lopez \cite{lopez2023picard} requires `stacky curves' to be DM.
        \item Ellenberg--Satriano--Zureick-Brown \cite[Appendix B]{ESZB} do not require their `stacky curves' to be DM, but do require them to have birational coarse space, i.e. to be `generically scheme-y'.
        \qedhere
    \end{itemize}
\end{rem}
Our goal in the remainder of this section is to establish some basic facts about the long exact sequence in cohomology associated to \cref{ses:divisor-stack} and to compare this sequence on $\meX$ to the analogous sequence on its coarse space $X$.
\begin{notn}
    To ease notation, for $x\in\abs\meX_1$, set $\ul\Z_{\meG_x}\coloneqq i_{x,*}\ul\Z$.
\end{notn}
\begin{lemma}\label{lem:H^1XZx=0}
    For all $x\in\abs\meX_1$,
    \[\hom^1(\meX,\ul\Z_{\meG_x})\simeq\hom^1(\meG_x,\ul\Z)=0.\]
\end{lemma}
\begin{proof}
    Because $x\in\abs\meX_1$ is a closed point, $i_x\colon\meG_x\into\meX$ is a closed immersion from which we deduce that $\hom^1(\meX,\ul\Z_{\meG_x})\simeq\hom^1(\meG_x,\ul\Z)$. Furthermore, because $\ul\Z$ is a smooth group scheme, we have $\hom^1(\meG_x,\ul\Z)\simeq\fhom^1(\meG_x,\ul\Z)$. To compute this latter group, write $f\colon\meG_x\to\spec\kappa(x)$ for the morphism realizing $\meG_x$ as a gerbe over its residue field $\kappa(x)$. Note that, by considering the fppf-Leray spectral sequence associated to $f$, we have an exact sequence
    \begin{equation}\label{es:leray-H1GZ}
        \fhom^1(\kappa(x),\push f\ul\Z)\too\fhom^1(\meG_x,\ul\Z)\too\fhom^0(\kappa(x),\fhomR^1\push f\ul\Z).
    \end{equation}
    It suffices to show that both of the outer terms in \cref{es:leray-H1GZ} vanish. First, $\push f\ul\Z\simeq\ul\Z$ since $f$ is a universal homeomorphism \cite[\href{https://stacks.math.columbia.edu/tag/06R9}{Tag 06R9}]{stacks-project}, so the leftmost term in \cref{es:leray-H1GZ} is 
    \[\fhom^1(\kappa(x),\push f\ul\Z)\simeq\fhom^1(\kappa(x),\ul\Z)\simeq\hom^1(\kappa(x),\ul\Z)\simeq\ctsHom(G_{\kappa(x)},\Z)=0.\]
    For the rightmost term in \cref{es:leray-H1GZ}, we claim that $\fhomR^1\push f\ul\Z=0$, which suffices to finish the proof. It suffices to check this at the level of fppf stalks, so it suffices to prove that $\fhom^1(\meG,\ul\Z)=0$ when $\meG=\meG_x\by_{\kappa(x)}A$ for some affine $(\Sch_{\kappa(x)}$,fppf$)$-local scheme $\spec A$ \cite[Remark 1.8 and Theorem 2.3]{AG-points}. Since $\spec A$ is fppf-local \cite[Definition 0.1]{AG-points}, it follows that $\meG\to A$ admits a section, so $\meG$ is a neutral gerbe over $\spec A$. Since $\meX$ has finite inertia by assumption, we conclude that $\meG\simeq BG_A$ for some finite $A$-group scheme $G$. Thus, we are reduced to proving that $\fhom^1(BG_A,\ul\Z)=0$. For this, we remark that a $\Z$-torsor on $BG$ is determined by the data of a $\Z$-torsor on $\spec A$ equipped with an action of $G$, i.e. with a homomorphism $G\to\ul\Z$. Since $A$ is fppf-local, every $\Z$-torsor on it is trivial and since $G$ is finite, every homomorphism $G\to\ul\Z$ is trivial. Therefore, $\fhom^1(BG_A,\ul\Z)=0$, completing the proof.
\end{proof}
\begin{cor}
    $\hom^1(\meX,\push j\G_{m,\meG})\simeq\coker\p{\bigoplus_{x\in\abs\meX_1}\Z\too\Pic\meX}$
\end{cor}
\begin{proof}
    This follows from the cohomology exact sequence associated to \cref{ses:divisor-stack}.
\end{proof}
In words, the above corollary says that every line bundle on $\meX$ comes from a divisor if and only if $\hom^1(\meX,\push j\G_m)=0$. 
\begin{ex}\label{ex:CL=Pic-gen-schemey}
    If $\meX$ is `generically scheme-y' in the sense that its generic residual gerbe $\meG$ is (the spectrum of) a field $K$, then $\hom^1(\meX,\push j\G_m)=0$; indeed, $\hom^1(\meX,\push j\G_m)\into\hom^1(\meG,\G_m)=\Pic K=0$. This shows that, for such $\meX$, $\Pic\meX$ is the Weil divisor class group, reproving a result of Voight and Zureick-Brown \cite[Lemma 5.4.5]{dzb-voight}.
\end{ex}
We next show that \cref{ex:CL=Pic-gen-schemey} is, in some sense, almost the only case in which the Picard group of a stacky curve is equal to its divisor class group. More specifically, we show that this holds if and only if $\Pic\meG=0$.
\begin{lemma}\label{lem:H1-push-j}
    $\lisethomR^1\push j\G_m=0$. Consequently, $\hom^1(\meX,\push j\G_{m,\meG})=\hom^1(\meG,\G_m)=\Pic\meG$.
\end{lemma}
\begin{proof}
    To show that $\lehomR^1\push j\G_m=0$, it suffices to show that for any smooth cover $S\to\meX$ by a (necessarily regular) scheme $S$ and any line bundle $\msL$ on the generic fiber $S_\meG$, there exists an \'etale cover $S'\onto S$ such that the pullback $\msL\vert_{S'_\meG}$ is trivial. With that said, let $S\to\meX$ and $\msL$ be as just indicated. Note that $\meG\into\meX$, being a monomorphism into a noetherian stack, is both quasi-compact and quasi-separated (qcqs). Consequently, $\iota\colon S_\meG\into S$ is also qcqs, so $\push\iota\msL$ is a quasi-coherent sheaf on $S$ \cite[\href{https://stacks.math.columbia.edu/tag/03M9}{Tag 03M9}]{stacks-project}. Since $S$ is regular, it follows that there exists a line bundle $\msL'\subset\push\iota\msL$ such that $\msL'\vert_{S_\meG}=\msL$. Thus, we may take $S'\onto S$ to be any \'etale cover trivializing $\msL'$. This proves $\lehomR^1\push j\G_m=0$.

    Now, the exact sequence of low-degree terms for the Leray spectral sequence associated to $j\colon\meG\into\meX$ begins
    \[0\too\hom^1(\meX,\push j\G_{m,\meG})\too\hom^1(\meG,\G_m)\too\hom^0(\meX,\lehomR^1\push j\G_m)=0,\]
    from which the rest of the claim follows.
\end{proof}
\begin{cor}\label{cor:Pic-Cl}
    Let 
    \[\Cl\meX\coloneqq\coker\p{\G_m(\meG)\too\bigoplus_{x\in\abs\meX_1}\Z}\]
    be the divisor class group of $\meX$. Then, there is a short exact sequence
    \[0\too\Cl\meX\too\Pic\meX\too\Pic\meG\too0.\]
\end{cor}
\begin{proof}
    This follows from considering the cohomology exact sequence associated to \cref{ses:divisor-stack} and appealing to \cref{lem:H^1XZx=0,lem:H1-push-j}.
\end{proof}
\begin{rem}
    There always exists some generically scheme-y stacky curve $\meY$ under $\meX$ such that $\meX\to\meY$ is a gerbe; see \cite[Proposition 1.5]{lopez2023picard} and \cite[Appendix A]{AOV}. It is possible to identify $\Cl\meX$ in \cref{cor:Pic-Cl} with $\Pic\meY$ and so obtain a short exact sequence
    \[0\too\Pic\meY\too\Pic\meX\too\Pic\meG\too0.\]
    This is similar to the exact sequence appearing in \cite[Corollary 1.7]{lopez2023picard}. A short exact sequence exactly of the form appearing above was first suggested to me by David Zureick-Brown.
\end{rem}
Focusing back in on Brauer groups, taking the long exact sequence in cohomology associated to \cref{ses:divisor-stack} and appealing to \cref{lem:H^1XZx=0} produces the exact sequence
\[0\too\hom^2(\meX,\G_m)\too\hom^2(\meX,\push j\G_{m,\meG})\too\bigoplus_{x\in\abs\meX_1}\hom^2(\meG_x,\Z)\too\hom^3(\meX,\G_m)\too\dots.\]
Above, one can further compute $\hom^2(\meG_x,\Z)$.
\begin{lemma}\label{lem:Q-coh-van}
    $\hom^i(\meG_x,\Q)=0$ for $i=1,2$ and $x\in\abs\meX_1$. Consequently, $\hom^2(\meG_x,\Z)\simeq\hom^1(\meG_x,\Q/\Z)$.
\end{lemma}
\begin{proof}
    Once the first part of the lemma is established, the second follows from the exact sequence $0\to\Z\to\Q\to\Q/\Z\to0$. For the first part, write $f\colon\meG_x\to\spec k$ for the morphism realizing $f$ as a gerbe over (the spectrum of) a field $k$. One can argue as in \cref{lem:H^1XZx=0} (i.e. use the fppf-Leray spectral sequence for $f$) to show that $\hom^1(\meG_x,\Q)=0$ and that $\fhomR^1\push f\ul\Q=0$. Since $\fhom^2(k,\push f\ul\Q)=\hom^2(k,\Q)=0$, to prove that $\hom^2(\meG_x,\Q)=0$, it now suffices to show that $\hom^0(k,\fhomR^2\push f\ul\Q)=0$. Let $k'/k$ be a field extension over which $\meG_x$ admits a point, so $\meG_{x,k'}\simeq BG$ for some finite $k'$-group scheme $G$. By replacing $k'$ with a further finite extension if necessary, we may and do assume that the maximal \'etale quotient $G^\tet$ of $G$ is a constant group scheme. To show $\hom^0(k,\fhomR^2\push f\ul\Q)=0$, it suffices to show that $\fhom^2(\meG_{x,k'},\Q)=\hom^2(BG,\Q)$ vanishes. For this, one uses the descent spectral sequence
    \[E_2^{ij}=\ghom^i(G,\ul\hom^j(k'/k',\Q))\implies\hom^{i+j}(BG,\Q).\]
    Note that $E_2^{02}=0$ since $\hom^2(k',\Q)=0$, and that $E_2^{11}=\Hom(G,\ul\hom^1(k'/k',\Q))=0$ since there are no non-trivial homomorphisms from a finite group to a $\Q$-vector space. Finally, because $\ul\Q$ is \'etale, every morphism $G^n\to\ul\Q$ factors through $(G^\tet)^n$, so $E_2^{20}=\ghom^2(G,\Q)=\ghom^2(G^\tet,\Q)$ which vanishes since $G^\tet$ is a finite (abstract) group.
\end{proof}
\begin{prop}\label{prop:residue-stacky-curve}
    There is a long exact sequence
    \begin{equation}\label{les:residue-stacky-curve}
        0\too\hom^2(\meX,\G_m)\too\hom^2(\meX,\push j\G_{m,\meG})\too\bigoplus_{x\in\abs\meX_1}\hom^1(\meG_x,\Q/\Z)\too\hom^3(\meX,\G_m)\too\dots.
    \end{equation}
    Furthermore, $\hom^2(\meX,\push j\G_{m,\meG})\simeq\ker\p{\hom^2(\meG,\G_m)\too\hom^0(\meX,\lehomR^2\push j\G_m)}$.
\end{prop}
\begin{proof}
    Taking the long exact sequence associated to \cref{ses:divisor-stack}, and appealing to \cref{lem:H^1XZx=0,lem:Q-coh-van}, we obtain the exact sequence
    \cref{les:residue-stacky-curve}. The last part of the claim follows from the exact sequence 
    \[\hom^0(\meX,\lehomR^1\push j\G_m)\to\hom^2(\meX,\push j\G_m)\to\ker\p{\hom^2(\meG,\G_m)\to\hom^0(\meX,\lehomR^2\push j\G_m)}\to\hom^1(\meX,\lehomR^1\push j\G_m)\]
    of low degree terms in the lisse-\'etale Leray spectral sequence. Above, $\lehomR^1\push j\G_m=0$ by \cref{lem:H1-push-j}.
\end{proof}
This shows that stacky curves enjoy a residue exact sequence \cref{les:residue-stacky-curve} reminiscent of the one attached to schemey curves; see \cite[Theorem 3.6.1]{CT-S}. In the following section, after adding the additional assumption that $\meX$ is generically schemey, we compare this residue sequence to the analogous one on its coarse space.

\subsection{The tame generically schemey case}\label{sect:res-gen-schemey}
\begin{assump}
    Assume now that there exists distinct closed points $x_1,\dots,x_r\in X$ and positive integers $e_1,\dots,e_r>1$ such that
    \begin{equation}\label{iteratated-root-form}
        \meX\simeq\sqrt[e_1]{x_1/X}\by_X\dots\by_X\sqrt[e_r]{x_r/X}.
    \end{equation}
    In particular, $\meX$ is tame and `generically schemey' in the sense that there exists a dense open $U\subset X$ (e.g. $U=X\sm\{x_1,\dots,x_r\}$) over which the coarse space map $c\colon\meX\to X$ is an isomorphism, but is not necessarily DM.

    It is known that $\meX$ will be of the form \cref{iteratated-root-form} if it is a tame, DM, and generically schemey; see \cite[Theorem 1]{GS} and/or \cite[2.24]{lopez2023picard}.
\end{assump}
\begin{rem}
    Recall that \cref{lem:Br=Br'-stacky-curves} shows that $\Br\meX\simeq\hom^2(\meX,\G_m)$.
\end{rem}
\begin{notn}
    The generic residual gerbe $j\colon\meG\into\meX$ is now the spectrum of a field, so we write $\meG=\spec K$. We write $\eta\colon\spec K\to X$ for the corresponding map to $X$, so we have a commutative diagram
    \[\compdiag{\spec K}j\meX c{X.}\eta\]
    For a closed point $x\in\abs X_1=\abs\meX_1$, we write $\iota_x\colon\spec\kappa(x)\into X$ for the corresponding closed immersion, we have a commutative (but not necessarily Cartesian) diagram
    \begin{equation}\label{cd:closed-pt-comp}
        \commsquare{\meG_x}{i_x}\meX{}c{\spec\kappa(x)}{\iota_x}{X.}
    \end{equation}
    Furthermore, write $\ul\Z_x\coloneqq\iota_{x,*}\ul\Z$.
\end{notn}
We wish to compare residue maps on $\meX$ to those on $X$. A priori the exact sequences \cref{ses:divisor-stack} for $\meX$ and $X$ live in different categories. In order to remedy this, we first apply $\push c$ to the exact sequence \cref{ses:divisor-stack} over $\meX$.
\begin{rec}
    While we use the lisse-\'etale site over the stack $\meX$, we use the small \'etale site over the scheme $X$.
\end{rec}
\begin{lemma}\label{lem:R1c(jGm)}
    $\lehomR^1\push c\p{\push j\G_{m,K}}=0$
\end{lemma}
\begin{proof}
    Let $U\to X$ be an \'etale $X$-scheme. Consider the following commutative diagram, all of whose squares are Cartesian:
    \[\begin{tikzcd}
        U_K\ar[d, "\tet"]\ar[r, "j'"]&\meU\ar[d]\ar[r]&U\ar[d, "\tet"]\\
        \spec K\ar[r, "j"]&\meX\ar[r, "c"]&X.
    \end{tikzcd}\]
    In order to prove the theorem, it suffices to show the second equality in
    \[\hom^1(U\by_X\meX,(\push j\G_{m,K})\vert_U)=\hom^1(\meU,\push j'\G_{m,U_K})=0.\]
    For this, we remark that $\hom^1(\meU,\push j'\G_{m,U_K})\into\hom^1(U_K,\G_m)=\Pic U_K$ and that $\Pic U_K=0$ since $U_K$ is \'etale over a field.
\end{proof}
\begin{cor}
    The sequence \cref{ses:divisor-stack} pushes forward to the exact sequence
    \begin{equation}\label{es:push-divisor-ses}
        0\too\G_{m,X}\too\push\eta\G_{m,K}\too\bigoplus_{x\in\abs X_1}\ul\Z_x\too\lehomR^1\push c\G_{m,\meX}\too0
    \end{equation}
    of sheaves on $X$.
\end{cor}
\begin{proof}
    Given \cref{lem:R1c(jGm)}, the only part of this claim which is not immediate is the computation $\ppush{c\circ i_x}\ul\Z=\iota_{x,*}\ul\Z\eqqcolon\ul\Z_x$ for $x\in\abs X_1=\abs\meX_1$. Write $f\colon\meG_x\to\spec\kappa(x)$ for the natural morphism. By \cref{cd:closed-pt-comp}, to prove $\ppush{c\circ i_x}\ul\Z=\ul\Z_x$, it suffices to show that $\push f\ul\Z=\ul\Z$.

    For this, set $\meZ\coloneqq\kappa(x)\by_X\meX$ and note that $\meZ\to\spec\kappa(x)$ is a universal homeomorphism which follows, e.g., from \cite[Proposition 5.2.9(3)]{Alper-adequate}. In particular, its underlying topological space $\abs\meZ$ is a singleton, and so \cite[\href{https://stacks.math.columbia.edu/tag/06MT}{Tag 06MT}]{stacks-project} implies that $\red\meZ$ is the residual gerbe $\meG_x$ of $\meX$ at $x$. In particular, $f$ factors as $\meG_x\simeq\red\meZ\into\meZ\to\spec\kappa(x)$, making it a universal homeomorphism. This implies that $\push f\ul\Z=\ul\Z$, as desired.
\end{proof}
Note that \cref{es:push-divisor-ses} sits in the following homomorphism of exact sequences:
\begin{equation}\label{hom-es:div}
    \begin{tikzcd}
        0\ar[r]&\G_{m,X}\ar[r]\ar[d, equals]&\push\eta\G_{m,K}\ar[r, "\Delta_X"]\ar[d, equals]&\bigoplus_{x\in\abs X_1}\ul\Z_x\ar[d, "\phi"]\ar[r]&0\ar[d]\ar[r]&0\\
        0\ar[r]&\G_{m,X}\ar[r]&\push\eta\G_{m,K}\ar[r, "\Delta_\meX"]&\bigoplus_{x\in\abs\meX_1}\ul\Z_x\ar[r]&\lehomR^1\push c\G_{m,\meX}\ar[r]&0,
    \end{tikzcd}
\end{equation}
whose top row is the usual divisor exact sequence on the scheme $X$ and whose bottom row is \cref{es:push-divisor-ses}.
\begin{notn}\label{notn:ex}
    Recall the inertial degree $\Ideg$ of \cite[Section 2.2]{stacky-CW}. For a point $x\in\abs X_1=\abs\meX_1$, set
    \[e_x\coloneqq\Ideg(\meG_x\to\spec\kappa(x))=\twocases{e_i}{i\in\{1,\dots,r\}\tand x=x_i}1\]
\end{notn}
\begin{rem}
    The results in this section have so far only used that $\meX$ is generically schemey. We use that $\meX$ is specifically a multiply rooted stack in order to compute $\phi$ in the next lemma.
\end{rem}
\begin{lemma}
    The morphism $\phi$ in \cref{hom-es:div} is $\phi=\bigoplus\phi_x$, where $\phi_x\colon\ul\Z_x\to\ul\Z_x$ sends $1\mapsto e_x$.
\end{lemma}
\begin{proof}
    By construction, one can compute $\phi$ separately near each point, so it suffices to assume that $X=\spec R$ is the spectrum of a dvr with closed point $x$ (think: $R=\msO_{X,x}$). Since $\meX$ is a root stack over $X$, it follows from \cref{prop:root-stack-summary}\bp2 that
    \[\meX=\sqrt[n]{\pi/R}\simeq\sq{\frac{\spec R[T]/(T^n-\pi)}{\mu_n}},\]
    where $n\coloneqq e_x$ and $\pi$ is a uniformizer for $R$. Let $K=\Frac R$, so we are attempting to compute the map $\phi=\phi_x$ in the commutative diagram
    \[\commsquare{\units K}{\Delta_X}\Z={\phi_x}{\units K}{\Delta_\meX}\Z\]
    with horizontal maps as in \cref{prop:stacky-divisor-ses}. Note that $\Delta_X(\pi)=1$, so it suffices to show that $\Delta_\meX(\pi)=n$. From the proof of \cref{prop:stacky-divisor-ses}, we see that $\Delta_\meX$ can be computed on a smooth cover of $\meX$, so consider the Cartesian square\footnote{this is the top face of the Cartesian cube in \cite[Proof of Proposition 2.3.5]{cadman} obtained from pulling back the following Cartesian square (which is the bottom face of the aforementioned cube)
    \[\commsquare{\A^1}{(-)^n}{\A^1}{}{}{[\A^1/\G_m]}{(-)^n}{[\A^1/\G_m]}\]
    along $\spec R\xto{(\msO_R,\pi)}[\A^1/\G_m]$.}
    \[\commsquare{\G_{m,R}}{}\meX{(-)^n}{}{\G_{m,R}}{}{\spec R}\]
    whose horizontal arrows are (smooth) $\G_m$-torsors. With this square in mind, it becomes clear that $\Delta_\meX(\pi)=n\Delta_X(\pi)=n$ since the rational function $\pi\in K=\Frac R$ pulls back to in the top $\G_{m,R}$ is the $n$th power of the function it pulls back to in the bottom $\G_{m,R}$.
\end{proof}
\begin{cor}\label{cor:alt-R1Gm-root}
    $\lehomR^1\push c\G_{m,\meX}\simeq\coker\phi\simeq\bigoplus_{x\in\abs X_1}\ul{\zmod{e_x}}_x$.
\end{cor}
In order to better understand \cref{es:push-divisor-ses}, we introduce the two-term complexes
\[\msC_\meX \coloneqq\sq{\push\eta\G_{m,K}\xto{\Delta_\meX}\bigoplus_{x\in\abs\meX_1}\ul\Z_x}\tand\msC_X \coloneqq\sq{\push\eta\G_{m,K}\xto{\Delta_X}\bigoplus_{x\in\abs X_1}\ul\Z_x}\simeq\G_{m,X},\]
both concentrated in degrees $0$ and $1$.
Above, $\msC_X$ is quasi-isomorphic to $\G_{m,X}$ because of the top row of \cref{hom-es:div}. For $\msC_\meX$, we have the following result.
\begin{prop}\label{prop:C-trunc 1}
    $\msC_\meX\simeq\tau_{\le2}\lehomR\push c\G_m$. In particular,
    \[\hom^n(X,\msC_\meX)\simeq\hom^n(\meX,\G_m)\tfor n\le2.\]
\end{prop}
\begin{proof}
    The exact sequence \cref{ses:divisor-stack} shows that $\G_{m,\meX}\simeq[\push j\G_m\to\bigoplus\ul\Z_{\meG_x}]$ so applying $\push c$ gives a morphism
    \[\msC_\meX=\sq{\push c\push j\G_m\too\push c\p{\bigoplus_{x\in\abs\meX_1}\ul\Z_{\meG_x}}}\to\lehomR\push c\sq{\push j\G_m\to\bigoplus_{x\in\abs\meX_1}\ul\Z_{\meG_x}}\simeq\lehomR\push c\G_m.\]
    Furthermore, \cref{es:push-divisor-ses} shows that this morphism induces an isomorphism on $\pi_0$ and $\pi_1$ (while $\pi_i(\msC)=0$ for $i\ge2$) and so induces a (quasi-)isomorphism $\msC\iso\tau_{\le1}\lehomR\push c\G_m$. Finally,
    \[\pi_2\p{\lehomR\push c\G_m}=\lehomR^2\push c\G_m=0,\]
    with last equality by \cref{thm:R2Gm-vanishes}, so $\tau_{\le1}\lehomR\push c\G_m\simeq\tau_{\le2}\lehomR\push c\G_m$.
\end{proof}
Now, we remark that \cref{hom-es:div} induces the following morphism of distinguished triangles:
\[\begin{tikzcd}
    &\G_{m,X}\ar[d, "\sim" sloped]\\
    \p{\bigoplus_{x\in\abs X_1}\ul\Z_x}[-1]\ar[d, "\phi"]\ar[r]&\msC_X\ar[d]\ar[r]&\push\eta\G_{m,K}\ar[r, "+1"]\ar[d, equals]&\phantom.\\
    \p{\bigoplus_{x\in\abs\meX_1}\ul\Z_x}[-1]\ar[r]&\msC_\meX\ar[r]\ar[d, "\sim" sloped]&\push\eta\G_{m,K}\ar[r, "+1"]&\phantom.\\
    &\tau_{\le2}\lehomR\push c\G_{m,\meX}
\end{tikzcd}\]
Therefore, upon taking cohomology, we obtain:
\begin{equation}\label{cd:root-res-comp-gen}
    \begin{tikzcd}
        0\ar[r]&\hom^2(X,\G_m)\ar[d, "\pull c"]\ar[r]&\hom^2(X,\push\eta\G_m)\ar[d, equals]\ar[r, "\bigoplus_x\res_{x/X}"]&\bigoplus_{x\in\abs X_1}\hom^1(\kappa(x),\Q/\Z)\ar[d, "\push\phi"]\\
        0\ar[r]&\hom^2(\meX,\G_m)\ar[r]&\hom^2(X,\push\eta\G_m)\ar[r, "\bigoplus_x\res_{x/\meX}"]&\bigoplus_{x\in\abs\meX_1}\hom^1(\kappa(x),\Q/\Z)
    \end{tikzcd}
\end{equation}
whose top row is the usual residue exact sequence associated to $X$.
\begin{cor}
    For any $x\in\abs X_1=\abs\meX_1$, $\res_{x/\meX}=e_x\cdot\res_{x/X}$ as maps $\hom^2(X,\push\eta\G_m)\to\hom^1(\kappa(x),\Q/\Z)$.
\end{cor}
\begin{proof}
    This follows from commutativity of \cref{cd:root-res-comp-gen}.
\end{proof}
Finally, we make $\hom^2(X,\push\eta\G_m)$ more explicit in certain cases. The below proposition is likely well-known, but I failed to find a statement of it in the literature.
\begin{prop}\label{prop:H2-push-eta}
    Assume that $X$ is a nice curve over a field $k$. Then, $\hom^2(X,\push\eta\G_m)\simeq\ker\p{\Br k(X)\to\Br k^s(X)}$.
\end{prop}
\begin{proof}
    We begin with the Leray spectral sequence for the morphism $X\to\spec k$:
    \begin{equation}\label{ss:Xk}
        E_2^{ij}=\hom^i(k,\hom^j(\bar X,\push\eta\G_m))\implies\hom^{i+j}(X,\push\eta\G_m),
    \end{equation}
    where $\bar X\coloneqq X_{k^s}$. Note that $\hom^1(\bar X,\push\eta\G_m)=0$ because it embeds by $\hom^1(k^s(X),\G_m)=0$. Furthermore, on $\bar X$, one has the exact sequence
    \[0\too\hom^2(\bar X,\G_m)\too\hom^2(\bar X,\push\eta\G_m)\too\bigoplus_{x\in\abs{\bar X}_1}\hom^1(\kappa(x),\qz).\]
    Above, $\hom^2(\bar X,\G_m)=0$ by \cite[Corollaire 5.8]{Gr3} and $\hom^1(\kappa(x),\Q/\Z)=0$ for all $x\in\abs{\bar X}_1$ since each $\kappa(x)$ is separably closed. Thus, $\hom^2(\bar X,\push\eta\G_m)=0$ as well. Thus, \cref{ss:Xk} gives an isomorphism $\hom^2(k,\units{k^s(X)})\iso\hom^2(X,\push\eta\G_m)$, so we need only compute this former group. For this, we use the spectral sequence
    \[F_2^{ij}=\hom^i(k,\hom^j(k^s(X),\G_m))\implies\hom^{i+j}(k(X),\G_m)\]
    whose exact sequence of low degree terms quickly yields $\hom^2(k,k^s(X)^\by)\iso\ker\p{\Br k(X)\to\Br k^s(X)}$.
\end{proof}

\subsection{A stacky Faddeev exact sequence}\label{sect:faddeev}
\begin{rem}
    In \cite[Proposition 5.5]{santens2023brauermanin}, Santens gives one description of the Brauer group of a generically scheme-y tame DM stacky $\P^1$. In the present section, we give a different description of such Brauer groups (see \cref{prop:stacky-P1-Brauer,prop:stacky-Faddeev}), and then show how it can be used to aid in the computations of Brauer groups of (possibly non-proper) stacky rational curves.
\end{rem}
\begin{set}
    Fix a field $k$, set $X=\P^1_k$, and use the notation of \cref{sect:res-gen-schemey}. In particular, there are distinct closed points $x_1,\dots,x_r\in\P^1_k$ and integers $e_1,\dots,e_r>1$ such that
    \[\meX=\sqrt[e_1]{x_1/\P^1_k}\by_{\P^1_k}\dots\by_{\P^1_k}\sqrt[e_r]{x_r/\P^1_k}\xtoo c\P^1_k.\]
    Furthermore, for each $x\in\abs{\P^1_k}_1=\abs{\meX}_1$, we defined its ``stacky degree'' $e_x$ in \cref{notn:ex}. We also have the morphisms $\phi_x\colon\ul\Z_x\to\ul\Z_x$, of sheaves supported on $\spec\kappa(x)\into X$, given by $1\mapsto e_x$.
\end{set}
Note that, by \cref{prop:H2-push-eta} and \cite[Theorem 5.1]{berg2023brauermaninobstructionsrequiringarbitrarily}, the top row of \cref{cd:root-res-comp-gen} extends to the Faddeev exact sequence and so \cref{cd:root-res-comp-gen} yields the following commutative diagram with exact rows (note also that $\Br\P^1_k\simeq\Br k$ either as a biproduct of Faddeev or by \cite[Propostion 6.9.9]{poonen-rat-pts}):
\begin{equation}\label{cd:faddeev-comp}
    \begin{tikzcd}
        0\ar[r]&\Br\P^1_k\ar[d, "\pull c"]\ar[r]&\ker(\Br k(\P^1)\to\Br k^s(\P^1))\ar[d, equals]\ar[r, "\bigoplus_x\res_{x/\P^1_k}"]&\bigoplus_{x\in\abs{\P^1_k}_1}\hom^1(\kappa(x),\Q/\Z)\ar[d, "\push\phi=\bigoplus_xe_x"]\ar[r, "\Cor"]&\hom^1(k,\Q/\Z)\ar[r]&0\\
        0\ar[r]&\Br\meX\ar[r]&\ker\p{\Br k(\P^1)\to\Br k^s(\P^1)}\ar[r, "\bigoplus_x\res_{x/\meX}"]&\bigoplus_{x\in\abs\meX_1}\hom^1(\kappa(x),\Q/\Z),
    \end{tikzcd}
\end{equation}
where $\Cor=\sum_x\Cor_{\kappa(x)/k}$.
\begin{prop}\label{prop:stacky-P1-Brauer}
    There is an exact sequence
    \[0\too\underbrace{\Br\P^1_k}_{\simeq\Br k}\xtoo{\pull c}\Br\meX\too\bigoplus_{i=1}^r\hom^1\p{\kappa(x_i),\left.\frac1{e_i}\Z\right/\Z}\xtoo\Cor\hom^1(k,\Q/\Z)\]
\end{prop}
\begin{proof}
    From a diagram chase or appropriate application of the snake lemma to \cref{cd:faddeev-comp}, one sees that
    \[\coker\pull c\simeq\ker\p{\bigoplus_{x\in\abs{\P^1_k}_1}\hom^1(\kappa(x),\Q/\Z)[e_x]\xtoo\Cor\hom^1(k,\Q/\Z)}.\]
    Since $e_x=1$ unless $x=x_i$ for some $i$ (in which case $e_x=e_i$), it now suffices to show that $\hom^1(\kappa(x_i),\Q/\Z)[e_i]\simeq\hom^1(\kappa(x_i),\frac1{e_i}\Z/\Z)$. This follows from taking cohomology of the exact sequence $0\to\frac1{e_i}\Z/\Z\too\Q/\Z\xtoo{e_i}\Q/\Z\too0$.
\end{proof}
\begin{ex}
    Fix integers $a,b,c>0$ and consider the generalized Fermat equation $x^a+y^b=z^c$. To slightly simplify this example, assume that $a$ or $b$ is odd. Primitive integral solutions to this equation correspond to integral points on $S=\spec\Z[x,y,z]/(x^a+y^b-z^c)\sm\{(0,0,0)\}\subset\A^3_\Z$. Note that $\G_m^3\actson\A^3$ and let $H\subset\G_m^3$ be the subgroup preserving $S$; this is the subgroup generated by the image of $\G_m\into\G_m^3,\lambda\mapsto(\lambda^{bc},\lambda^{ac},\lambda^{ab})$ and by $\mu_a\by\mu_b\by\mu_c\subset\G_m^3$. Set $\meX\coloneqq[S/H]$. One often studies $S(\Z)$ by first studying $\meX(\Z)$; see \cite{faltings-plus-epsilon,faltings-plus-epsilon:correction} and/or \cite{PSS:X(7)}. One can check that $\meX$ is isomorphic to $\P^1$ rooted at $0,\infty,-1$ to degrees $a,b,c$, respectively; the coarse space map $\meX\to\P^1$ is descended from the map $S\to\P^1,(x,y,z)\mapsto[x^a:y^b]$. With this description of $\meX$, \cref{prop:stacky-P1-Brauer} shows that the Brauer group of its generic fiber sits in an exact sequence
    \[0\too\Br\Q\too\Br\meX_\Q\too\hom^1\p{\Q,\frac1a\Z/\Z}\oplus\hom^1\p{\Q,\frac1b\Z/\Z}\oplus\hom^1\p{\Q,\frac1c\Z/\Z}\xtoo\Sigma\hom^1(\Q,\Q/\Z).\qedhere\]
\end{ex}

In the simplest case, where $\P^1$ is rooted at $k$-points to pairwise coprime degrees, \cref{cd:faddeev-comp} can be reduced into a single stacky Faddeev exact sequence.
\begin{prop}[Stacky Faddeev]\label{prop:stacky-Faddeev}
    Assume that each $x_1,\dots,x_r$ is a $k$-point and that the numbers $e_1,\dots,e_r$ are pairwise coprime. Let $N\coloneq\prod_{i=1}^re_i=\lcm_i(e_i)$. Then, $\Br\meX\simeq\Br\P^1_k\simeq\Br k$ and there is an exact sequence
    \[0\to\Br k\to\ker\p{\Br k(\P^1)\to\Br k^s(\P^1)}\xtoo{\bigoplus_x e_x\cdot\res_{x/\P^1_k}}\bigoplus_{x\in\abs\meX_1}\hom^1(\kappa(x),\Q/\Z)\xtoo{\Cor_\meX}\hom^1(k,\Q/\Z)\to0,\]
    where $\Cor_\meX\coloneqq\sum_{x\in\abs\meX_1}\frac N{e_x}\cdot\Cor_{\kappa(x)/k}$.
\end{prop}
\begin{proof}
    Exactness at $\Br k$ is clear. To ease notation below, for any $i\in\{1,\dots,r\}$, we set $N_i\coloneqq N/e_i\in\N$ and for any $x\in\abs\meX_1$, we set $N_x\coloneqq N/e_x\in\N$.
    \begin{itemize}
        \item Exactness at $\ker\p{\Br k(\P^1)\to\Br k^s(\P^1)}$.

        By exactness of the top row in \cref{cd:faddeev-comp}, it suffices to show that, given some $\alpha=(\alpha_x)_{x\in\abs\meX_1}\in\ker\Cor$ such that $e_x\cdot\alpha_x=0$ for all $x$, one must in fact have that $\alpha_x=0$ for all $x$. Suppose we have such an $\alpha$. We first observe that $\alpha_x=0$ if $x\not\in\{x_1,\dots,x_r\}$ (since then $e_x=1$), so
        \[\sum_{i=1}^r\alpha_{x_i}=-\sum_{x\not\in\{x_1,\dots,x_r\}}\Cor_{\kappa(x)/k}\alpha_x=0,\]
        using that each $x_i$ is a $k$-point. For any fixed $j\in\{1,\dots,r\}$, this implies that
        \[\alpha_{x_j}=-\sum_{\substack{i=1\\i\neq j}}^r\alpha_{x_i}\in\hom^1(k,\Q/\Z)[N_j],\]
        using that the $e_i$'s are pairwise coprime. Since $\alpha_{x_j}$ is $e_j$-torsion by assumption and $\gcd(e_j,N_j)=1$, we conclude that $\alpha_j=0$, proving exactness.

        We remark that exactness here proves that $\Br\meX\simeq\Br k$ since both groups are the kernel of $\bigoplus_xe_x\cdot\res_{x/\P^1_k}=\bigoplus_x\res_{x/\meX}$.

        \item Exactness at $\bigoplus_{x\in\abs\meX_1}\hom^1(\kappa(x),\Q/\Z)$.

        Suppose we are given some $\alpha=(\alpha_x)_{x\in\abs\meX_1}\in\ker\Cor_\meX$. We need to show there exists some $\beta=(\beta_x)_{x\in\abs\meX_1}\in\ker\Cor$ such that $\alpha_x=e_x\cdot\beta_x$ for all $x$. For any $x\not\in\{x_1,\dots,x_r\}$ we (necessarily) set $\beta_x=\alpha_x$. For any $j\in\{1,\dots,r\}$, choose integers $a_j,b_j$ such that $a_jN_j+1=b_je_j$ and set
        \[\beta_{x_j}\coloneqq b_j\alpha_{x_j}+a_j\sum_{x\neq x_j}\frac {N_x}{e_j}\Cor_{\kappa(x)/k}\alpha_x\in\hom^1(k,\Q/\Z).\]
        Then,
        \[e_j\cdot\beta_{x_j}=(a_jN_j+1)\alpha_{x_j}+a_j\sum_{x\neq x_j}N_x\alpha_x=\alpha_{x_j}+a_j\Cor_\meX(\alpha)=\alpha_{x_j}.\]
        Hence, we only need verify that $\beta=(\beta_x)_{x\in\abs\meX_1}$ satisfies $\sum_x\Cor_{\kappa(x)/k}\beta_x=0$. At this point we note that, by construction, $\sum_{i=1}^ra_iN_i$ is congruent to $-1\pmod{e_j}$ for all $j$ and so must be congruent to $-1\pmod N$ as well. Write $1+\sum_{i=1}^ra_iN_i=mN$ for some $m\in\Z$. Finally,
        \begin{align*}
            \sum_x\Cor_{\kappa(x)/k}\beta_x
            &=\sum_{i=1}^r\p{b_i\alpha_{x_i}+a_i\sum_{x\neq x_i}\frac{N_x}{e_i}\Cor_{\kappa(x)/k}\alpha_x}+\sum_{x\not\in\{x_1,\dots,x_r\}}\Cor_{\kappa(x)/k}\alpha_x \\
            &= \sum_{i=1}^r\p{b_i+\sum_{\substack{j=1\\j\neq i}}^r\frac{a_jN_i}{e_j}}\alpha_{x_i}+\sum_{x\not\in\{x_1,\dots,x_r\}}\p{1+\sum_{i=1}^ra_iN_i}\Cor_{\kappa(x)/k}\alpha_x\\
            &= \sum_{i=1}^r\p{\frac{a_iN_i+1}{e_i}+\sum_{\substack{j=1\\j\neq i}}^r\frac{a_jN_j}{e_i}}\alpha_{x_i}+\sum_{x\not\in\{x_1,\dots,x_r\}}\p{1+\sum_{i=1}^ra_iN_i}\Cor_{\kappa(x)/k}\alpha_x\\
            &=\sum_{i=1}^r\frac{1+\sum_{j=1}^ra_jN_j}{e_i}\alpha_{x_i}++\sum_{x\not\in\{x_1,\dots,x_r\}}\p{1+\sum_{i=1}^ra_iN_i}\Cor_{\kappa(x)/k}\alpha_x\\
            &=m\sum_{i=1}^rN_i\cdot\alpha_{x_i}+m\sum_{x\not\in\{x_1,\dots,x_r\}}N\cdot\Cor_{\kappa(x)/k}\alpha_x\\
            &=m\Cor_\meX(\alpha)\\
            &=0.
        \end{align*}
        \item Exactness at $\hom^1(k,\Q/\Z)$.

        Fix some $\sigma\in\hom^1(k,\Q/\Z)$. Because $\Cor$ is surjective, there exists some $\alpha=(\alpha_x)_x\in\bigoplus_x\hom^1(\kappa(x),\Q/\Z)$ such that $\sum_x\Cor_{\kappa(x)/k}\alpha_x=\sigma$. If $r=0$, then $\Cor_\meX(\alpha)=\Cor(\alpha)=\sigma$ and we win, so assume $r\ge1$. Then, $(N_1,N_2,\dots,N_r)=(1)$ so there exists integers $a_1,\dots,a_r$ such that $a_1N_1+\dots+a_rN_r=1-N$. Set $\beta=(\beta_x)_x$ where
        \[\beta_x=\Twocases{\alpha_x}{x\not\in\{x_1,\dots,x_r\}}{e_i\alpha_{x_i}+a_1\sigma}{x=x_i}\in\hom^1(\kappa(x),\Q/\Z).\]
        Then, by construction, $\Cor_\meX(\beta)=N\Cor(\alpha)+\p{a_1N_1+\dots+a_rN_r}\sigma=\sigma$.
        \qedhere
    \end{itemize}
\end{proof}
\begin{ex}[$\Br(\meY(1)\thickslash\mu_2)$]
    Let $k$ be a perfect field with $\Char k\nmid6$, and let $\meY(1)$ denote the moduli stack of elliptic curves. Consider the rigidification $\meY\coloneq\meY(1)_k\thickslash\mu_2$ (defined as in \cite[Theorem A.1]{AOV}), so $\meY\simeq\sqrt[3]{0/\A^1}\by_{\A^1}\sqrt[2]{1728/\A^1}$. Let $\meX\coloneqq\sqrt[3]{0/\P^1}\by_{\P^1}\sqrt[2]{1728/\P^1}$, a natural compactification of $\meY$. Then, the bottom row of \cref{cd:closed-pt-comp} (using that $\hom^2(\A^1_k,\push\eta\G_m)\simeq\Br k(\A^1)$ when $k$ is perfect, e.g. by \cite[Proof of Theorem 3.6.1]{CT-S}) shows that $\Br\meY\subset\Br k(\A^1)=\Br k(\P^1)$ consists of the Brauer classes for which the residue maps $\res_{x/\meX}$ vanish for all $x\neq\infty\in\abs\meX_1$. With this in mind, \cref{prop:stacky-Faddeev} gives rise to the exact sequence
    \[0\too\Br\meY\too\Br k(\P^1)\xtoo{\res_{\infty/\P^1_k}}\hom^1(k,\Q/\Z)\xtoo6\hom^1(k,\Q/\Z)\too0\]
    from which one deduces that $\Br\meY$ is an extension
    \[0\too\Br k\too\Br\meY\xtoo{\res_{\infty/\P^1_k}}\hom^1(k,\zmod6)\too0.\]
    This extension is furthermore split since $\meY(k)\neq\emptyset$.
\end{ex}
\begin{rem}\label{rem:Faddeev-Y(1)}
    Given the discussion in this section, one can naturally ask about similarly behaving residue maps and Faddeev-type sequences for stacky $\P^1$'s which are \important{not} generically scheme-y. For example, let $\meX=\meX(1)$ and $\meY=\meY(1)$, both over a perfect field $k$ of characteristic not $2$ or $3$. \cref{thma:mod-curve}\bp{1,2} show that $\Br\meX=\Br k$ and that there is an exact sequence
    \[0\too\Br\meX\too\Br\meY\too\hom^1(k,\zmod{12})\too0.\]
    This suggests that, $\meY(1)$ should have a residue map valued in $\hom^1(k,\zmod{12})$. It was explained to me by Tim Santens that such a thing can be provided by his and Loughran's notion of a `residue along a sector'; see \cite[Definition 5.25]{loughran2025mallesconjecturebrauergroups}. This specific application is not spelled out in their paper because it has different goals/focus, but their notion allows one to answer a question asked in an earlier version of this paper.
\end{rem}

\numberwithin{thm}{subsection}
\section{\bf Examples}\label{sect:examples}
\subsection{The Brauer group of $\meX(1)$}\label{sect:X(1)-example}
We compute the Brauer group of $\meX(1)$, the moduli stack of generalized elliptic curves. We first do this over $\Z[1/6]$-schemes since $\meX(1)$ is tame away from characteristics $2$ and $3$. Afterwards, we show that the techniques of this paper can also be applied to compute $\Br\meX(1)_{\Z[1/2]}$ despite the fact that $\meX(1)$ is (very mildly) wild in characteristic $3$.

\subsubsection{\bf over $\Z[1/6]$}
Fix any noetherian scheme $S/\Z[1/6]$, and let $\meX=\meX(1)_S$. Let $f\colon\meX\to S$ and $g\colon\P^1_S\to S$ denote their structure maps. We will compute $\Br\meX$ using the Leray spectral sequence
\begin{equation}\label{eqn:leray-f-X(1)}
    E_2^{ij}=\hom^i(S,\homR^j\push f\G_m)\implies\hom^{i+j}(\meX,\G_m).
\end{equation}
\begin{lemma}\label{lem:push-f-X(1)}
    $\push f\G_m\simeq\G_m$, $\homR^1\push f\G_m\simeq\ul\Z$ and $\homR^2\push f\G_m\simeq\homR^2\push g\G_m$.
\end{lemma}
\begin{proof}
    The fact that $\push f\G_m\simeq\G_m$ follows from the fact that $\meX$'s coarse moduli space is $\P^1_S$ (for this, one can either argue as in \cite[Lemma 4.4]{FO} or use the fact that formation of coarse spaces commutes with arbitrary base change in the tame setting \cite[Coroarlly 3.3(a)]{AOV}). The fact that $\homR^1\push f\G_m\simeq\ul\Z$ is a consequence of \cite[Theorem 1.3]{FO} which shows that $\Pic\meX_T\simeq\hom^0(T,\Z)\by\Pic(T)$ for any $T/S$. 
    
    Finally, $\homR^2\push f\G_m=\homR^2\push g\G_m$ because \cref{ex:X(1)R} showed that, at the level of stalks, the natural map $\homR^2\push g\G_m\to\homR^2\push f\G_m$, is an isomorphism \important{except} this argument did not establish the existence of the exact sequence \cref{ses:R1-X(1)}. To remedy this, set $\meY=\meX\thickslash\mu_2$ and then appeal to \cref{lem:two-step-push-es}. Noting that $\meY\simeq\sqrt[3]{0/\P^1}\by_{\P^1}\sqrt[2]{1728/\P^1}$ (because $6$ is invertible on the base), one can use \cref{lem:root-R1} twice to compute that $\homR^1\push\rho\G_m\simeq\ul{\zmod3}_0\oplus\ul{\zmod2}_{1728}$ ($\rho$ as in \cref{lem:two-step-push-es}) and can use \cref{lem:root-R2} to compute that $\homR^2\push\rho\G_m=0$. Thus, in this case, the exact sequence of \cref{lem:two-step-push-es} exactly produces the exact sequence \cref{ses:R1-X(1)}.
\end{proof}
\begin{cor}\label{cor:BrX(1)-es}
    There is an exact sequence
    \[0\too\hom^2(S,\G_m)\xtoo{\pull f}\ker\p{\hom^2(\meX,\G_m)\to\hom^0(S,\homR^2\push g\G_m)}\too\hom^1(S,\Z)\too0\]
\end{cor}
\begin{proof}
    This comes from the exact sequence of low degree terms in the Leray spectral sequence \cref{eqn:leray-f-X(1)}, using that $\pull f\colon\hom^i(S,\G_m)\to\hom^i(\meX,\G_m)$ is injective for all $i$ (in particular, for $i=2,3$) since $\meX(S)\neq\emptyset$ (the cuspidal point, corresponding to a nodal cubic, is defined over $\Z$ and so over $S$).
\end{proof}
\begin{cor}\label{cor:BrX(1)-Z[1/6]}
    For any noetherian scheme $S/\Z[1/6]$, 
    \[\Br'\meX(1)_S\simeq\Br'S.\]
\end{cor}
\begin{proof}
    By \cite[Lemma A.3]{shin-Gm}, $\hom^1(S,\Z)$ is torsion-free. As a consequence of \cite[Chapter II, Theorem 2]{gabber:thesis}, which states that $\Br'S\iso\Br'\P^1_S$ for any scheme $S$, one also sees that stalks (and so global sections) of $\homR^2\push g\G_m$ are torsion-free. Take torsion in the exact sequence in \cref{cor:BrX(1)-es}.
\end{proof}

\subsubsection{\bf over $\Z[1/2]$}\label{sect:BrX(1)/Z[1/2]}
Now let $S$ be any noetherian $\Z[1/2]$-scheme. Still let $\meX=\meX(1)_S\xto fS$. We will compute $\Br\meX$ by again first computing it in the case that $S$ is strictly local and then leveraging this to compute it for general sense. In the first step (with $S$ strictly local), it will no longer suffice to just apply \cref{cor:R2Gm-vanishes-sequence} because $\meX$ is no longer tame (if $S_{\F_3}\neq\emptyset$).
\begin{ex}
    Consider the elliptic curve $E_0\colon y^2=x^3-x$ over $S_{\F_3}$. One can check (similarly as in \cite[Appendix A]{silverman}) that $\ul\Aut(E_0)\simeq\ul{\zmod 3}\rtimes\mu_4$, where $1\in\zmod3$ acts via $(x,y)\mapsto(x+1,y)$ and $\zeta\in\mu_4$ acts via $(x,y)\mapsto(\zeta^2x,\zeta y)$. In particular, $\ul\Aut(E_0)$ is a finite \'etale group of order divisible by $3$, so $\meX$ is not tame in characteristic $3$.
\end{ex}
On the other hand, \cite[Proposition A.1.2]{silverman} shows that the above example captures essentially all of the non-tameness of $\meX$.\footnote{The cusp of $\meX(1)$ represents a N\'eron 1-gon $C$ (i.e. a nodal cubic viewed as a generalized elliptic curve), and one can check that its automorphism group is $\ul{\zmod2}$, with generator acting by inversion on the smooth locus $\G_m\subset C$.} Let $\pi\colon\meX\to\P^1_S$ be the coarse space of $\meX=\meX(1)_S$, given by the $j$-invariant. \cref{thm:R2Gm-vanishes}, applied to the open tame locus of $\meX$, shows that $\homR^2\push\pi\G_m$ is supported on the closed subscheme $i_0\colon S_{\F_3}\into\P^1_S$ of $\P^1_S$ defined by $j=0=3$. We claim that, in fact, $\homR^2\push\pi\G_m=0$; because it is supported along $i_0$, to prove this it will suffice to show that its stalks over $i_0$ vanish.

\begin{lemma}\label{lem:S3-coh-12mod4}
    Let $M$ be an abelian group with $S_3$-action. Then, the natural map
    \[\ghom^n(\zmod2,M^{\zmod3})\too\ghom^n(S_3,M)\]
    is an isomorphism for all $n\equiv1,2\pmod4$. In particular, $\ghom^n(S_3,M)$ is $2$-torsion for such $n$.
\end{lemma}
\begin{proof}
    This will follow from the Hochschild--Serre spectral sequence 
    \[E_2^{ij}=\ghom^i(\zmod2,\ghom^j(\zmod3,M))\implies\ghom^{i+j}(S_3,M)\] 
    once we know that $E_2^{ij}=0$ if $i,j\ge1$ or if $i=0$ and $j\equiv1,2\pmod4$. When $i,j\ge1$, $E_2^{ij}$ must be both $2$-torsion and $3$-torsion, so it must vanish. Consider $E_2^{0j}=\ghom^j(\zmod3,M)^{\zmod2}$ when $j\equiv1,2\pmod4$. By \cite[Example 6.7.10]{weibel}, $\zmod2$ acts on $\ghom^j(\zmod3,M)$ via multiplication by $-1$ (if $j\equiv1,2\pmod4$). Thus, $E_2^{0j}=\ghom^j(\zmod3,M)[2]=0$.
\end{proof}
\begin{prop}\label{prop:char-3-j0-stalk}
    Let $R$ be a strictly local $\F_3$-algebra, and define $\meY$ via the following pullback square:
    \[\begin{tikzcd}
        \meY\ar[r]\ar[d]&\meY(1)\ar[r, symbol=\subset]\ar[d]&\meX(1)\ar[d, "\pi"]\\
        \spec R\ar[r, "0"]&\A^1\ar[r, symbol=\subset]&\P^1
        .
    \end{tikzcd}\]
    Then, $\hom^2(\meY,\G_m)=0$.
\end{prop}
\begin{proof}
    Note that $\meY$ is `the moduli stack of elliptic curves with constant $j$-invariant $0$, in characteristic $3$.' Following the strategy of \cite{AM}, we compute $\Br\meY$ using a presentation/smooth cover for $\meY$ derived from the Legendre family $\A^1\sm\{0,1\}\to\meY(1)$. That is, we consider the following commutative diagram (all of whose squares are Cartesian):
    \[\begin{tikzcd}
        &Y'\ar[r]\ar[d]&\A_R^1\sm\{0,1\}\ar[d, "y^2=x(x-1)(x-\lambda)"]&Y(2)\ar[l, symbol=\coloneq]\ar[ddd, bend left=55, "2^8\frac{(\lambda^2-\lambda+1)^3}{\lambda^2(\lambda-1)^2}"]\\
        B\ul{\zmod2}_{Y'}\ar[r, equals]&\meY'\ar[r]\ar[d]&\meY(2)_R\ar[d, "S_3\t{-torsor}"]\ar[r, equals]&B\ul{\zmod2}_{Y(2)}\\
        &\meY\ar[r]\ar[d]&\meY(1)_R\ar[d, "\pi"]\\
        &\spec R\ar[r, "0"]&\A^1_R&Y(1)\ar[l, symbol=\coloneq]
        .
    \end{tikzcd}\]
    Above, one can directly compute that 
    \[Y'\simeq\spec R[\lambda]/(\lambda^2-\lambda+1)^3=\spec R[\eps]/(\eps^6)\twhere\eps=\lambda+1.\]
    Furthermore, by \cite[Below Corollary 4.6]{AM}, $S_3\actson Y(2)=\spec R[\lambda][1/\lambda(\lambda-1)]$ via
    \begin{equation}\label{S3-action-lambda}
        \sigma=(132)\colon\lambda\mapsto\frac{\lambda-1}\lambda\tand\tau=(23)\colon\lambda\mapsto\inv\lambda.
    \end{equation}
    With all of this set up, consider the Hochschild--Serre spectral sequence
    \begin{equation}\label{SS:HS-S3-Y}
        E_2^{ij}=\ghom^i(S_3,\hom^j(\meY',\G_m))\implies\hom^{i+j}(\meY,\G_m).
    \end{equation}
    Its $E_2$-page is pictured in \cref{fig:j-0-mod-SS}, whose contents are justified in claims \bp1--\bp3 below. 
    \begin{figure}
        \centering
        \[\begin{tikzcd}
            0\\
            \mu_2(R) & \mu_2(R)\ar[drr, "d_2^{1,1}"]\\
            \G_m(Y')^{S_3} & \mu_2(R) & 0 & \ghom^3(S_3,\G_m(Y'))
        \end{tikzcd}\]
        \caption{The $E_2$-page of the Hochschild--Serre spectral sequence computing the $\G_m$-cohomology of $\meY=[\meY'/S_3]$.}
        \label{fig:j-0-mod-SS}
    \end{figure}
    To begin computing this spectral sequence, we first remark that $\G_m(\meY')=\G_m(Y')$ sits in a short exact sequence
    \begin{equation}\label{ses:Gm(Y')}
        0\too\underbrace{R\eps\oplus R\eps^2\oplus R\eps^3\oplus R\eps^4\oplus R\eps^5}_{R^{\oplus 5}}\too\G_m(Y')\xtoo{\eps=0}\units R\too1,
    \end{equation}
    which becomes $S_3$-equivariant when $S_3\actson\units R$ trivially.\footnote{The particular $S_3$-action on $R^{\oplus5}=R\eps\oplus R\eps^2\oplus R\eps^3\oplus R\eps^4\oplus R\eps^5$ is not directly important for the present argument, but for the sake of completeness, $\sigma=(132)$ and $\tau=(23)$ act on $R^{\oplus5}$ via the matrices
    \[\sigma\colon\Matrix1{}{}{}{}11{}{}{}1{-1}1{}{}1001{}11011\tand\tau\colon\Matrix{-1}{}{}{}{}{-1}1{}{}{}{-1}{-1}{-1}{}{}{-1}001{}{-1}101{-1}.\]}
    \begin{enumerate}
        \item Claim: $\hom^2(\meY',\G_m)=0$ so $E_2^{i2}=0$ for all $i\ge0$.

        Applying \cite[Proposition 3.2]{AM} (or \cref{prop:BG-Gm-coh} since $\mu_2=\ul{\zmod2}$ over $Y'$) to $\meY'=B\ul{\zmod2}_{Y'}$ shows that
        \begin{equation}\label{eqn:H2Y'}
            \hom^2(\meY',\G_m)\simeq\hom^2(Y',\G_m)\oplus\hom^1(Y',\mu_2).
        \end{equation}
        Because $Y'\simeq\spec R[\eps]/(\eps^6)$ is affine, \cite[proof of Proposition 3.2.5]{CT-S} (or \cref{lem:tame-red-Gm}) shows that $\hom^2(Y',\G_m)\simeq\hom^2(R,\G_m)=0$ and that $\hom^1(Y',\G_m)\simeq\hom^1(R,\G_m)=0$. The Kummer sequence then shows that $\hom^1(Y',\mu_2)\simeq\G_m(Y')/2$. Applying the snake lemma to
        \[\homses{R^{\oplus 5}}{}{\G_m(Y')}{\eps=0}{\units R}222{R^{\oplus 5}}{}{\G_m(Y')}{\eps=0}{\units R}\]
        shows that $\G_m(Y')/2=0$ because $R/2=0$ (since $2\in\units R$) and $\units R/2=0$ (since $R$ is strictly henselian). Thus, $\hom^1(Y',\mu_2)=0$ as well, so \cref{eqn:H2Y'} shows that $\hom^2(\meY',\G_m)=0$.

        \item Claim: $\hom^1(\meY',\G_m)\simeq\mu_2(R)$, so $S_3\actson\hom^1(\meY',\G_m)$ trivially. Consequently, $E_2^{01}=\mu_2(R)=E_2^{11}$.

        As before, \cite[Proposition 3.2]{AM} (or \cref{prop:BG-Gm-coh}) shows that $\hom^1(\meY',\G_m)\simeq\hom^1(Y',\G_m)\oplus\hom^0(Y',\mu_2)$. While proving \bp1, we showed that $\hom^1(Y',\G_m)=0$, so $\hom^1(\meY',\G_m)\simeq\hom^0(Y',\mu_2)=\mu_2(R[\eps]/(\eps^6))=\mu_2(R)$. Furthermore, the $S_3$-action is trivial because $\mu_2(R)=\{\pm1\}$ has no non-trivial automorphisms. Finally, $E_2^{01}=\mu_2(R)^{S_3}=\mu_2(R)$ and $E_2^{11}=\ghom^1(S_3,\mu_2(R))=\Hom(S_3,\mu_2(R))=\Hom(\zmod2,\mu_2(R))\simeq\mu_2(R)$.

        \item Claim: $E_2^{10}=\mu_2(R)$ and $E_2^{20}=0$.

        We first claim that $\ghom^1(S_3,\G_m(Y'))\iso\ghom^1(S_3,\units R)$ and that $\ghom^2(S_3,\G_m(Y'))\into\ghom^2(S_3,\units R)$. Taking $S_3$-cohomology of the exact sequence \cref{ses:Gm(Y')} shows that these both follow from knowing that $\ghom^i(S_3,R^{\oplus 5})=0$ for $i=1,2$. \cref{lem:S3-coh-12mod4} implies that $\ghom^i(S_3, R^{\oplus5})$ is $2$-torsion (for $i=1,2$), but $R^{\oplus 5}$ itself is $3$-torsion, so these groups must be killed by both $2$ and $3$; thus, they must vanish. Therefore,
        \[E_2^{10}=\ghom^1(S_3,\G_m(Y'))\iso\ghom^1(S_3,\units R)=\Hom(S_3,\units R)=\Hom(\zmod2,\units R)=\mu_2(R)\]
        and
        \[E_2^{20}=\ghom^2(S_2,\G_m(Y'))\into\ghom^2(S_3,\units R)\overset{\cref{lem:S3-coh-12mod4}}=\ghom^2(\zmod2,\units R)=\units R/2=0.\]
    \end{enumerate}
    By \cref{fig:j-0-mod-SS}, in order to conclude that $\hom^2(\meY,\G_m)=0$, it suffices to show that the differential
    \[d_2^{1,1}=\d_2^{1,1}(R)\colon\mu_2(R)\too\ghom^3(S_3,\G_m(Y'))\]
    is injective (equivalently, nonzero). Of course, it is enough to know that the composition $\mu_2(R)\xto{\d_2^{1,1}}\ghom^3(S_3,\G_m(Y'))\to\ghom^3(S_3,\units R)$ is nonzero. Since $S_3$ acts trivially on $\units R$, one knows (see e.g. \cite[Lemma 5.2]{AM}) that $\ghom^3(S_3,\units R)=\mu_6(R)$. Since $\mu_2(R)$ is $2$-torsion, this composition factors through a map
    \[\phi=\phi_R\colon\mu_2(R)\to\mu_6(R)[2]=\mu_2(R)\subset\ghom^3(S_3,\units R).\]
    Now, the spectral sequence \cref{SS:HS-S3-Y} (and so also the above map $\phi_R$) is functorial in $R$ by construction. With this in mind, consider the sequence of maps $\Q_3\from\Z_3\to\F_3\to R$. Functoriality of $\phi$ gives rise to a commutative diagram
    \[\begin{tikzcd}
        \mu_2(R)\ar[r, "\phi_R"]&\mu_2(R)\\
        \mu_2(\F_3)\ar[r, "\phi_{\F_3}"]\ar[u, "=" sloped]&\mu_2(\F_3)\ar[u, "=" sloped]\\
        \mu_2(\Z_3)\ar[u, "=" sloped]\ar[d, "="' sloped]\ar[r, "\phi_{\Z_3}"]&\mu_2(\Z_3)\ar[u, "=" sloped]\ar[d, "="', sloped]\\
        \mu_2(\Q_3)\ar[r, "\phi_{\Q_3}"]&\mu_2(\Q_3)
    \end{tikzcd}\]
    Thus, it suffices to show that $\phi_{\Q_3}$ is nonzero. Let $K=\Q_3$. We analyze the analogous spectral sequence $\ghom^i(S_2,\hom^j(\meY'_K,\G_m))\implies\hom^{i+j}(\meY_K,\G_m)$ over $K$. Since $\Char K=0$, we know that $\hom^2(\meY_K,\G_m)=0$ (e.g. because it is a stalk of the sheaf $\homR^2\push\pi\G_m$ of \cref{thm:R2Gm-vanishes} applied to the tame stack $\meX(1)_K$). This vanishing implies that the differential 
    \[\d_2^{1,1}(K)\colon\ghom^1(S_3,\Pic\meY'_K)\too\ghom^3(S_3,\G_m(Y'_K))\]
    must be injective. At the same time, because $K$ is a $\Q$-algebra, the exact sequence $0\to K^{\oplus5}\to\G_m(Y'_K)\to\units K\to1$, analogous to \cref{ses:Gm(Y')}, shows that $\ghom^1(S_3,\G_m(Y'_K))\iso\ghom^1(S_3,\units K)=\mu_2(K)$ and $\ghom^3(S_3,\G_m(Y'_K))\iso\ghom^3(S_3,\units K)=\mu_6(K)$. Furthermore, one can argue as in the earlier claim \bp2 to see that $\ghom^1(S_3,\Pic\meY'_K)\simeq\mu_2(K)$. Thus, the composition
    \[\mu_2(K)\simeq\ghom^1(S_3,\Pic\meY'_K)\xtoo{\d_2^{1,1}(K)}\ghom^3(S_3,\G_m(Y'_K))\iso\ghom^3(S_3,\units K)\]
    must be injective, and so $\phi_K$ is injective as well (recall the above composition factors through $\phi_K$). This completes the proof.
\end{proof}
\begin{cor}\label{cor:char3-R2Gm=0}
    $\homR^2\push\pi\G_m=0$.
\end{cor}
\begin{proof}
    As remarked above \cref{lem:S3-coh-12mod4}, \cref{thm:R2Gm-vanishes} shows that $\homR^2\push\pi\G_m$ is supported on the closed subscheme $i_0\colon S_{\F_3}\into\P^1_S$ defined by $j=0=3$. On top of this, \cref{prop:char-3-j0-stalk} shows that all stalks of $\homR^2\push\pi\G_m$ along this subscheme vanish as well, so $\homR^2\push\pi\G_m=0$.
\end{proof}
With this in place, we can now extend \cref{cor:BrX(1)-Z[1/6]} to $\Z[1/2]$-schemes. Recall that $f\colon\meX=\meX(1)_S\to S$ and $g\colon\P^1_S\to S$ are their structure morphisms. As before,
\[\push f\G_m\simeq\G_m\tand\homR^1\push f\G_m\simeq\ul\Z\]
because $\meX$'s coarse space is $\P^1_S$ and as a consequence of \cite[Theorem 1.3]{FO}.
\begin{lemma}
    $\homR^2\push f\G_m\simeq\homR^2\push g\G_m$.
\end{lemma}
\begin{proof}
    It suffices to show that $\hom^2(\meX(1)_R,\G_m)=\hom^2(\P^1_R,\G_m)$ for any noetherian strictly local $\Z[1/2]$-algebra $R$. Let $R$ be such a ring and let $\pi\colon\meX(1)_R\to\P^1_R$ be the coarse space map. Using that $\homR^2\push\pi\G_m=0$ by \cref{cor:char3-R2Gm=0}, the argument of \cref{ex:X(1)R} goes through for this case as well with the following modification. Let $\meY\coloneqq\meX(1)_R\thickslash\mu_2\xtoo\rho\P^1_R$. The exact sequence \cref{ses:R1-X(1)} in \cref{ex:X(1)R} should be replaced with the exact sequence
    \[0\too\homR^1\push\rho\G_m\too\homR^1\push\pi\G_m\too\zmod2\too0\]
    coming from \cref{lem:two-step-push-es}. Above, $\homR^1\push\rho\G_m$ is supported on the closed subscheme $\spec R\xinto0\P^1_R$ (since $\rho$ is an isomorphism away from this subscheme) and so is acyclic. Hence, one still has $\hom^1(\P^1_R,\homR^1\push\pi\G_m)\iso\hom^1(\P^1_R,\zmod2)=0$ as is necessary for the rest of the argument of \cref{ex:X(1)R} to apply here.
\end{proof}
\begin{cor}\label{cor:BrX(1)-general}
    For any noetherian scheme $S/\Z[1/2]$,
    \[\Br'\meX(1)_S\simeq\Br'S.\]
\end{cor}
\begin{proof}
    Using our computations of $\homR^i\push f\G_m$ for $i=0,1,2$, we argue exactly as in \cref{cor:BrX(1)-es,cor:BrX(1)-Z[1/6]}.
\end{proof}

\subsection{The Brauer group of $\meY(1)$}\label{sect:ex:Y(1)}

Let $\meY(1)/\Z$ denote the moduli stack of elliptic curves.
Let $S/\Z[1/2]$ be a regular noetherian scheme. As our next example, we will compute $\Br\meY(1)_S$, partially generalizing results of \cite{AM,meier-cs,dilorenzo2022cohomological}.

\begin{set}
    Set $\meX\coloneqq\meY(1)_S$ and let $f\colon\meX\to S$ be its structure map. Note that, by \cite[Lemma 4.4]{FO}, the coarse moduli space of $\meX$ is given by the $j$-invariant $c\colon\meX\to\A^1_S$. Let $g\colon\A^1_S\to S$ be its structure map, so we have a commutative triangle
    \[\mapover\meX c{\A^1_S}fg{S.}\]
\end{set}
We will compute $\hom^2(\meX,\G_m)$ by leveraging $f$'s Leray spectral sequence
\begin{equation}\label{SS:Leray-Y(1)}
    E_2^{ij}=\hom^i(S,\homR^j\push f\G_m)\implies\hom^{i+j}(\meX,\G_m).
\end{equation}
\begin{lemma}\label{lem:Y(1)-R0}
    $\push f\G_m=\push g\G_m=\G_m$.
\end{lemma}
\begin{proof}
    The first equality $\push f\G_m=\push g\G_m$ holds simply because $\A^1_S$ is $\meX$'s coarse moduli space and this remains true after smooth base change over $S$. The second equality $\push g\G_m=\G_m$ holds because $S$ is reduced and so, at the level of stalks over a geometric point $\bar s\to S$, $\G_m(\msO_{S,\bar s})=\units\msO_{S,\bar s}\iso\units{\msO_{S,\bar s}[T]}=\G_m(\A^1_{\msO_{S,\bar s}})$.
\end{proof}
\begin{lemma}\label{lem:Y(1)-R1}
    $\homR^1\push f\G_m\simeq\ul{\zmod{12}}$ while $\homR^1\push g\G_m=0$. To fix a particular identification, we insist that the isomorphism $\homR^1\push f\G_m\iso\ul{\zmod{12}}$ maps the class of the Hodge bundle on $\meX$ to $1\in\zmod{12}$.
\end{lemma}
\begin{proof}
    The first claim is a consequence of \cite[Theorem 1.1]{FO} while the second is easily verified on stalks; $S$ is normal so $\Pic\A^1_{\msO_{S,\bar s}}=0$ for any geometric point $\bar s\to S$.
\end{proof}
\begin{lemma}\label{lem:R2g=R2f}
    $\homR^2\push g\G_m\iso\homR^2\push f\G_m$
\end{lemma}
\begin{proof}
    We show that the natural map $\homR^2\push g\G_m\to\homR^2\push f\G_m$ is an isomorphism by checking on stalks, so assume $S=\spec R$ is a noetherian, regular strictly local $\Z[1/2]$-scheme. In this case, we need to show that
    \[\hom^2(\A^1_R,\G_m)\xtoo{\pull c}\hom^2(\meX,\G_m)\]
    is an isomorphism. One can now argue as in \cref{ex:X(1)R}, as we briefly explain. \cref{cor:char3-R2Gm=0} implies that $\homR^2\push c\G_m=0$ (because $\meY(1)\openset\meX(1)$). Setting $\meY\coloneqq\meX\thickslash\mu_2\xtoo\rho\A^1_R$, \cref{lem:two-step-push-es} produces an exact sequence
    \[0\too\homR^1\push\rho\G_m\too\homR^1\push c\G_m\too\ul{\zmod2}\too0\]
    of \'etale sheaves on $\A^1_R$. Above, $\homR^1\push\rho\G_m$ is supported on an $R$-finite subscheme of $\A^1_R$ and so is acyclic; hence,
    \[\hom^1(\A^1_R,\homR^1\push c\G_m)\simeq\hom^1(\A^1_R,\zmod2)\simeq\hom^1(\A^1_R,\mu_2)\simeq\G_m(\A^1_R)/2\simeq\units R/2=0.\]
    Now, the Leray spectral sequence for $c$ yields
    \[0\dashtoo\hom^2(\A^1_R,\G_m)\xtoo{\pull c}\hom^2(\meX,\G_m)\too\hom^1(\A^1_R,\homR^1\push c\G_m)=0.\]
    Above, $\pull c$ is injective e.g. as a consequence of \cref{prop:brauer-injects} and the fact that there exists a dense open subscheme $U\openset\A^1_R$ above which $c$ admits a section. Thus, this exact sequence shows that $\pull c$ is an isomorphism, as desired.
\end{proof}
\begin{rem}
    Showing that $\pull c\colon\hom^2(\A^1_R,\G_m)\to\hom^2(\meX,\G_m)$ is injective in \cref{lem:R2g=R2f} is the only place in this argument (i.e. in the proof of \cref{thma:mod-curve}\bp2) that we use that $S$ is regular (say as opposed to only being normal). If $6\in\Gamma(S,\msO_S)^\by$, then one can instead show injectivity of $\pull c$ by using that, with notation as in \cref{lem:R2g=R2f}, $\homR^1\push\rho\G_m\simeq\ul{\zmod3}_0\oplus\ul{\zmod2}_{1728}$ to show that $\Pic\meX$ surjects onto $\hom^0(\A^1_R,\homR^1\push c\G_m)$. Hence, one can show that \cref{thma:mod-curve}\bp2 holds for \important{normal} noetherian $\Z[1/6]$-schemes as well.
\end{rem}
\begin{thm}\label{thm:BrY(1)}
    For any regular noetherian scheme $S/\Z[1/2]$, there is a short exact sequence
    \begin{equation}\label{ses:BrY(1)}
        0\too\hom^2(\A^1_S,\G_m)\xtoo{\pull c}\hom^2(\meY(1)_S,\G_m)\xtoo f\hom^1(S,\zmod{12})\too0.
    \end{equation}
\end{thm}
\begin{proof}
    We compare the Leray spectral sequences for $f\colon\meX=\meY(1)_S\to S$ and $g\colon\A^1_S\to S$. Their exact sequences of low degree terms sit in the following commutative diagram
    \[\begin{tikzcd}
        &&\ker\p{\hom^2(\A^1_S,\G_m)\too\hom^0(S,\homR^2\push g\G_m)}\\
        0\ar[r, dashed]&\hom^2(S,\G_m)\ar[d, equals]\ar[r, "\pull g"]&K_1\ar[u, symbol=\coloneqq]\ar[d, "\pull c"]\ar[r]&0\\
        0\ar[r, dashed]&\hom^2(S,\G_m)\ar[r, "\pull f"]&K_2\ar[d, symbol=\coloneqq]\ar[r]&\hom^1(S,\zmod{12})\ar[r, "0"]&K\\
        &&\ker\p{\hom^2(\meX,\G_m)\too\hom^0(S,\homR^2\push f\G_m)},
    \end{tikzcd}\]
    where $K\coloneqq\ker\p{\hom^3(S,\G_m)\xto{\pull f}\hom^3(\meX,\G_m)}$. Above, note that $\A^1_S(S)\neq\emptyset$ and also $\meX(S)\neq\emptyset$ (e.g. it contains the elliptic curve $y^2=x^3-x$), so $\pull g,\pull f$ are injective and $K=0$. Hence, the above diagram yields
    \begin{equation}\label{eqn:coker K1 K2}
        \coker\p{K_1\xinto{\pull c}K_2}\simeq\hom^1(S,\zmod{12}).
    \end{equation}
    We would like to upgrade this to a computation of the cokernel of $\hom^2(\A^1_S,\G_m)\to\hom^2(\meX,\G_m)$. Note that, because $\meX(S)\neq\emptyset$, there are no nonzero differentials in the Leray spectral sequence \cref{SS:Leray-Y(1)} which map to the $j=0$ row (and the analogous statement holds for $\A^1_S$). Thus, $E_\infty^{0,2}=\ker\p{\d_2^{0,2}\colon\hom^0(S,\homR^2\push f\G_m)\to\hom^2(S,\homR^1\push f\G_m)}$ and the analogous statement holds for $\A^1_S$. Hence, we have the following homomorphism of exact sequences:
    \begin{equation}\label{cd:Y(1)-SS-comp}
        \begin{tikzcd}
            &K_1&&&\hom^2(S,\homR^1\push g\G_m)\\
            0\ar[r]&\hom^2(S,\G_m)\ar[d]\ar[r, "\pull g"]\ar[u, "\sim" sloped]&\hom^2(\A^1_S,\G_m)\ar[d, "\pull c"]\ar[r]&\hom^0(S,\homR^2\push g\G_m)\ar[r]\ar[d]&0\ar[u, equals]\\
            0\ar[r]&K_2\ar[r]&\hom^2(\meX,\G_m)\ar[r]&\hom^0(S,\homR^2\push f\G_m)\ar[r, "\d_2^{0,2}"]&\hom^2(S,\homR^1\push f\G_m).
        \end{tikzcd}
    \end{equation}
    Now, \cref{lem:R2g=R2f} shows that $\homR^2\push g\G_m\iso\homR^2\push f\G_m$, so commutativity shows that $\d_2^{0,2}=0$ and the above is really a homomorphism of \important{short} exact sequences. The snake lemma then yields that $\pull c\colon\hom^2(\A^1_S,\G_m)\to\hom^2(\meX,\G_m)$ is injective with cokernel isomorphic to $\coker\p{K_1\to K_2}$ which is isomorphic to $\hom^1(S,\zmod{12})$ by \cref{eqn:coker K1 K2}, finishing the proof.
\end{proof}
To finish this section, we will show that the exact sequence in \cref{thm:BrY(1)} is actually split. Recall that we have fixed a regular, noetherian $\Z[1/2]$-scheme $S$ and that we set $\meX=\meY(1)_S\xto fS$. Let $\msL$ denote the Hodge bundle on $\meX$. Then, the discriminant (of elliptic curves) gives a trivialization $\Delta\colon\msO_\meX\iso\msL^{12}$ of its twelfth power, and we consider the $\mu_{12}$-torsor $\meT\coloneqq\meX\p{\sqrt[12]{\Delta}}\too\meX$ of 12th roots of this trivialization. We let $\sigma=[\meT]\in\fhom^1(\meX,\mu_{12})$ denote its corresponding cohomology class, and we consider the map
\begin{equation}\label{map:s-Y(1)}
    \mapdesc s{\hom^1(S,\zmod{12})}{\hom^2(\meX,\G_m)}\alpha{\pull f\alpha\cup\sigma.}
\end{equation}
We claim that $s$ is a section of \cref{ses:BrY(1)}. As in \cref{sect:split-BG}, to prove this, we will need to wander into derived categories. Define the objects
\[\msC\coloneqq\p{\tau_{\le2}\fhomR\push f\G_m}[1]\tand\msD\coloneqq\p{\tau_{\le2}\fhomR\push g\G_m}[1]\]
in the (bounded below) derived category of fppf sheaves on $S$. Note that \cref{lem:Y(1)-R0,lem:Y(1)-R1,lem:R2g=R2f} show that these sit in a distinguished triangle (note: $\ul{\zmod{12}}\simeq\fhomR^1\push f\G_m$ below)
\begin{equation}\label{tri:Y(1)}
    \msD\too\msC\too{\ul{\zmod{12}}}\xtoo{+1}.
\end{equation}
Applying $\Hom(\ul{\zmod{12}},-)$ to this triangle produces the exact sequence
\[0\too\Hom_S(\zmod{12},\msD)\too\Hom_S(\zmod{12},\msC)\xtoo\chi\Hom_S(\zmod{12},\zmod{12}).\]
Furthermore, the identifications
\begin{align}
    \Hom_S(\zmod{12},\msC)
    &= \Hom_S(\zmod{12},(\tau_{\le2}\fhomR\push f\G_m)[1]) \nonumber\\
    &\simeq \Ext^1_S(\zmod{12},\tau_{\le2}\fhomR\push f\G_m) \nonumber\\
    &\simeq \Ext^1_S(\zmod{12}, \fhomR\push f\G_m) \nonumber\\
    &\simeq \Ext^1_\meX(\zmod{12}, \G_m) \nonumber\\
    &\simeq \fhom^1(\meX,\mu_{12}) &&\t{by \cref{lem:Ext-G-Gm}}
\end{align}
and the similarly obtained $\Hom_S(\zmod{12},\msD)\simeq\fhom^1(\A^1_S,\mu_{12})$ allow us to rewrite this exact sequence as
\begin{equation}\label{es:Y(1)-mu12}
    0\too\fhom^1(\A^1_S,\mu_{12})\xtoo{\pull c}\fhom^1(\meX,\mu_{12})\xtoo{\chi'}\Hom(\ul{\zmod{12}},\ul{\zmod{12}})\dashtoo0.
\end{equation}
\begin{rem}\label{rem:chi'-comp-Y(1)}
    Recall that the second $\ul{\zmod{12}}$ in \cref{es:Y(1)-mu12} is $\homR^1\push f\G_m$, the relative Picard scheme of $\meX/S$. Identifying the first $\ul{\zmod{12}}$ with $\dual\mu_{12}$ -- by sending $1\in\zmod{12}$ to the natural inclusion $\mu_{12}\into\G_m$ -- one can compute the map
    \[\chi'\colon\fhom^1(\meY(1),\mu_{12})\too\Hom(\zmod{12},\zmod{12})\]
    of \cref{es:Y(1)-mu12} as follows: for $\beta\in\fhom^1(\meY(1),\mu_{12})$, $\chi'(\beta)(1)\in\zmod{12}$ is the image of $\beta$ under the composition
    \[\fhom^1(\meX,\mu_{12})\to\fhom^1(\meX,\G_m)=\Pic\meX\onto\Pic\meX/\Pic\A^1_S\iso\zmod{12},\]
    where, as in \cref{lem:Y(1)-R1}, $\Pic\meX/\Pic\A^1_S\iso\zmod{12}$ sends the Hodge bundle $[\msL]$ to $1$.
\end{rem}
Recalling the earlier defined $\sigma=[\meX(\sqrt[12]{\Delta})\to\meX]\in\fhom^1(\meX,\mu_{12})$, one consequence of \cref{rem:chi'-comp-Y(1)} is that $\chi'(\sigma)=\id_{\zmod{12}}$ so \cref{es:Y(1)-mu12} is short exact.
\begin{prop}\label{prop:Y(1)-splitting}
    The map $s$ of \cref{map:s-Y(1)} gives a right-splitting of \cref{ses:BrY(1)}. In particular, we have an isomorphism
    \[\pull c\oplus s\colon\hom^2(\A^1_S,\G_m)\oplus\hom^1(S,\zmod{12})\iso\hom^2(\meY(1)_S,\G_m)\]
    for any regular noetherian $\Z[1/2]$-scheme $S$.
\end{prop}
\begin{proof}
    Note that the short exact sequence \cref{ses:BrY(1)} is the exact sequence obtained by applying $\fhom^1(S,-)$ to the distinguished triangle \cref{tri:Y(1)}. Write $\rho\colon\msC\to\ul{\zmod{12}}$ for the map appearing in \cref{tri:Y(1)}. Consider the commutative diagram (commutativity follows from \cite[Proposition V.1.20]{milne-et})
    \[\begin{tikzcd}
        \hom^1(S,\zmod{12})\by\Hom_S(\zmod{12},\msC)\ar[d, "\sim" sloped]\ar[r]&\fhom^1(S,\msC)\ar[d, "\sim" sloped]\ar[r, "\push\rho"]&\hom^1(S,\zmod{12})\ar[d, equals]\\
        \hom^1(S,\zmod{12})\by\fhom^1(\meX,\mu_{12})\ar[r, "\pull f(-)\cup(-)"]&\hom^2(\meX,\G_m)\ar[r, "r"]&\hom^1(S,\zmod{12}).
    \end{tikzcd}\]
    Let $\Sigma\colon\zmod{12}\to\msC$ be the map corresponding to $\sigma\in\fhom^1(\meX,\mu_{12})\simeq\Hom_S(\zmod{12},\msC)$. Recalling that $\chi'(\sigma)=\id_{\zmod{12}}$, commutativity shows that, for $\alpha\in\hom^1(S,\zmod{12})$
    \[r(s(\alpha))=r(\pull f\alpha\cup\sigma)=\push\rho(\push\Sigma(\alpha))=\ppush{\rho\circ\sigma}(\alpha)=\push{\chi'(\sigma)}(\alpha)=\id_{\zmod{12},*}(\alpha)=\alpha.\qedhere\]
\end{proof}

\subsection{The Brauer group of $\meY_0(2)$}\label{sect:Y0(2)-example}
Let $\meY_0(2)/\Z[1/2]$ denote the moduli stack of elliptic curves equipped with an (\'etale) subgroup of order $2$. Let $S/\Z[1/2]$ be any regular noetherian $\Z[1/2]$-scheme. We will compute $\Br\meY_0(2)_S$, generalizing the main theorem of \cite{achenjang2024brauer}. 
\begin{set}
    Set $\meX\coloneqq\meY_0(2)_S\xtoo fS$. As shown in \cite[Section 3]{achenjang2024brauer}, $\meX$ is a tame stack with coarse moduli space $c\colon\meX\to\A^1_S\sm\{0\}=\Spec_S\msO_S[s,\inv s]\eqqcolon X$ given by the ``$s$-invariant'' of \cite[Section 3]{achenjang2024brauer}. Let $g\colon X\to S$ be the structure map, so we have a commutative triangle
    \[\mapover\meX cXfg{S.}\]
    Note that we use `$s$' to denote the coordinate on $X=\A^1_S\sm\{0\}$.
\end{set}
\begin{rem}
    It was shown in \cite[Lemma 7.5]{achenjang2024brauer} that $\meX$ is a $\mu_2$-gerbe over $\A^1\sm\{0,-1/4\}$ while points above $-1/4$ has automorphism group $\mu_4$. In particular, $\meX$ is locally Brauerless.
\end{rem}
As is maybe routine by now, we will compute $\hom^2(\meX,\G_m)$ by leveraging the Leray spectral sequence
\begin{equation}
    E_2^{ij}=\hom^i(S,\homR^j\push f\G_m)\implies\hom^{i+j}(\meX,\G_m).
\end{equation}
\begin{lemma}
    $\push f\G_m=\push g\G_m\fiso\G_m\oplus\ul\Z$.
\end{lemma}
\begin{proof}
    The equality $\push f\G_m=\push g\G_m$ holds simply because $X$ is $\meX$'s coarse space. The isomorphism $\push g\G_m\simeq\G_m\oplus\ul\Z$ holds because $S$ is regular and so, at the level of stalks over a geometric point $\bar x\to S$, one has
    \[\G_m(\msO_{S,\bar x})\oplus s^\Z\iso\units{\msO_{S,\bar x}[s,\inv s]}=\G_m(\A^1_{\msO_{S,\bar x}}\sm\{0\}).\qedhere\]
\end{proof}
\begin{lemma}\label{lem:Y0(2)-R1}
    $\homR^1\push f\G_m\simeq\ul{\zmod4}$ (sending the class of the Hodge bundle on $\meX$ to $1\in\zmod4$) while $\homR^1\push g\G_m=0$.
\end{lemma}
\begin{proof}
    The first is a consequence of \cite[Theorem A.3]{achenjang2024brauer} which shows that $\Pic\meY_0(2)_T\simeq\Pic X_T\oplus\zmod4$ (with the second factor generated by the image of the Hodge bundle) for any $T/S$. The second claim can be verified on stalks; since $S$ is regular, $\Pic X_{\msO_{S,\bar x}}=0$ for any geometric point $\bar x\to S$.
\end{proof}
\begin{lemma}\label{lem:Y0(2)-R2}
    There is an exact sequence
    \[0\too\homR^2\push g\G_m\too\homR^2\push f\G_m\too\ul{\zmod2}\too0.\]
\end{lemma}
\begin{proof}
    As usual, we verify this on the level of stalks, so assume $S=\spec R$ is a noetherian, regular strictly local $\Z[1/2]$-scheme. Then, \cref{ex:Y0(2)} already computed that we have an exact sequence
    \[0\too\hom^2(\A^1_R\sm\{0\},\G_m)\too\hom^2(\meY_0(2)_R,\G_m)\too\zmod2\too0,\]
    \important{except} it did not establish the existence of the exact sequence \cref{ses:R1-Y0(2)}. This exact sequence is obtained by applying \cref{lem:two-step-push-es} to $\meX$ with $\meY=\meX\thickslash\mu_2$. With \cref{ses:R1-Y0(2)} established, we complete the argument of \cref{ex:Y0(2)} and so of the present lemma.
\end{proof}
\begin{prop}\label{prop:this-effing-Zmod2}
    The exact sequence of \cref{lem:Y0(2)-R2} is split, so $\homR^2\push f\G_m\simeq\homR^2\push g\G_m\oplus\ul{\zmod2}$.
\end{prop}
\begin{proof}
    Let $\msL$ denote the Hodge bundle on $\meX=\meY_0(2)_S$. Fix any trivialization $D\colon\msO_\meX\to\msL^4$ of the fourth power of $\msL$, and consider the $\mu_4$-torsor $\meX\p{\sqrt[4]D}\to\meX$ with class $\sigma\in\hom^1(\meX,\mu_4)$. Let $[s]\in\hom^1(\meX,\zmod4)$ denote the image of $s\in\G_m(\meX)=\G_m(X)$ under the composition
    \begin{equation}\label{comp:s-defn}
        \G_m(\meX)\to\G_m(\meX)/2\xto{\kappa_2}\hom^1(\meX,\mu_2)=\hom^1(\meX,\zmod2)\to\hom^1(\meX,\zmod4),
    \end{equation}
    where $\kappa_2$ is the Kummer map and the last map is induced from the usual inclusion $\zmod2\into\zmod4$. Define
    \begin{equation}\label{eqn:alpha-class-defn}
        \alpha\coloneqq[s]\cup\sigma\in\im\p{\hom^2(\meX,\mu_4)\to\hom^2(\meX,\G_m)}.
    \end{equation}
    This is the class of a cyclic algebra over $\meX$; see \cite[Section 2]{AM} for more details on relating $\alpha$ to cyclic algebras. Note that $\alpha$ is $2$-torsion because $[s]$ is. We claim that $\ul{\zmod2}\to\homR^2\push f\G_m$ sending $1$ to the class of $\alpha$ (i.e. its image under $\hom^2(\meX,\G_m)\to\hom^0(S,\homR^2\push f\G_m)$) is our desired splitting. To prove this, it suffices to assume that $S=\spec R$ is a regular strictly local $\Z[1/2]$-scheme and then to show that $\alpha\not\in\hom^2(X,\G_m)\subset\hom^2(\meX,\G_m)$, so this is what we do.

    Let $\meY\coloneqq\meY(2)_R$ denote the moduli stack of elliptic curves equipped with a basis for their $2$-torsion. Then, $\meY$ is a trivial $\mu_2$-gerbe over $Y\coloneqq\A^1_R\sm\{0,1\}=\spec R[t,\inv t,1/(t-1)]$, i.e. $\meY\simeq B\mu_{2,Y}$; see \cite[Proposition 4.5]{AM}. Furthermore, the coarse space map $\meY\to Y$ sends the elliptic curve $y^2=(x-e_1)(x-e_2)(x-e_3)$ -- with chosen basis $(e_1,0),(e_2,0)$ for its $2$-torsion -- to the point $t=\frac{e_3-e_1}{e_2-e_1}\in Y$ \cite[Corollary 4.6]{AM}. In addition, the natural map $\pi\colon\meY\to\meX$, remembering the subgroup generated by the first basis element, induces the map
    \[\pi_c=\frac t{(t-1)^2}\colon Y\too X\]
    on coarse spaces \cite[Lemma 3.10]{achenjang2024brauer}. As our last piece of setup, the isomorphism $\meY\simeq B\mu_{2,Y}$ allows us to compute, via \cite[Proposition 3.2]{AM} or \cref{prop:BG-Gm-coh}, that
    \[\hom^2(\meY,\G_m)\simeq\hom^2(Y,\G_m)\oplus\hom^1(Y,\mu_2)\tand\Pic\meY\simeq\Pic Y\oplus\mu_2(Y)\simeq\Pic Y\oplus\zmod2,\]
    with the above $\zmod2$ generated by the class of the Hodge bundle $\pull\pi\msL$. Consider the commutative square
    \[\commsquare{\hom^2(X,\G_m)}{\pull c}{\hom^2(\meX,\G_m)}{\pull\pi_c}{\pull\pi}{\hom^2(Y,\G_m)}{}{\hom^2(\meY,\G_m).}\]
    To finish the proof, it suffices to show that $\pull\pi\alpha\not\in\hom^2(Y,\G_m)\subset\hom^2(\meY,\G_m)$.
    
    To begin, note that $\pull\pi\msL^2\simeq\msO_\meY$ is already trivial and fix some trivialization $D'\colon\msO_\meY\to\msL^2$. Consider the $\mu_2$-torsor $\meY\p{\sqrt[2]{D'}}\to\meY$ with class $\sigma'\in\hom^1(\meY,\mu_2)$. Write $\phi\colon\hom^1(\meY,\mu_2)\to\hom^1(\meY,\mu_4)$ for the map induced by the natural inclusion $\mu_2\into\mu_4$. Then, $\phi(\sigma')$ and $\pull\pi\sigma$ has the same class when projected onto $\Pic(\meY)[4]=\Pic\meY$ (namely, the class of $\pull\pi\msL$), so, letting $\kappa_4$ denote the relevant Kummer map,
    \[\phi(\sigma')=\pull\pi\sigma+\kappa_4(\lambda)\tforsome\lambda\in\G_m(Y).\]
    At the same time, since $\pull\pi s=t/(t-1)^2\in\G_m(\meY)$ differs from $t$ by a square, the construction of $[s]$ via \cref{comp:s-defn} shows that $\pull\pi[s]=[t]$, where $[t]$ is similarly defined as the image of $t$ under the composition $\G_m(\meY)\xto{\kappa_2}\hom^1(\meY,\mu_2)=\hom^1(\meY,\zmod2)\to\hom^1(\meY,\zmod4)$. From this, we conclude that
    \[\pull\pi\alpha=\pull\pi[s]\cup\pull\pi\sigma=[t]\cup(\phi(\sigma')-\kappa_4(\lambda))=[t]\cup\phi(\sigma')-[t]\cup\kappa_4(\lambda)=\kappa_2(t)\cup\sigma'-[t]\cup\kappa_4(\lambda)\in\hom^2(\meY,\G_m).\]
    Since $[t]\cup\kappa_4(\lambda)$ is visibly in the subgroup $\hom^2(Y,\G_m)\subset\hom^2(\meY,\G_m)$, to show that $\pull\pi\alpha$ is \important{not} in this subgroup (and so to conclude the proof), it suffices to show that $\kappa_2(t)\cup\sigma'$ is \important{not} in this subgroup. For this, we first observe that $\sigma'$ is the class of a universal torsor on $\meY=B\mu_{2,Y}$; this follows from \cite[Proof of Corollary 4.6]{AM} where it was shown that one obtains a universal torsor by adjoining the square root of a trivialization of $\pull\pi\msL^2$. Given this, it now follows from \cite[Lemma 3.3]{AM} (or even \cref{prop:BG-Br-split}) that $\kappa_2(t)\cup\sigma'\not\in\hom^2(Y,\G_m)\subset\hom^2(\meY,\G_m)$ since it will have nonzero image under the projection map $\hom^2(\meY,\G_m)\onto\hom^1(Y,\mu_2)$. This completes the proof.
\end{proof}
\begin{rem}\label{rem:this-effing-Zmod2}
    The proof of \cref{prop:this-effing-Zmod2} shows that the summand $\ul{\zmod2}\subset\homR^2\push f\G_m$ is generated by the class
    \[\alpha=[s]\cup\sigma\in\hom^2(\meX,\G_m)\]
    (see \cref{eqn:alpha-class-defn}) which exists (and is nonzero) over every base $S$. We abuse notation by writing $\alpha\colon\hom^0(S,\zmod2)\to\hom^2(\meX,\G_m)$ also for the map sending $1$ to this class (separately over each connected component of $S$).
\end{rem}
\begin{thm}\label{thm:BrY0(2)-gen-1}
    There is an exact sequence
    \begin{equation}\label{ses:BrY0(2)-take-1}
        0\too\hom^2(X,\G_m)\oplus\hom^0(S,\zmod2)\xtoo{\pull c\oplus\alpha}\hom^2(\meX,\G_m)\too\hom^1(S,\zmod4)\too0.
    \end{equation}
\end{thm}
\begin{proof}
    We compare the Leray spectral sequences
    \begin{equation}\label{SS:leray-Y0(2)}
        E_2^{ij}=\hom^i(S,\homR^j\push f\G_m)\implies\hom^{i+j}(S,\G_m)\tand F_2^{ij}=\hom^i(S,\homR^j\push g\G_m)\implies\hom^{i+j}(S,\G_m)
    \end{equation}
    for $f\colon\meX=\meY_0(2)_S\to S$ and $g\colon X=\A^1_S\sm\{0\}\to S$, as in the proof of \cref{thm:BrY(1)}. Noting that $X(S)\neq\emptyset\neq\meX(S)$ (e.g. because of the elliptic curve $y^2=x(x-1)(x+1)$ over $\Z[1/2]$), we see that the low degree terms of these spectral sequences sit in the commutative diagram
    \begin{equation}\label{cd:Bry0(2)-1}
        \begin{tikzcd}
            &&\ker\p{\hom^2(X,\G_m)\to\hom^0(S,\homR^2\push g\G_m)}\\
            0\ar[r]&\hom^2(S,\G_m)\ar[r, "\pull g"]\ar[d, equals]&K_1\ar[u, symbol=\coloneqq]\ar[d, "\pull c"]\ar[r]&0\\
            0\ar[r]&\hom^2(S,\G_m)\ar[r, "\pull f"]&K_2\ar[d, symbol=\coloneqq]\ar[r]&\hom^1(S,\zmod4)\ar[r]&0\\
            &&\ker\p{\hom^2(\meX,\G_m)\to\hom^0(S,\homR^2\push f\G_m)}
        \end{tikzcd}
    \end{equation}
    with exact rows. Note that \cref{cd:Bry0(2)-1} shows that $\coker\p{K_1\into K_2}\simeq\hom^1(S,\zmod4)$. Now, because $X(S),\meX(S)\neq\emptyset$, there are no nonzero differentials in either of the spectral sequences \cref{SS:leray-Y0(2)} which map to the $j=0$ row, so
    \[E_\infty^{0,2}=\ker\p{\hom^0(S,\homR^2\push f\G_m)\to\hom^2(S,\homR^1\push f\G_m)}\tand F_\infty^{0,2}=\ker\p{\hom^0(S,\homR^2\push g\G_m)\to\hom^2(S,\homR^1\push g\G_m)}.\]
    Hence, we obtain the following homomorphism of exact sequences:
    \begin{equation}\label{cd:Bry0(2)-2}
        \begin{tikzcd}
            &&&&\hom^2(S,\homR^1\push g\G_m)\ar[d, equals, "\cref{lem:Y0(2)-R1}"]\\
            0\ar[r]&K_1\ar[d]\ar[r]&\hom^2(X,\G_m)\ar[d, "\pull c"]\ar[r]&\hom^0(S,\homR^2\push g\G_m)\ar[r]\ar[d]&0\\
            0\ar[r]&K_2\ar[r]&\hom^2(\meX,\G_m)\ar[r]&\hom^0(S,\homR^2\push f\G_m)\ar[d, equals, "\cref{prop:this-effing-Zmod2}"]\ar[r, "\d_2^{0,2}"]&\hom^2(S,\homR^1\push f\G_m)\\
            &&&\hom^0(S,\homR^2\push g\G_m)\oplus\hom^0(S,\zmod2)
        \end{tikzcd}
    \end{equation}
    Note that commutativity of \cref{cd:Bry0(2)-2} shows that the differential $\d_2^{0,2}$ vanishes when restricted to the summand $\hom^0(S,\homR^2\push g\G_m)$ of $\hom^0(S,\homR^2\push f\G_m)$. At the same time, \cref{prop:this-effing-Zmod2} (see also \cref{rem:this-effing-Zmod2}) shows that the $\hom^0(S,\zmod2)$ summand must survive to the end of the spectral sequence, so the $\d_2^{0,2}$ differential must vanish when restricted to it as well, so $\d_2^{0,2}=0$. Hence, \cref{cd:Bry0(2)-2} is really a homomorphism of \important{short} exact sequences, so the snake lemma produces a short exact sequence
    \[0\too\coker\p{K_1\to K_2}\too\coker\p{\pull c\colon\hom^2(X,\G_m)\to\hom^2(\meX,\G_m)}\too\hom^0(S,\zmod2)\too0\]
    and shows that $\pull c$ is injective. Furthermore, \cref{rem:this-effing-Zmod2} shows that the composition $\hom^2(\meX,\G_m)\to\hom^2(\meX,\G_m)/\hom^2(X,\G_m)\to\hom^0(S,\zmod2)$ is split, so we have an exact sequence
    \[0\too\hom^2(X,\G_m)\oplus\hom^0(S,\zmod2)\xtoo{\pull c\oplus\alpha}\hom^2(\meX,\G_m)\too\coker\p{K_1\to K_2}\too0.\]
    Finally, \cref{cd:Bry0(2)-1} shows that $\coker\p{K_1\to K_2}\simeq\hom^1(S,\zmod4)$, completing the proof.
\end{proof}
As in the case of $\meY(1)$, the exact sequence of \cref{thm:BrY0(2)-gen-1} is split. Let $\pi\colon\meX=\meY_0(2)_S\to\meY(1)_S\eqqcolon\meY$ denote the natural projection. Let $\msM$ denote the Hodge bundle on $\meY$, and let $\msL\coloneqq\pull\pi\msM$ be the Hodge bundle on $\meX$. Fix trivializations $\Delta\colon\msO_\meY\iso\msM^{12}$ and $D\colon\msO_\meX\iso\msL^4$. Consider the cohomology classes
\[\sigma\coloneqq\sq{\meY\p{\sqrt[12]\Delta}\to\meY}\in\hom^1(\meY,\mu_{12})\tand\tau\coloneqq\sq{\meX\p{\sqrt[4]D}\to\meX}\in\hom^1(\meX,\mu_4).\]
Recall that $f\colon\meX=\meY_0(2)_S\to S$ denotes $\meX$'s structure map and letting $h\colon\meY=\meY(1)_S\to S$ be $\meY$'s, define the maps
\begin{equation}\label{maps:s-t}
    \mapdesc s{\hom^1(S,\zmod{12})}{\hom^2(\meY,\G_m)}\beta{\pull h\beta\cup\sigma}\tand\,\mapdesc t{\hom^1(S,\zmod4)}{\hom^2(\meX,\G_m)}\beta{\pull f\beta\cup\tau.}
\end{equation}
In \cref{prop:Y(1)-splitting}, we showed that $s$ above is a right-splitting of \cref{ses:BrY(1)}.
\begin{prop}\label{prop:Y0(2)-splitting}
    The map $t$ of \cref{maps:s-t} is a right-splitting of \cref{ses:BrY0(2)-take-1}. In particular, we have an isomorphism
    \[\pull c\oplus t\oplus\alpha\colon\hom^2(\A^1_S\sm\{0\},\G_m)\oplus\hom^1(S,\zmod4)\oplus\hom^0(S,\zmod2)\isoo\hom^2(\meY_0(2)_S,\G_m)\]
    for any regular noetherian $\Z[1/2]$-scheme $S$.
\end{prop}
\begin{proof}
    Write $\pi_c\colon X=\A^1_S\sm\{0\}\to\A^1_S$ for the map $\pi\colon\meX\to\meY$ induces on coarse space. Then, comparing \cref{ses:BrY0(2)-take-1} to \cref{ses:BrY(1)} shows that $\pi$ induces the following homomorphism of exact sequences:
    \begin{equation}\label{cd:BrY(1)-Y0(2)}
        \homses{\hom^2(\A^1_S,\G_m)}{}{\hom^2(\meY,\G_m)}{}{\hom^1(S,\zmod{12})}{\pull\pi_c}{\pull\pi}\phi{\hom^2(X,\G_m)}{\pull c}{\hom^2(\meX,\G_m)}{}{\hom^1(S,\zmod4)\oplus\hom^0(S,\zmod2)}
    \end{equation}
    Furthermore, \cref{lem:Y(1)-R1,lem:Y0(2)-R1} show that both the $\zmod{12}$ and the $\zmod4$ appearing in \cref{cd:BrY(1)-Y0(2)} are generated by the class of the Hodge bundle, respectively on $\meY$ and $\meX$. Thus, the map 
    \[\hom^1(S,\zmod{12})\xtoo\phi\hom^1(S,\zmod4)\oplus\hom^0(S,\zmod2)\onto\hom^1(S,\zmod4)\]
    is the map induces by the projection $\zmod{12}\onto\zmod4$ sending $1\mapsto1$. In particular, it is split by the map $\psi\colon\hom^1(S,\zmod4)\to\hom^1(S,\zmod{12})$ induced by the map $\zmod4\to\zmod{12}$ sending $1\mapsto9$. Since the map $s$ of \cref{maps:s-t} is a right splitting of the top row of \cref{cd:BrY(1)-Y0(2)} (by \cref{prop:Y(1)-splitting}), commutativity of \cref{cd:BrY(1)-Y0(2)} shows that the composition
    \begin{equation}\label{map:split-comp}
        \hom^1(S,\zmod4)\xto\psi\hom^1(S,\zmod{12})\xto s\hom^2(\meY,\G_m)\xto{\pull\pi}\hom^2(\meX,\G_m)
    \end{equation}
    is a splitting map to the surjection $\hom^2(\meX,\G_m)\onto\hom^1(S,\zmod4)$ in the bottom row of \cref{cd:BrY(1)-Y0(2)}. Finally, one can check that, by construction, the composition \cref{map:split-comp} is the map $t$ of \cref{maps:s-t}.
\end{proof}

\printbibliography

\end{document}